\documentclass[12pt]{report}

\usepackage{amsmath}
\usepackage{amsfonts}
\usepackage{amssymb}
\usepackage{amsthm}
\usepackage{amsbsy}
\usepackage{amscd}
\usepackage{mathtools}
\usepackage{wasysym}
\usepackage{bm}
\usepackage[all]{xy}
\usepackage{url}
\usepackage{color}
\usepackage[greek,english]{babel}
\usepackage{cmap}
\usepackage{ucs}
\usepackage[utf8x]{inputenc}
\usepackage{graphicx}
\usepackage{diagmac2} 
\usepackage{comment}

\newtheorem{defn}{Definition}[section]
\newtheorem{theo}{Theorem}[section]
\newtheorem{lem}{Lemma}[section]
\newtheorem{prop}{Proposition}[section]
\newtheorem{cor}{Corollary}[prop]

\newcommand{\set}[2]{\{#1 ~|~#2\}}
\newcommand{\Con}[1]{\textbf{Cn}(#1)}
\newcommand{\card}[1]{\textit{card}(#1)}
\newcommand{\PS}{\mathbb{P}_{\mathcal{S}}}
\newcommand{\QS}{\mathbb{Q}_{\mathcal{S}}}
\newcommand{\TS}{\mathbb{T}_{\mathcal{S}}}

\newcommand{\Lan}{\mathcal{L}}
\newcommand{\Var}{\mathcal{V}}
\newcommand{\Mvar}{\textit{\textbf{Mv}}_{\mathcal{L}}}
\newcommand{\Cons}{\mathcal{C}}
\newcommand{\ConS}[1]{\textbf{Cn}_{\mathcal{S}}(#1)}
\newcommand{\Func}{\mathcal{F}}
\newcommand{\Sub}[1]{\mathbf{Sub}(#1)}
\newcommand{\FOstar}{\textbf{FO}_{\Lan}^{\ast}}

\newcommand{\p}[1]{\textit{\textbf{p}}_{#1}}

\newcommand{\one}{\mathbf{1}}
\newcommand{\zero}{\mathbf{0}}
\newcommand{\tb}{\bm{\tau}}
\newcommand{\im}{\bm{\iota}}

\newcommand{\Forms}{\textit{\textbf{Fm}}}
\newcommand{\Mform}{\textit{\textbf{Mf}}_{\mathcal{L}}}
\newcommand{\FormAl}{\mathfrak{F}_{\mathcal{L}}}
\newcommand{\MformAl}{\mathfrak{M}_{\mathcal{L}}}
\newcommand{\alg}[1]{\textsf{\textbf{#1}}}
\newcommand{\mat}[1]{\textsf{\textbf{#1}}}
\newcommand{\Lin}{\textsf{\textbf{Lin}}}
\newcommand{\fA}{\mathfrak{A}}

\newcommand{\booleTwo}{\textsf{\textbf{B}}_{2}}

\newcommand{\lukasThree}{\textsf{\textbf{\L}}_{3}}

\newcommand{\godelThree}{\textsf{\textbf{G}}_{3}}

\newcommand{\theory}{\Sigma_{\mathcal{S}}}
\newcommand{\theoryC}{\Sigma_{\mathcal{S}}^{c}}

\newcommand{\LinS}{\textsf{\textbf{Lin}}_{\mathcal{S}}}
\newcommand{\valA}{V_{\textsf{\textbf{A}}}}
\newcommand{\ThmS}{\bm{T}_{\mathcal{S}}}

\newcommand{\LT}{\textbf{LT}_{\mathcal{S}}}
\newcommand{\LTCl}{\textbf{LT}_{\textsf{Cl}}}
\newcommand{\LTInt}{\textbf{LT}_{\textsf{Int}}}

\newcommand{\Top}{\mathcal{O}}

\newcommand{\Cl}{\textsf{Cl}}
\newcommand{\Int}{\textsf{Int}}
\newcommand{\LC}{\textsf{LC}}
\newcommand{\Pos}{\textsf{P}}

\newcommand{\Su}{\mathbf{S}}

\newcommand{\Pred}{\mathbf{P_{r}}}

\date{}

\title{{\Huge \textbf{Lindenbaum Method}}\\
\vspace{1cm}
Tutorial --- UNILOG 2018 }

\author{\LARGE Alex Citkin and Alexei Muravitsky}

\begin{document}
	\maketitle

\tableofcontents	
\chapter{Preliminaries}

\section{Preliminaries from topology}\label{section:topology}
Classic books~\cite{kelley1975} and~\cite{bourbaki1998} are main references in this section.

A \textit{\textbf{topological space}} (or simply \textit{\textbf{space}}) is a pair $(X,\Top)$ (perhaps both $X$ and $\Top$ with subscripts), where $\Top\subseteq\mathcal{P}(X)$ and such that 
\[
\begin{array}{cl}
1^{\circ} &\emptyset\in\Top~\text{and}~X\in\Top;\\\\
2^{\circ} &\text{for any $U_1,\ldots,U_n\in\Top$, $U_1\cap\ldots\cap U_n\in\Top$};\\\\
3^{\circ} &\text{for any $\lbrace U_i\rbrace_{i\in I}\subseteq\Top$, $\bigcup\lbrace U_i\rbrace_{i\in I}\in\Top$}.
\end{array}
\]

In a topological space $(X,\Top)$, $X$ is called the carrier of the space and $\Top$ is the topology of the space. The elements of $\Top$ are called the \textit{\textbf{open sets}} of the space. The complement of an open set is called \textit{\textbf{closed}}. The topology $\Top$ is called \textit{\textbf{discrete}} if $\Top=\mathcal{P}(X)$. In our applications, we will be assuming that $X\neq\emptyset$.
If $X$ is a finite set, any space $(X,\Top)$ is called \textit{\textbf{finite}.}

It is customary to apply the term \textit{space} (or \textit{topological space}) to $X$, meaning that $X$ is endowed with a family $\Top$ of subsets of $X$ satisfying the conditions $1^{\circ}$--$3^{\circ}$. It must be clear that $X$ can be endowed with different families of open sets. Slightly abusing notation, we write $X=(X,\Top)$.

A family $\lbrace U_i\rbrace_{i\in I}\subseteq\Top$ is called an \textit{\textbf{open cover}} of a space $X=(X,\Top)$ if $\bigcup\lbrace U_i\rbrace_{i\in I}=X$. Any
subfamily $\lbrace U_i\rbrace_{i\in I_0}$ with $I_0\subseteq I$ is called a \textit{\textbf{subcover}} of the cover $\lbrace U_i\rbrace_{i\in I}$ if $\bigcup\lbrace U_i\rbrace_{i\in I_0}=X$. A subcover is called \textit{\textbf{finite}} if $I_0$ is a nonempty finite set.

A topological space is called \textit{\textbf{compact}} if each open cover has a finite subcover.

The following observation is obvious.
\begin{prop}\label{P:finite-discrete=compact}
	Any finite discrete topological space is compact.
\end{prop}

Let $X$ be a nonempty set. Suppose we have a family $\text{B}_{0}:=\lbrace Y_i\rbrace_{i\in I}\subseteq\mathcal{P}(X)$. Assume that we want to define a topology $\Top$ on $X$ such that sets $Y_i$ are open in $X$ endowed with this topology. 

First we form all intersections $Y_{i_1}\cap\ldots\cap Y_{i_n}$ collecting them in a set $\text{B}$. Then, we define all unions $\bigcup_{j\in J}\set{Z_j}{Z_{j}\in\text{B}}$ collecting them in a set $\Top$, to which we add $\emptyset$ and $X$ (if they are not yet there). The sets of the resulting set $\Top$ we announce open and define the space $X:=(X,\Top)$. It is not difficult to check that the conditions $1^{\circ}$--$3^{\circ}$ are satisfied. The family $\text{B}_{0}$ is called a \textit{\textbf{subbase}} of the space $X$, and the set $\text{B}$ is a \textit{\textbf{base} of $X$}. It is obvious that all sets of the subbase $\text{B}_{0}$ are open in $X$. \\

Now we employ this way of introducing topology on a set to define the cartesian product of topological spaces. Assume that we have a collection $\lbrace X_i\rbrace_{i\in I}$ of topological spaces $X_{i}:=(X_{i},\Top_i)$, where $i\in I$. First, we define the cartesian product of the sets $X_i$; that is
\[
\text{P}:=\prod_{i\in I}X_i.
\]
An arbitrary element \textbf{x} of \text{P} will be denoted by $(x_i)_{i\in I}$; that is $\textbf{x}:=(x_i)_{i\in I}$. Given an $i\in I$, a \textit{\textbf{projection}} $\rho_{i}:\text{P}\longrightarrow X_i$ is defined by the mapping $\rho_{i}:\textbf{x}\mapsto x_i$. Now we define:
\[
\text{B}_{0}:=\set{\rho_{i}^{-1}(U)}{U\in\Top_i~\text{and}~i\in I}.
\]

Finally, $\text{B}_{0}$ is employed as a subbase to define a \textit{\textbf{product topology}} (aka \textit{\textbf{Tychonoff topology)}} on \text{P}. Denoting the product topology on \text{P} by $\Top$, we call the topological space $(\text{P},\Top)$ a \textit{\textbf{cartesian product}} of the collection $\lbrace(X_i,\Top_i)\rbrace_{i\in I}$.

From the definition of the product topology, we see that its base \text{B} constituted by the sets of the form
\[
\rho_{i_1}^{-1}(U_{i_1})\cap\ldots\cap\rho_{i_k}^{-1}(U_{i_k}),
\]
where $U_{i_1}\in\Top_{i_1},\ldots,U_{i_k}\in\Top_{i_k}$. This means that the sets of this base can be represented as follows:
\begin{equation}\label{E:representation}
	\prod_{i\in I}Z_i,
\end{equation}
where either $Z_i\in\Top_i$, if $i\in I_0$, for some $I_0\Subset I$, or $Z_i=X_i$, if $i\notin I_0$. If each $X_i$ in the cartesian product \text{P} is a finite discrete space, then in each representation~\eqref{E:representation} either $Z_i\subseteq X_i$, whenever
$i\in I_0$, for some finite subset $I_0$ of $I$, or $Z_i=X_i$. 

In our application of cartesian product (Section~\ref{section:finitary-matrix-consequence}) the following proposition is essential.

\begin{prop}[Tychonoff theorem]\label{P:tychonoff}
	The cartesian product of a collection of compact topological spaces is compact relative to the product topology.	
\end{prop}

\begin{cor}\label{C:product-finite-discrete-spaces}
	The cartesian product of a collection of discrete spaces is compact relative to the product topology.	
\end{cor}

\section{Preliminaries from algebra}

\subsection{Distributive lattices and Boolean algebras}\label{section:boolean-algebra}
Boolean algebras belong to to the class of distributive lattices.

An algebra $\alg{L}=\langle\textsf{L};\wedge,\vee\rangle$ of type $\langle\wedge,\vee\rangle$, where $\wedge$ (meet) and $\vee$ (join) are binary operations, is called a \textit{\textbf{distributive lattice}} if the identities following equalities hold in $\alg{L}$ for arbirary elements $x$, $y$ and $z$ of \textsf{L}:
\[
\begin{array}{cll}
(\text{l}_1) &i)~~~x\wedge y=y\wedge x, &ii)~~~x\vee y=y\vee x,\\
(\text{l}_2) &i)~~~x\wedge(y\wedge z)=(x\wedge y)\wedge z,
&ii)~~~x\vee(y\vee z)=(x\vee y)\vee z,\\
(\text{l}_3) &i)~~~(x\wedge y)\vee y=y, &ii)~~~x\wedge(x\vee y)=x,\\
(\text{l}_4) &i)~~~x\wedge(y\vee z) =(x\wedge y)\vee(x\wedge z),
&ii)~~~x\vee(y\wedge z) =(x\vee y)\wedge(x\vee z).\\
\end{array}
\]

In any distributive lattice the following identities hold:
\[
x\wedge x=x~~\text{and}~~x\vee x= x \tag{\textit{idempotent laws}}
\]

Indeed, for any $x,y\in\textsf{L}$,
\[
\begin{array}{rl}
x\!\!\! &= x\wedge(x\vee y)\\&=(x\wedge x)\vee(x\wedge y)\quad[\mbox{according to ($\text{l}_{4}$--$i$)}]\\
&= ((x\wedge y)\vee x)\wedge((x\wedge y)\vee x)\quad[\mbox{in virtue of ($\text{l}_{1}$--$ii$) and ($\text{l}_{4}$--$ii$)}]\\
&= x\wedge x.\quad[\mbox{according to ($\text{l}_{3}$--$i$)}]
\end{array}
\]

The second idempotent law can be proven in a similar fashion. \\

We observe that if $\alg{L}$ is a distributive lattice, then for any $x,y\in\textsf{L}$,
\[
x\wedge y=x~\Longleftrightarrow~x\vee y=y.
\]

Indeed, assume that $x\wedge y=x$. Then we have: $x\vee y=(x\wedge y)\vee y=y$.

On the other hand, if $x\vee y=y$, then $x\wedge y=x\wedge(x\vee y)=x$.

The last equivalence induces the following definition. Given a distributive lattice $\alg{L}$, we define a binary relation on $\textsf{L}$ as follows:
\begin{equation}\label{E:ordering-in-lattice}
	x\le y\stackrel{\text{df}}{\Longleftrightarrow} x\wedge y=x.
\end{equation}

According to the last equivalence, we have:
\[
x\le y~\Longleftrightarrow~x\vee y=y.
\]

We also notice that $\le$ is a partial ordering on \textsf{L} and, according to this ordering, $x\wedge y$ is the greatest lower bound and $x\vee y$ is the least upper bound of $\lbrace x,y\rbrace$, respectively. This, in particular, implies that the operations $\wedge$ and $\vee$ are monotone w.r.t. $\le$. Namely,
\begin{center}
	$x\le y$ implies $x\wedge z\le y\wedge z$ and $x\vee z\le y\vee z$.
\end{center}

A \textit{\textbf{Boolean algebra}} is an algebra $\alg{B}=\langle\textsf{B};\wedge,\vee,\neg,\one\rangle$ of type $\langle\wedge,\vee,\neg,\one\rangle$, where $\wedge$ (meet) and $\vee$ (join) are binary operations, $\neg$ (complementation) is a unary operation and $\one$
(unit) is a 0-ary operation, if the equalities $(\text{l}_1)$--$(\text{l}_4)$ above and $(\text{b}_1)$--$(\text{b}_2)$ below are satisfied in \alg{B} for arbitrary elements $x$, $y$ and $z$ of \textsf{B}:
\[
\begin{array}{cll}
(\text{b}_1) &i)~~~ x\wedge\one=x, &ii)~~~x\vee\one=\one,\\
(\text{b}_2) &i)~~~(x\wedge\neg x)\vee y=y,
&ii)~~~(x\vee\neg x)\wedge y=y.\\
\end{array}
\]

If we apply the definition~\eqref{E:ordering-in-lattice} for $\alg{B}$, we realize that $\one$ is the greatest element and $\neg\one$ the least element in $\alg{B}$, respectively. Indeed, the former comes from 
($\text{b}_{1}$--$i$) and the latter, with the help of ($\text{b}_{2}$--$i$)
and ($\text{b}_{1}$--$i$), can be obtained as follows:
\[
x=(\one\wedge\neg\one)\vee x= \neg\one\vee x.
\]

Denoting
\begin{equation}\label{E:zero-definition}
	\zero:=\neg\one,
\end{equation}
we observe that for any $x\in\textsf{B}$, $\neg x$ is a unique element $y$ such that $x\wedge y=\zero$ and $x\vee y=\one$. The element $\neg x$ is called a \textit{complement} of $x$. Thus any Boolean algebra is a distributive lattice with a greatest element $\one$, a least element $\zero$ and a \textit{complementation} $\neg x$. This implies that, the view of commutativity of $\wedge$ and $\vee$ (the equalities $(\text{l}_1)$), we can say the $x$ is the complement of $\neg x$, which immediately implies the equality
\begin{equation}\label{E:double-negation-equality}
	\neg\neg x=x.
\end{equation}

We also observe that
\begin{equation}\label{E:one-zero-in-boolean}
	\one=x\vee\neg x~~\text{and}~~\zero=x\wedge\neg x,
\end{equation}
for an arbitrary element $x\in\textsf{B}$.

It is obvious that a Boolean algebra is non-degenerate if, and only if, $\zero\neq\one$. This implies that the simplest non-degenerate Boolean algebras consists of two elements --- $\one$ and $\zero$.
Usually, this algebra is depicted by the following diagram.

\begin{figure}[!ht]
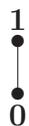
	
	\[
	\ctdiagram{
		\ctnohead
		\ctinnermid
		\ctel 0,0,0,20:{}
		\ctv 0,0:{\bullet}
		\ctv 0,20:{\bullet}
		\ctv 0,27:{\mathbf{1}}
		\ctv 0,-9:{\mathbf{0}}
	}
	\]\label{figure:two-element-boolean}
	\caption{A 2-element Boolean algebra}
	
\end{figure}

\pagebreak
\begin{prop}\label{P:finitely-generated-boolean}
	Let a Boolean algebra $\alg{B}$ be generated by a nonempty finite set {\em$\textsf{B}_0$}. Then any element {\em$x\in\textsf{B}$} can be represented in the form:
	\[
	x=\bigvee_{i=1}^{m}\bigwedge_{j=1}^{n_i}b_{ij},\tag{\textit{disjunctive normal form}}
	\]
	or in the form:
	\[
	x=\bigwedge_{i=1}^{m}\bigvee_{j=1}^{n_i}b_{ij},\tag{\textit{conjunctive normal form}}
	\]
	where for any $i$ and $j$ either {\em$b_{ij}\in\textsf{B}_0$}
	or {\em$\neg b_{ij}\in\textsf{B}_0$}. 
\end{prop}
\begin{proof}
	To prove this proposition, it suffices to notice that all elements of $\textsf{B}_0$ can be represented in both forms. Then, one shows that all elements represented in the disjunctive normal form constitute a subalgebra of $\alg{B}$; and the same is true for the elements represented in the conjunctive normal form; see~\cite{rasiowa-sikorski1970}, theorem II.2.1, for detail.
\end{proof}

As we will see below (Proposition~\ref{P:boolean-algebra-as-heyting}), the class of Boolean algebras can be regarded as a subclass of a larger class of Heyting algebras which are discussed in the next subsection. Therefore, many properties related to Heyting algebras are applicable to Boolean algebras.
For example, the properties that are discussed below in Proposition~\ref{P:Int-properties} for Heyting algebras (Section~\ref{section:heyting-algebra}) can be applied to Boolean algebras as well.

\subsection{Heyting algebras}\label{section:heyting-algebra}
A \textit{\textbf{Heyting algebra}} is an algebra $\alg{H}=\langle\textsf{H};\wedge,\vee,\rightarrow,\neg,\one\rangle$ of type $\langle\wedge,\vee,\rightarrow,\neg,\one\rangle$, where $\wedge$ (meet) and $\vee$ (join), $\rightarrow$ (relative pseudo-complementation) are binary operations, $\neg$ (pseudo-complementation) is a unary operation and $\one$
(unit) is a 0-ary operation, if, besides the equalities ($\text{l}_{1}$)--($\text{l}_{2}$) and ($\text{b}_{1}$) (Section~\ref{section:boolean-algebra}), the following equalities are satisfied for arbitrary elements $x$, $y$, $z$ of \textsf{H}:
\[
\begin{array}{cc}
(\text{h}_1) &~~~x\wedge(x\rightarrow y)=x\wedge y,\\
(\text{h}_2) &~~~(x\rightarrow y)\wedge y=y,\\
(\text{h}_3) &~~~(x\rightarrow y)\wedge(x\rightarrow z)=x\rightarrow(y\wedge z),\\
(\text{h}_4) &~~~x\wedge(y\rightarrow y)=x,\\
(\text{h}_5) &~~~\neg\one\vee y=y,\\
(\text{h}_6) &~~~\neg x=x\rightarrow\neg\one.\\
\end{array}
\]

It is customary to denote:
\[
x\leftrightarrow y:=(x\rightarrow y)\wedge(y\rightarrow x).
\]

As in the case of Boolean algebras, in Heyting algebras, in virtue of
$(\text{b}_1)$, the element $\one$ is a greatest element in Heyting algebra.
The property $(\text{h}_4)$ implies that
\begin{equation}\label{E:x-implies-x}
	x\rightarrow x=\one.
\end{equation}

Using the notation \eqref{E:zero-definition} and the property $(\text{h}_5)$, we conclude that $\zero$ is a least element in Heyting algebra. According to $(\text{h}_6)$, we have:
\[
\neg x=x\rightarrow\zero.
\]

Thus, the last identity and $(\text{h}_1)$ imply the following:
\begin{equation}\label{E:zero-in-heyting}
	x\wedge\neg x=\zero.
\end{equation}

The following property characterizes pseudo-complementation:
\begin{equation}\label{E:pseudo-complementation}
	x\le y\rightarrow z~\Longleftrightarrow~x\wedge y\le z.
\end{equation}

Indeed, assume first that $x\le y\rightarrow z$. Then, in view of ($\text{l}_1$--$i$), the monotonicity of $\wedge$ w.r.t. $\le$ and $(\text{h}_1)$, we have:
\[
x\wedge y\le y\wedge(y\rightarrow z)=y\wedge z\le z.
\]

Conversely, suppose that $x\wedge y\le z$, that is $x\wedge y= x\wedge y\wedge z$. Then, in virtue of $(\text{h}_2)$, $(\text{h}_4)$, $(\text{h}_3)$,
we obtain:
\[
\begin{array}{rl}
x\le y\rightarrow x\!\!\!&=(y\rightarrow x)\wedge(y\rightarrow y)=y\rightarrow(x\wedge y)=y\rightarrow(x\wedge y\wedge z)\\
&=(y\rightarrow(x\wedge y))\wedge(y\rightarrow z)\le y\rightarrow z.
\end{array}
\]

Using~\eqref{E:pseudo-complementation}, we receive  immediately:
\begin{equation}\label{E:less-than=implication}
	x\le y~\Longleftrightarrow~x\rightarrow y=\one;
\end{equation}
which in turn implies:
\[
x=y~\Longleftrightarrow~x\leftrightarrow y=\one.
\]
\begin{prop}\label{P:boolean-algebra-as-heyting}
	Let {\em$\alg{H}=\langle\textsf{H};\wedge,\vee,\rightarrow,\neg,\one\rangle$} be a Heyting algebra. Then its restriction {\em$\alg{B}=\langle\textsf{H};\wedge,\vee,\neg,\one\rangle$} is a Boolean algebra if the identity $x\rightarrow y=\neg x\vee y$ holds in {\em\alg{H}}.
	Conversely, if {\em$\alg{B}=\langle\textsf{B};\wedge,\vee,\neg,\one\rangle$} is a Boolean algebra. Then its expansion {\em$\alg{H}=\langle\textsf{B};\wedge,\vee,\rightarrow,\neg,\one\rangle$},
	where $x\rightarrow y:=\neg x\vee y$, is a Heyting algebra.
\end{prop}
\begin{proof}
	Let \alg{H} be a Heyting algebra in which $x\rightarrow y=\neg x\vee y$.
	We observe that, because of~\eqref{E:zero-in-heyting}, the identity ($\text{b}_2$--$i$) is true in \alg{H}. The identity ($\text{b}_2$--$ii$) is also true, for we have:
	\[
	(x\vee\neg x)\wedge y=(x\rightarrow x)\wedge y=y.
	\]
	
	Now assume that \alg{H} is an expansion of a Boolean algebra \alg{B}, where, by definition, $x\rightarrow y=\neg x\vee y$. We aim to show that all the properties $(\text{h}_1)$--$(\text{h}_6)$ are valid in \alg{H}. Indeed, we obtain:
	\[
	\begin{array}{l}
	x\wedge(x\rightarrow y)=x\wedge(\neg x\vee y)=(x\wedge\neg x)\vee(x\wedge y)=x\wedge y;\\
	(x\rightarrow y)\wedge y=(\neg x\vee y)\wedge y=y;\\
	(x\rightarrow y)\wedge(x\rightarrow z)=(\neg x\vee y)\wedge(\neg x\vee z)=
	\neg x\vee(y\wedge\neg x)\vee(\neg x\wedge z)\vee(y\wedge z)\\
	\qquad\qquad\qquad\qquad~=\neg x\vee(y\wedge z)=x\rightarrow(y\wedge z);\\
	x\wedge(y\rightarrow y)=x\wedge(\neg y\vee y)=x;\\
	\neg x=\neg x\vee\zero=\neg x\vee\neg\one=x\rightarrow\neg\one.
	\end{array}
	\]
\end{proof}

\begin{cor}\label{C:Cl-property}
	Let {\em$\alg{B}=\langle\textsf{H};\wedge,\vee,\neg,\one\rangle$} be a Boolean algebra. Then for an arbitrary {\em$x\in\textsf{B}$}, $\neg\neg x\rightarrow x=\one$, where the operation $x\rightarrow y$ is understood as $\neg x\vee y$.
\end{cor}
\begin{proof}
	Augmenting \alg{B} with $\rightarrow$ defined as above, we receive a Heyting algebra. Then we use~\eqref{E:x-implies-x} and~\eqref{E:double-negation-equality}.
\end{proof}

It should be clear that the simplest non-degenerate Boolean algebra (Fig.~\ref{figure:two-element-boolean}) is also the simplest non-generate Heyting algebra.\\

We will be using Heyting algebras as a basic semantics for the intuitionistic propositional logic (Section~\ref{section:some-lindenbaum-algebras}). For this purpose, in the sequel, we will need the following propertied of Heyting algebras.
\begin{prop}\label{P:Int-properties}
	Let {\em$\alg{H}=\langle\textsf{H};\wedge,\vee,\rightarrow,\neg,\one\rangle$} be a Heyting algebra. For arbitrary elements $x$, $y$ and $z$ of {\em\textsf{H}} the following properties hold:
	{\em\[
		\begin{array}{cl}
		(\text{a}) &x\le y\rightarrow x,\\
		(\text{b}) &x\rightarrow y\le(x\rightarrow (y\rightarrow z))
		\rightarrow(x\rightarrow z),\\
		(\text{c}) &x\le y\rightarrow(x\wedge y),\\
		(\text{d}) &x\wedge y\le x,\\
		(\text{e}) &x\le x\vee y,\\
		(\text{f}) &x\rightarrow z\le(y\rightarrow z)\rightarrow((x\vee y)\rightarrow z),\\
		(\text{g}) &x\rightarrow y\le(x\rightarrow\neg y)\rightarrow\neg x,\\
		(\text{h}) &x\le\neg x\rightarrow y.
		\end{array}
		\]}
\end{prop}
\begin{proof}
	We obtain successively the following.
	
	(a) follows straightforwardly from $(\text{h}_2)$.
	
	To prove (b), using~\eqref{E:pseudo-complementation}, we have:
	\[
	\begin{array}{rl}
	(x\rightarrow y)\wedge(x\rightarrow (y\rightarrow z))\wedge x\!\!\!
	&=(x\rightarrow y)\wedge x\wedge(y\rightarrow z)\\
	&=x\wedge y\wedge (y\rightarrow z)=x\wedge y\wedge z\le z.
	\end{array}
	\]
	Then, we apply \eqref{E:pseudo-complementation} again twice to get (b).
	
	We derive (c) immediately from $x\wedge y\le x\wedge y$.
	
	(d) and (e) are obvious.
	
	To prove (f), we first have:
	\[
	\begin{array}{rl}
	(x\rightarrow z)\wedge(y\rightarrow z)\wedge(x\vee y)\!\!\!
	&=((x\rightarrow z)\wedge(y\rightarrow z)\wedge x)\vee
	((x\rightarrow z)\wedge(y\rightarrow z)\wedge y)\\
	&=(x\wedge z\wedge(z\rightarrow y))\vee(y\wedge z\wedge(x\rightarrow z))\\
	&=(x\wedge z)\vee(y\wedge z)=(x\vee y)\wedge z\le z.
	\end{array}
	\]
	Then, we apply~\eqref{E:pseudo-complementation} twice.
	
	To prove (g), applying~\eqref{E:pseudo-complementation} twice, we have:
	\[
	(x\rightarrow y)\wedge (x\rightarrow \neg y)\wedge x=x\wedge y\wedge\neg y=x\wedge\zero\le\zero,
	\]
	which implies that
	\[
	(x\rightarrow y)\wedge (x\rightarrow \neg y)\le x\rightarrow\zero=\neg x.
	\]
	It remains to apply~\eqref{E:pseudo-complementation}.
	
	Finally, we obtain (h) as follows:
	\[
	x\wedge\neg x=\zero\le y.
	\]
	Then, we apply~\eqref{E:pseudo-complementation}.
\end{proof}

\chapter[Formal Languages]{Formal Languages and Their Interpretation}\label{chapter:languages}	
\section{Formal languages}\label{section:languages}
It is commonly accepted that formal logic is an enterprise, in the broad sense of the word, which is characterized, first of all, by the use, study and development  of formal languages. One of the grounds for this characterization is that formal language ensures the codification of accepted patterns of reasoning. From the outset, we have to admit that such a representation of the train of thought in argumentation confines it to a framework of the discrete rather than the continuous. From this viewpoint, an explicit argument is represented in a formal language by a sequence of discrete units.

It is also commonly accepted that logic, formal or otherwise, deals with the concept of truth. Each unit in a formally presented reasoning is assumed to bear truth-related content. Yet, the use of formal language for the codification of the units of truth-related contents does not make logic formal;
only the realization that the discrete units, called \textit{words}, of a formal language, constituting a deduction chain, are codes of \textit{forms of judgments} rather than ``formulations of factual, contingent knowledge'' (Carnap) does.\footnote{Compare with the notion of elementary propositional function in~\cite{whitehead-russell}, section A $\ast1$.} 

Thus we arrive at a twofold view on the units of formal language, which are used to form the codes of formally presented deductions. On the one hand, they are words of a given formal language, satisfying  some standard requirements; on the other, they are forms that do not possess any factual content but are intended to represent it. Moreover, the content that is meant to be represented by means of a formal unit is intended to generate an assertion, not command, exclamation, question and the like. Having this in mind, we will be calling these formal units  (\textit{sentential}) \textit{formulas} or \textit{terms}, depending on which role they play in our analysis.

Concerning formal deduction, or formal reasoning, we will make a distinction, and emphasize it, between
what is known in philosophy of logic as the \textit{theory of consequence}, or \textit{dialectic argument}, which is rooted in Aristotle's \textit{Analytica Priora}, where the central question is `What follows from what?', and what philosophy of logic calls the \textit{theory of demonstration}, which can be traced to \textit{Analytica Posteriora}, where the main question is `What can be obtained from known (accepted) premises?'. The formulas qualifying to be answers to the latter question will be called \textit{theses} or (formal) \textit{theorems}. We note that the class of such theses depends on the set of \textit{premises}, as well as on the means of the deduction in use. The theory of consequence will be analyzed in terms of a binary relation between the sets of formulas of a given formal language and the formulas of this language. As to ``the means of the deduction in use,'' it will be our main focus throughout the book.\\

After these preparatory remarks, we turn to definition of a formal language. Here we face an obstacle. Since each formal language depends on a particular set of symbols, in order to reach some generality, we are to deal with a schematic language which would contain all main characteristics of many particular languages that are in use nowadays; many other languages of this category can be defined in the future. All our formal languages in focus are \textit{sentential}, that is, they are intended to serve for defining \textit{sentential logical systems}.  However, for technical purposes, we will occasionally be employing predicate languages, which will be introduced as needed.

The schematic language we define in this section will not be the only one we are going to deal with. In order to advance Lindenbaum method, as we understand it, to the limits we are able to envisage, we will discuss it in the framework of some fragments of first order language as well.

A \textit{\textbf{sentential schematic language}} $\Lan$ is assumed to contain symbols of the following pairwise disjoint categories:
\begin{itemize}
	\item a nonempty set $\Var$ of \textit{sentential} (or \textit{propositional}) \textit{variables}; we will refer to the elements of $\Var$ as $p,q, r,\ldots$, using subscripts if needed; 
	\item a set $\Cons$ (maybe, empty) of  \textit{sentential} (or \textit{propositional}) \textit{constants}; we will refer to the elements of $\Cons$ as $a,b,c,\ldots$, using subscripts if necessary; we assume that the cardinality of the set $\Cons$ is not bounded from above, that is $0\le\card{\Cons}$;
	\item a set $\Func$  of  \textit{sentential} (or \textit{propositional}) \textit{connectives} is nonempty; we will refer to the elements of $\Func$ as $F_0,F_1,\ldots,F_{\gamma},\ldots$, where $\gamma$ is any ordinal; for each connective $F_i$, there is a positive natural number $\#(F_i)$ called the \textit{arity of the connective} $F_i$. We will be omitting index and write simply $F$, if confusion is unlikely; the cardinality of $\Func$ is not bounded from above;
	that is $0<\card{\Func}$.
\end{itemize}

We repeat: the sets $\Var$, $\Cons$, and $\Func$, each consisting of symbols, are pairwise disjoint, and assume also that the elements of each of these sets are distinguishable in such a way that a concatenation of a number of these symbols can be read uniquely; that is, given a word $w$ of $\Lan$ of the length 1, one can decide whether $w\in\Var$ or $w\in\Cons$ or $w\in\Func$; if the word $w$ has the length greater than 1, then, we assume, it is possible to recognize each symbol of $w$ in a unique way.

Thus $\Lan$ gives space for defining many specific languages.
To denote a language defined by a specification within  $\Lan$ we will be using subscripts or other marks attached to $\Lan$. 

The most important concept within a language is that of \textit{formula}. To refer to formulas formed by the means of $\Lan$, we use the letters $\alpha,\beta,\gamma,\ldots$ (possibly with subscripts or other marks), calling them $\Lan$-\textit{formulas}. Specifically, these symbols play a role of \textit{informal metavariables} for $\Lan$-formulas. For specified languages we reserve the right to employ other notations, about which the reader will be advised. For a special purpose (to deal with structural inference rules), we introduce \textit{formal metavariables} as well.

Given two sentential languages, $\Lan^{\prime}$ and $\Lan^{\prime\prime}$ with $\Var^{\prime}\subseteq\Var^{\prime\prime}$, $\Cons^{\prime}\subseteq\Cons^{\prime\prime}$ and $\Func^{\prime}\subseteq\Func^{\prime\prime}$, $\Lan^{\prime\prime}$ is called an \textit{\textbf{extension}} of $\Lan^{\prime}$ and the latter is a \textit{\textbf{restriction}} of the former. An extension $\Lan^{\prime\prime}$ of a language $\Lan^{\prime}$ is called \textit{\textbf{primitive}} if $\Cons^{\prime}=\Cons^{\prime\prime}$ and $\Func^{\prime}=\Func^{\prime\prime}$, that is, when $\Lan^{\prime\prime}$ is obtained only by adding new variables, if any, to $\Var^{\prime}$.

\begin{defn}[$\Lan$-formulas]\label{D:L-formulas}
	The set $\textbf{Fm}_{\mathcal{L}}$ of $\Lan$-\textbf{formulas} is formed inductively according to the following rules$:$
	\[
	\begin{array}{cl}
	(\emph{a}) &\Var\cup\Cons\subseteq\bm{Fm}_{\mathcal{L}};~\text{if $\alpha\in\Var\cup\Cons$, $\alpha$ is called \textbf{atomic}};\\ 
	(\emph{b}) & \text{if $\alpha_1,\ldots,\alpha_n\in\bm{Fm}_{\mathcal{L}}$ and $F\in\Func$ with $\#(F)=n$, then $F\alpha_1\ldots\alpha_n\in\bm{Fm}_{\mathcal{L}}$};\\
	(\emph{c}) & \text{the $\Lan$-formalas are only those $\Lan$-words which can be obtained}\\
	&\text{successively according to the rules {\em(a)} and {\em(b)}.}\\
	\end{array}
	\]
\end{defn}

We will find useful the notion of the \textit{\textbf{degree of}} a (given) \textit{\textbf{formula}} $\alpha$, according to the following clauses:
\[
\begin{array}{cl}
(\text{a}) &\text{if $\alpha\in\Var\cup\Cons$, then its degree is $0$};\\
(\text{b}) &\text{if $\alpha=F\alpha_1\ldots\alpha_n$, then the degree of $\alpha$ is the sum of}\\ 
&\text{the degrees of all $\alpha_i$ plus $1$}.
\end{array}
\]
It is easy to see that the degree of a formula $\alpha$ equals the number of occurrences of the sentential connectives that are contained in $\alpha$.

As can be seen from the last definition, the degree of a formula is an estimate of its ``depth.'' Also, in the last definition, it is assumed that if $\alpha=F_{i}\alpha_1\ldots\alpha_n$, that is if $\alpha\notin\Var\cup\Cons$, it begins with some sentential connective ($F$ with $\#(F)=n$ in our case). But then, the words $\alpha_1,\ldots,\alpha_n$ must be recognizable as formulas. Thus it is important to note that for any given word $w$ of the language $\Lan$, one can decide whether $w$ is an $\Lan$-formula or not. One of the ways, by which it can be done, is the construction of a formula tree for the formula in question.\\

Technically, it is convenient to have the notion of metaformula associated with a given language $\Lan$.
\begin{defn}[$\Lan$-metaformulas]
	Given a language $\Lan$, the set $\Mform$ of $\Lan$-metaformulas is defined inductively by the clauses$:$
	\[
	\begin{array}{cl}
	(\emph{a}) &\Mvar\cup\Cons\subseteq\Mform;~\text{where $\Mvar$ is a set of \textbf{metavariables} $\bm{\alpha},\bm{\beta},\bm{\gamma},\ldots$}\\
	& \text{$($with or without subscripts$)$ with $\card{\Mvar}=\card{\Forms_{\mathcal{L}}}$ and so that}\\
	&\Mvar\cap\Cons=\emptyset~\text{and}~\Mvar\cap\Func=\emptyset;\\ 
	(\emph{b}) & \text{if $\phi_1,\ldots,\phi_n\in\Mform$ and $F\in\Func$ with $\#(F)=n$, then $F\phi_1\ldots\phi_n\in\Mform$};\\
	(\emph{c}) & \text{the $\Lan$-metaformalas are only those formal words which can be obtained}\\
	&\text{successively according to the items {\em(a)} and {\em(b)}.}\\
	\end{array}
	\]
\end{defn}

The notion of the \textit{\textbf{degree of metaformula}} is similar to that of formula.

The intended interpretation of metavariables is $\Lan$-formulas.\\

Now we turn again to formulas. We will be using the following notation: Given an $\Lan$-formula $\alpha$,
\[
\Var(\alpha)~\textit{is the set of variables that occur in $\alpha$}.
\]

It will be assumed that the set $\Var$ is decidable and the sets $\Cons$ and $\Func$ are decidable when they are countable; that is, for any $\Lan$-word $w$ of the length 1, there is an effective procedure which decides whether $w\in\Var$, whether $w\in\Cons$, or whether $w\in\Func$. In this case, it is also effectively decidable whether any $\Lan$-word is an $\Lan$-formula or not.

The \textit{\textbf{cardinality of a language}} $\Lan$, symbolically $\card{\Lan}$, is $\card{\Forms_{\mathcal{L}}}$. We will be using both notations. We note that for any language $\Lan$, $\card{\Lan}\ge\aleph_{0}$.

The notion of a \textit{\textbf{subformula of}} a (given) \textit{\textbf{formula}} is defined inductively through the notion of the \textit{\textbf{set}} $\Sub{\alpha}$ \textit{\textbf{of all subformulas}} \textit{\textbf{of}} a (given) \textit{\textbf{formula}} $\alpha$ as follows:
\[
\begin{array}{cl}
(\text{a}) &\text{if $\alpha\in\Var\cup\Cons$, then $\Sub{\alpha}=\{\alpha\}$};\\
(\text{b}) &\text{if $\alpha=F\alpha_1\ldots\alpha_n$, then $\Sub{\alpha}=\{\alpha\}\cup\Sub{\alpha_1}\cup
	\ldots\cup\Sub{\alpha_n}$}.
\end{array}
\]
A subformula $\beta$ of a formula $\alpha$ is called \textit{\textbf{proper}} if $\beta\in\Sub{\alpha}$ and $\beta\neq\alpha$. 

A useful view on a subfomula of a formula is that the former is a subword of the latter and is a formula itself.\\

The \textit{\textbf{formula tree}} for a formula $\alpha$ is a finite directed edge-weighted labeled tree where each node is labeled by a subformula of the initial formula $\alpha$.  The node labeled by $\alpha$ is the \textit{root} of the tree; that is its indegree equals 0. Suppose a node is labeled by  $\beta\in\Sub{\alpha}$. If $\beta\in\Var\cup\Cons$, we call this node a \textit{leaf}; its outdegree equals 0. Now assume that $\beta=F_{i}\beta_1\ldots\beta_n$, that is $\#(F_i)=n$. Then outdegree of this node is $n$. Since $\alpha$ may have more than one occurrence of a subformula $\beta$, all nodes labeled by the last formula will have outdegree equal to $\#(F_i)$. Then, $n$ directed edges will connect this node to $n$ nodes labeled by the subformulas $\beta_1,\ldots,\beta_n$, wherein the weight $i$ is assigned to the edge connecting $\beta$ with $\beta_i$. These $n$ nodes are distinct from one another even if some of $\beta_1,\ldots,\beta_n$ may be equal. We note that each occurrence of every subformula of $\alpha$ is a label of a formula tree for $\alpha$. Also, we observe that the degree of each $\beta_i$ is less than the degree of $\beta$. A formula tree for a given formula $\alpha$ is completed if each complete path from the root ends with a leaf. This necessarily happens, for the degree of a formula is a nonnegative integer.

We note that we use weighted edges only in order to identify a (unique) path from the root of a tree to a particular node.

It should be obvious that all formula trees for a given formula are isomorphic as graphs. So one can speak of \textit{the} formula tree for a given formula.

Let us illustrate the last definition for the formula $F_{i}F_{j}app$ with $\#(F_i)=\#(F_j)=2$.
Below we identify the nodes with their labels.

\begin{figure}[h!]
	\[
	\xymatrix{
		&&F_{i}F_{j}app \ar[dl]_1 \ar[dr]^2\\
		&F_{j}ap \ar[dl]_1 \ar[dr]^2 &&p\\
		a &&p
	}
	\]
	\caption{The formula tree for $F_{i}F_{j}apq$}
	\label{Fig-1}
\end{figure}
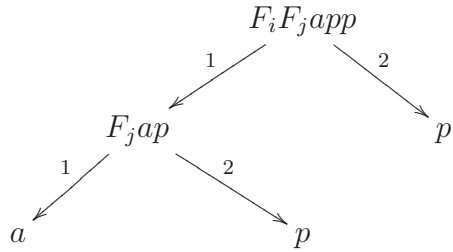

It should be clear that the formula tree for an initial formula allows us to work with each occurrence of every subformula of this formula separately and also to follow through the subformula relation even within a designated occurrence of a subformula of the initial formula. This observation will be used in the operation of replacement below.

Let us consider any leaf of the formula tree of Figure~\ref{Fig-1}, say the one labeled by $a$. We observe that the nodes of the path from the root to the selected leaf consist of all subformulas of the initial formula which contain $a$ as their subformula. Now let us color the first occurrence of $p$ in red and the second in blue so that we obtain
\[
F_{i}F_{j}a\textcolor{red}{p}\textcolor{blue}{p}. \tag{initial formula}
\]

We use color only for the sake of clarity to identify the path we are working with at the moment (see below); we could use weights of edges instead.

Next we consider the two paths: the first,
\[
F_{i}F_{j}a\textcolor{red}{p}\textcolor{blue}{p}\stackrel{1}{\longrightarrow} F_{j}a\textcolor{red}{p}
\stackrel{2}{\longrightarrow}\textcolor{red}{p},\tag{red}
\]
corresponds to the red occurrence of $p$ in the sense that  its nodes constitute all subformulas of the initial formula, which contain this red occurrence of $p$; the second path,
\[
F_{i}F_{j}a\textcolor{red}{p}\textcolor{blue}{p}\stackrel{1}{\longrightarrow}
\textcolor{blue}{p}, \tag{blue}
\]
corresponds to the blue occurrence of $p$ and satisfies the same property regarding this occurrence.
Now we substitute for both occurrences of $p$ in the initial formula any formula $\alpha$ so that we obtain the formula
\[
F_{i}F_{j}a\textcolor{red}{\alpha}\textcolor{blue}{\alpha}. \tag{red-blue $\alpha$}
\]
Instead of the tree of Figure 1 we will get the new tree:
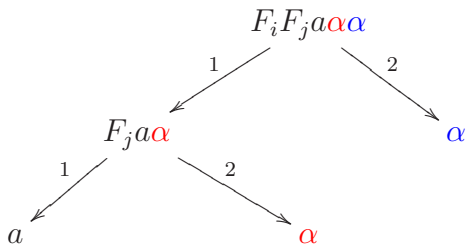
\begin{figure}[h!]
	\[
	\xymatrix{
		&&F_{i}F_{j}a\textcolor{red}{\alpha}\textcolor{blue}{\alpha}  \ar[dl]_1 \ar[dr]^2\\
		&F_{j}a\textcolor{red}{\alpha} \ar[dl]_1 \ar[dr]^2 &&\textcolor{blue}{\alpha}\\
		a &&\textcolor{red}{\alpha}
	}
	\]
	\caption{The transformed tree after the substitution $p\mapsto\alpha$}
	\label{Fig-2}
\end{figure}

If $\alpha\notin\Var\cup\Cons$, the tree of Figure~\ref{Fig-2} is not a formula tree. However, the last tree shows that the $\Lan$-word
$F_{i}F_{j}a\alpha\alpha$ obtained by this substitution is an $\Lan$-formula. Instead of the paths (red) and (blue) we observe the following paths:
\[
F_{i}F_{j}a\textcolor{red}{\alpha}\textcolor{blue}{\alpha}\stackrel{1}{\longrightarrow} F_{j}a\textcolor{red}{\alpha}
\stackrel{2}{\longrightarrow}\textcolor{red}{\alpha}\tag{red $\alpha$}
\]
and 
\[
F_{i}F_{j}a\textcolor{red}{\alpha}\textcolor{blue}{\alpha}\stackrel{2}{\longrightarrow}
\textcolor{blue}{\alpha}. \tag{blue $\alpha$}
\]
And again, we notice the same property: the nodes of (red $\alpha$) form the set of subsets of the red occurrence of $\alpha$ in the formula (red-blue $\alpha$); the same property regarding the blue occurrence of $\alpha$ is true for the path (blue $\alpha$).

This observation will help us discuss a syntactic transformation which is vital for the theory of sentential
logic. Essentially, this transformation makes speak of logic as a system of forms of judgments rather than that of individual judgments about facts.

First we define a transformation related to a formula $\alpha$ and a sentential variable $p$, denoting this operation by $\sigma^{\alpha}_{p}$. Before proceeding, we note that the initial formula $F_iF_japp$ was chosen rather arbitrarily; it might be that it does not contain $p$.
Having this in mind, given any $\Lan$-formula $\alpha$ and any sentential variable $p$, we define a
\textit{\textbf{unary substitution}} $\sigma^{\alpha}_{p}:\Forms_{\mathcal{L}}\longrightarrow\Forms_{\mathcal{L}}$ as follows:
\[
\begin{array}{cl}
\bullet & \text{if $\beta\in\Var\cup\Cons$ and $\beta\neq p$, then $\sigma^{\alpha}_{p}(\beta)=\beta$};\\
\bullet & \text{if $\beta = p$, then $\sigma^{\alpha}_{p}(\beta)=\alpha$};\\
\bullet &\text{if $\beta=F_{i}\beta_1\ldots\beta_n$, then $\sigma^{\alpha}_{p}(\beta)=F_{i}\sigma^{\alpha}_{p}(\beta_1)\ldots\sigma^{\alpha}_{p}(\beta_n)$}.
\end{array}
\]

At first sight, it may seem that unary substitution is not powerful enough. For example, let us take a formula $F_{i}F_{j}apq$, where  connectives $F_i$ and $F_j$ have the arity equal to 2. Assume that we want to produce a formula $F_{i}F_{j}a\alpha\beta$. Suppose we first obtain:
\[
\sigma^{\alpha}_{p}(F_{i}F_{j}apq)=F_{i}F_{j}a\alpha q.
\]
And we face a problem if $\alpha$ contains $q$, for then we get
\[
\sigma^{\beta}_{q}(F_{i}F_{j}a\alpha q)=F_{i}F_{j}a\sigma^{\beta}_{q}(\alpha)\beta,
\]
which is not necessarily equal to $F_{i}F_{j}a\alpha\beta$.
To overcome this difficulty, we select a variable which occurs neither in $\alpha$ nor in the initial formula $F_{i}F_{j}apq$. Suppose this variable is $r$. Then, we first produce:
\[
\sigma^{r}_{q}(F_{i}F_{j}apq)=F_{i}F_{j}apr.
\]
Now the desirable formula can be obtained in two steps:
\[
\begin{array}{l}
\sigma^{\alpha}_{p}(F_{i}F_{j}apr)=F_{i}F_{j}a\alpha r,\\
\sigma^{\beta}_{r}(F_{i}F_{j}a\alpha r)=F_{i}F_{j}a\alpha\beta.
\end{array}
\]
Thus we observe:
\[
\sigma^{\beta}_{r}\circ\sigma^{\alpha}_{p}\circ\sigma^{r}_{q}(F_{i}F_{j}apq)=F_{i}F_{j}a\alpha\beta.
\]

This leads to the conclusion that the composition of two and more unary substitutions may result in substitution of a different kind; namely, we can obtain the formula $F_{i}F_{j}a\alpha\beta$ from the formula $F_{i}F_{j}apq$ by  substituting simultaneously 
the formulas $\alpha$ and $\beta$ for the variables $p$ and $q$, respectively. Thus we arrive at the following definition.
\begin{defn}[substitution]\label{D:substitution} 
	A syntactic transformation $\sigma:\Forms_{\mathcal{L}}\longrightarrow\Forms_{\mathcal{L}}$ is called a \textit{\textbf{uniform}} $($or \textit{\textbf{simultaneous}}$)$ \textit{\textbf{substitution}} $($or simply a \textit{\textbf{substitution}} or $\Lan$-\textit{\textbf{substitution}}$)$ if it is an extension of a map $\sigma_{0}:\Var\longrightarrow
	\Forms_{\mathcal{L}}$, satisfying the following conditions$\,:$
	{\em\[
		\begin{array}{cl}
		(\text{a}) &\text{if $\alpha\in\Cons$, then $\sigma(\alpha)=\alpha$};\\
		(\text{b}) &\text{if $\alpha=F_{i}\alpha_1\ldots\alpha_n$, then $\sigma(\alpha)=F_{i}\sigma(\alpha_1)\ldots\sigma(\alpha_n)$}.
		\end{array}
		\]}
	Given a substitution $\sigma$ and a formula $\alpha$, $\sigma(\alpha)$ is called a \textbf{substitution instance} of $\alpha$. 
\end{defn}

For a set $X$ of $\Lan$-formulas and an $\Lan$-substitution $\sigma$, we denote:
\[
\sigma(X):=\set{\sigma(\alpha)}{\alpha\in X}~\text{and}~\sigma^{-1}(X):=\set{\alpha}{\sigma(\alpha)\in X}.
\]
Thus, for any substitution $\sigma$,
\[
\sigma(\emptyset)=\emptyset~\text{and}~\sigma^{-1}(\emptyset)=\emptyset.
\]
Also, we observe:
\begin{equation}\label{E:substitution-inequalities}
	i)~\sigma(\sigma^{-1}(X))\subseteq X\quad\text{and}\quad ii)~ X\subseteq\sigma^{-1}(\sigma(X)).
\end{equation}
(Exercise~\ref{section:languages}.\ref{EX:substitution-inequalities}) 

We leave for the reader to prove that, given a fixed formula $\alpha$, any simultaneous substitution applied to $\alpha$ results in the composition of singular substitutions applied to $\alpha$. (See Exercise~\ref{section:languages}.\ref{EX:substitution-2}.)
\begin{prop}\label{P:substitution}
	For any map $\sigma_{0}:\Var\longrightarrow\Forms_{\mathcal{L}}$, there is a unique substitution $\sigma:\Forms_{\mathcal{L}}\longrightarrow\Forms_{\mathcal{L}}$ extending $\sigma_0$.
\end{prop}
\noindent\textit{Proof}~is left to the reader. (Exercise~\ref{section:languages}.\ref{EX:substitution}.)\\

\textit{The last proposition allows us to think of a substitution simply as a map} $\sigma:\Var\longrightarrow\Forms_{\mathcal{L}}$, which will be convenient when we want to define a particular substitution. On the other hand, the property (b) of Definition~\ref{D:substitution} allows us to look at any substitution as a mapping satisfying a special algebraic property, which will also be convenient when we want to use this property.

For now, we want to note that composition of finitely many substitutions is a substitution. Moreover, for any substitutions $\sigma_1$, $\sigma_2$ and $\sigma_3$,
\begin{equation}\label{E:substitution-associative}
	\sigma_1\circ(\sigma_2\circ\sigma_3)=(\sigma_1\circ\sigma_2)\circ\sigma_3
\end{equation}
(We leave this property for the reader; see Exercise~\ref{section:languages}.\ref{EX:substitution-associative}.) This implies that the set of all substitutions, related to a language $\Lan$, along with composition $\circ$ constitute a monoid with the identity element
\[
\iota(p)=p,~\text{for any $p\in\Var$}
\]
which will be calling the \textit{\textbf{identity substitution}}. It must be clear that for any formula $\alpha$,
\[
\iota(\alpha)=\alpha.
\]
Because of \eqref{E:substitution-associative}, any grouping in
\[
\sigma_{1}\circ\ldots\circ\sigma_n
\]
leads to the same result. For this reason, parentheses can be omitted, if we do want to be specific. 

We will call a set $F\subseteq\Forms_{\mathcal{L}}$ \textit{\textbf{closed under substitution}}
if for any $\Lan$-formula $\alpha$ and substitution $\sigma$,
\[
\alpha\in F\Longrightarrow\sigma(\alpha)\in F.
\]

We have already noted that the tree of Figure~\ref{Fig-2} is not a formula tree if $\alpha\notin\Var\cup\Cons$. However, we can build the formula tree for $\alpha$. It does not matter how it looks; so we depict it schematically as follows:

\begin{figure}[h!]
	\[
	\xymatrix{
		\alpha \ar@{~}[d] \\
		\circ\circ\circ 
	}
	\]	
	\caption{The formula tree for $\alpha$}
	\label{Fig-3}
\end{figure}

Attaching two copies of the last tree to the tree of Figure~\ref{Fig-2}, we obtain the tree for $F_{i}F_{j}a\alpha\alpha$.

\begin{figure}[h!]
	\[
	\xymatrix{
		&&F_{i}F_{j}a\textcolor{red}{\alpha}\textcolor{blue}{\alpha}  \ar[dl]_1 \ar[dr]^2\\
		&F_{j}a\textcolor{red}{\alpha} \ar[dl]_1 \ar[dr]^2 &&\textcolor{blue}{\alpha} \ar@{~}[d]\\
		a &&\textcolor{red}{\alpha} \ar@{~}[d] &\circ\circ\circ\\
		&&\circ\circ\circ
	}
	\]
	\caption{The formula tree for $F_{i}F_{j}a\alpha\alpha$}
	\label{Fig-4}
\end{figure}
\pagebreak

We will call one copy of the formula tree for $\alpha$ \textit{the tree for red} $\alpha$ and the other copy \textit{the tree for the blue} $\alpha$.

Now let us take any formula $\beta$  and construct the formula tree for it depicted as follows:
\begin{figure}[h!]
	\[
	\xymatrix{
		\beta \ar@{~}[d] \\
		\bullet\bullet\bullet 
	}
	\]	
	\caption{The formula tree for $\beta$}
	\label{Fig-5}
\end{figure}

Now in the path
\[
F_{i}F_{j}a\textcolor{red}{\alpha}\textcolor{blue}{\alpha}\stackrel{1}{\longrightarrow}
F_{j}a\textcolor{red}{\alpha}\stackrel{2}{\longrightarrow} \textcolor{red}{\alpha}
\]
\pagebreak
of the tree of Figure~\ref{Fig-4}, we replace all red occurrences of $\alpha$ by $\beta$ and replace the tree for red alpha by the tree for $\beta$ (Figure~\ref{Fig-5}), arriving thus at the following tree:

\begin{figure}[h!]
	\[
	\xymatrix{
		&&F_{i}F_{j}a\beta\textcolor{blue}{\alpha}  \ar[dl]_1 \ar[dr]^2\\
		&F_{j}a\beta \ar[dl]_1 \ar[dr]^2 &&\textcolor{blue}{\alpha} \ar@{~}[d]\\
		a &&\beta \ar@{~}[d] &\circ\circ\circ\\
		&&\bullet\bullet\bullet
	}
	\]
	\caption{The formula tree for $F_{i}F_{j}a\beta\alpha$}
	\label{Fig-6}
\end{figure}

This is the formula tree for $F_{i}F_{j}a\beta\alpha$. The syntactic transformation just described shows that replacing a designated subformula of a formula, we obtain a formula. Such a transformation is called \textit{\textbf{replacement}}. Designating a subformula $\beta\in\Sub{\alpha}$, we write $\alpha[\beta]$. The result of the replacement of this particular subformula $\beta$ with a formula $\gamma$, we denote by $\alpha[\gamma]$.\\

Before going over examples of some sentential languages, we discuss how any metaformula can be seen in the light of the notion of substitution. It must be clear that the intended interpretation of any metaformula $\phi$ is the set of formulas of a particular shape.

Similar to Definition~\ref{D:substitution}, any map $\bm{\sigma}_0:\Mvar\longrightarrow
\Forms_{\mathcal{L}}$ can uniquely be extended to $\bm{\sigma}:\Mform\longrightarrow\Forms_{\mathcal{L}}$.  We will call both $\bm{\sigma}_0$ and its extension $\bm{\sigma}$ a \textit{\textbf{realization}}, or \textit{\textbf{instantiation}}, of metaformulas in $\Lan$. The following observation will be useful in the sequel. An instantiation $\bm{\sigma}$ is called \textit{\textbf{simple}} if it is the extension of $\bm{\sigma}_0:\Mvar\longrightarrow
\Var_{\mathcal{L}}$.
\begin{prop}\label{P:metaformula-instantiations}
	Let $\phi$ be an $\Lan$-metaformula and $\Sigma_\phi$ be the class of all instantiations of $\phi$. Then for any $\Lan$-substitution $\sigma$ and any $\Lan$-formula $\alpha$,
	\[
	\alpha\in\Sigma_\phi\Longrightarrow\sigma(\alpha)\in\Sigma_\phi.
	\]
\end{prop}
\noindent\textit{Proof}~is leaft to the reader. (Exercise~\ref{section:languages}.\ref{EX:metaformula-instantiations})\\

Below we give examples of languages that are specifications of the schematic language $\Lan$.
All these languages include a denumerable set $\Var$ as the set of sentential variables. 
\begin{description}
	\item[$\Lan_{A}$:] In this language, there is no sentential constants; there are four (assertoric) logical connectives: $\wedge$ (\textit{conjunction}), $\vee$ (\textit{disjunction}), $\rightarrow$ (\textit{implication} or \textit{conditional}), and $\neg$ (\textit{negation}), having the arities $\#(\wedge)=\#(\vee)=\#(\rightarrow)=2$ and
	$\#(\neg)=1$. It is customary to use an infix notation for this language; however, the use of
	infix notation for this language requires two additional symbols: ``('' and ``)'' called the left and right parenthesis, respectively. Accordingly, the definition of formula will be different: $(\alpha\wedge\beta)$, $(\alpha\vee\beta)$, $(\alpha\rightarrow\beta)$ and $\neg\alpha$ are formulas, whenever $\alpha$ and $\beta$ are. \textit{The parentheses do not apply to atomic formulas and negations}. Sometimes, we will not employ all the connectives of $\Lan_{A}$ but some of them, continuing to call this restricted language $\Lan_{A}$. Sometimes, $\Lan_{A}$ is considered in its extended version, when a connective $\leftrightarrow$ is added; if not, then the last connective is used as an abbreviation:
	\[
	\alpha\leftrightarrow\beta:=(\alpha\rightarrow\beta)\wedge(\beta\rightarrow\alpha).
	\]
	\item[$\Lan_{B}$:] This language extends $\Lan_{A}$ by adding to it two constants: $\top$ (\textit{truth} or \textit{verum}) and $\bot$ (\textit{falsehood} or \textit{falsity} or \textit{falsum} or \textit{absurdity}). Sometimes, only one of these constants will be employed.
	\textit{The parentheses do not apply to atomic formulas, negations and to the constants $\top$ and $\bot$.} We will call all these variants by $\Lan_{B}$, indicating every time the list of constants. 
	\item[$\Lan_{C}$:] This language extends either $\Lan_{A}$ or $\Lan_{B}$ with two sentential connectives (\textit{modalities}): $\Box$ (\textit{necessity}) and $\Diamond$ (\textit{possibility}) both of the arity 1, that is $\#(\Box)=\#(\Diamond)=1$. In addition to the corresponding requirements above, \textit{the parentheses do not apply to the formulas that start with the modalities $\Box$ and $\Diamond$.} Sometimes, when only one of these modalities is employed, we continue to call this language $\Lan_{C}$.
\end{description}

In the sequel we will employ more formal languages, which do not fit the proposed framework but closely related to it.\\

\pagebreak
\noindent{\textbf{Exercises~\ref{section:languages}}}
\begin{enumerate}
	\item\label{EX:formula-tree} Here we give another definition of a formula tree. This new definition is convenient when we deal with an algorithmic problem, in which case the sets $\Var$, $\Cons$ and $\Func$ are regarded countable and effectively decidable. One of such problems will be raised below as an exercise. We note that in this definition a designated path can be identified exclusively by the nodes it comprises.
	
	The \textit{\textbf{formula tree}} for a formula $\alpha$ is a finite directed labeled tree where each node is labeled with a pair $\langle n,\beta\rangle$, where $n$ is a natural number and $\beta$ is a subformula of $\alpha$.  
	
	There is a unique node, called the \textit{root}, labeled by $\langle\p{0},\alpha\rangle$, the ingdegree of which is $0$.\footnote{We remind the reader that $\p{0},\p{1},\p{2},\ldots$ is the sequence of prime numbers.} The ingdegree of all other nodes of the tree, if any, equals $1$.
	Let a node $v$ be labeled with $\langle n,\beta\rangle$, where $\beta\in\Var\cup\Cons$. Then the outdegree of $v$ is always $0$. Such nodes are called \textit{leaves}. (As the reader will see below, leaves are always present in a formula tree.)
	Thus if $\alpha\in\Var\cup\Cons$, the formula tree of $\alpha$ consists of only one node, the root, which is also a leaf.  Assume that $\beta=F\beta_1\ldots\beta_k$ with $F\in\mathcal{F}$ is a proper subformula of $\alpha$ and a node $v$ is labeled with $\langle n,\beta\rangle$. Suppose that $\p{j}$ is the greatest prime number that divides $n$.  Then the outdegree of $v$ is $\#(F)$, that is in our case, equals $k$. And then $v$ will be connected via $k$ edges to $k$ different nodes, $v_1,\ldots,v_k$. (The direction of each such edge is from $v$ to $v_i$.) We label $v_i$ with $\langle n\cdot\p{j+i},\beta_i\rangle$. Thus, 
	if $\beta_i=\beta_j$, then the two distinct nodes, $v_i$ and $v_j$, will be labeled with labels distinct in their first place and equal in the second. Also, we note that the degree of each $\beta_i$ is less than the degree of $\beta$. 
	
	Use the new definition of formula tree to prove the following.
	\begin{quote}
		{\em 	There is a mechanical procedure $($an algorithm$)$ which, given a list of pairs $\langle n_1,\alpha_1\rangle,\ldots,\langle n_k,\alpha_k\rangle$, where each $n_i$ is a positive integer, can assemble a formula tree, where the pairs are regarded as the nodes of the tree, or states that such a tree does not exist.}	
	\end{quote}
	\item \label{EX:substitution-inequalities} Prove the inclusions \eqref{E:substitution-inequalities}.
	\item\label{EX:substitution-2} Prove that $(a)$ for any fixed formula $\alpha$ and any fixed substitution $\sigma$, there are unary substations $\sigma_1,\ldots,\sigma_n$ such that $\sigma(\alpha)=\sigma_1\circ\ldots\circ\sigma_n(\alpha)$; however, $(b)$ for any fixed substitution $\sigma$, there is no finitely many fixed unary substitutions $\sigma_1,\ldots,\sigma_n$ such that for any formula $\alpha$, $\sigma(\alpha)=\sigma_1\circ\ldots\circ\sigma_n(\alpha)$.	
	\item\label{EX:substitution} Prove Proposition~\ref{P:substitution}.
	\item \label{EX:substitution-associative} Prove \eqref{E:substitution-associative}.
	\item \label{EX:metaformula-instantiations}Prove Proposition~\ref{P:metaformula-instantiations}.
	
\end{enumerate}

\section{Semantics of formal language}\label{section:semantics}
One of our concerns in this book is to show how truth can be transmitted through formulas of a formal language by formally presented deduction. Calling a formula true, we will mean a kind of truth which does not depend on any factual content. Moreover, sharing the view of the multiplicity of the nature of truth, we will be dealing rather with degrees of truth, calling them \textit{truth values}. Then, the idea of truth will find its realization in the idea of \textit{validity}.

Concerning the independence of any factual, contingent content in relation to truth values, let us consider the following two sentences --- `Antarctica is large' and `Antarctica is large or Antarctica is not large'. If our formal language contains a symbol for the grammatical disjunction ``$\ldots_{1}~\text{or}~\ldots_{2}$'', say ``$\ldots_{1}\vee\ldots_{2}$''
(as in the languages $\Lan_{A}$--$\Lan_{C}$ of Section~\ref{section:languages}), and a symbol for creating the denial of sentence, say ``$\neg\ldots$'' (also as in $\Lan_{A}$--$\Lan_{C}$), we can write in symbolic form the first sentence as $p$ and the second as $p\vee\neg p$. From the viewpoint of Section~\ref{section:languages}, both formal expressions are formulas of the language $\Lan_{A}$. Now, depending on familiarity with geography, as well as of personal experience, one can count $p$ true or false. The second formula $p\vee\neg p$, however, in some semantic systems takes steadily the truth value \textit{true}, regardless of any personal relation towards $p$. The \textit{universality} of this kind will be stressed in the idea of validity.

From the last example, it should be clear the distinction between the statements, the truth value of which (understood as the  degree of trust or that of informativeness) is dependent on the contingency of facts, the level of knowledge or personal beliefs, and the formal judgments, which obtain their logical status or meaning determined by \textit{semantic rules}. These rules will allow us to interpret the formal judgments of a language in abstract structures, reflecting the structural character of these judgments, as well as the universal character of their validity, thereby understanding their invalidity as the failure to satisfy this universality.

In this book, semantics will play pragmatic rather than explanatory role;\footnote{However, since we state validity conditions for formal judgments, these conditions generate the meanings of the judgments, according to the point of view that has been emphasized by Wittgenstein in \textit{Tractatus}; cf.~\cite{wittgenstein2001}, 4.024. Also, compare with Curry's statement  that ``the meaning of a concept is determined by the conditions under which it is introduced into discourse;'' cf.~\cite{curry1966}, section II.2.} namely we will be stressing the use of semantic structures as \textit{separating means}. This will be explained in Section~\ref{section:separating-means}. Moreover, we want to emphasize that in our approach to semantics we adhere to the relativistic position, according to which a single formal judgment can be valid according to one semantic system and invalid with respect to another.\\

Now we turn to the definition of semantics of the schematic language $\Lan$ exemplified by the languages $\Lan_{A}$--$\Lan_{C}$ in Section~\ref{section:languages}. Depending on one's purposes, or depending on the lines, along which the semantics in question can be employed, the semantics we propose can use the items of the following two categories --- the \textit{algebras of type} $\Lan$ and expansions of the latter called \textit{logical matrices}. 
Later on, we will generalize the notion of logical matrix, when we discuss consequence relation in Chapter~\ref{chapter:consequence}.

The following semantic concepts will be specified later on as needed.
\begin{defn}[algebra of type $\Lan$]
	An \textbf{algebra of type} $\Lan$ $($or an $\Lan$-\textbf{algebra} or simply an \textbf{algebra} when a type is specified ahead and confusion is unlikely$)$ is a structure {\em$\alg{A}=\langle\textsf{A};\Func,\Cons\rangle$}, where {\em\textsf{A}} is a nonempty set called the \textbf{carrier} $($or \textbf{universe}$)$ \textbf{of} {\em\alg{A}}$;$ $\Func$, which in the notion of $\Lan$-formulas represents logical connectives, now represents a set of finitary operations on {\em\textsf{A}} of the corresponding arities$;$ and $\Cons$, which is a set of logical constants, now represents a set of nullary operations on {\em\textsf{A}}. As usual, we use the notation {\em$|\alg{A}|=\textsf{A}$}. The part $\langle\Func,\Cons\rangle$ of the structure {\em\alg{A}} is called its \textbf{signature}. Thus, referring to an algebra {\em\alg{A}}, we can say that it is an algebra of signature $\langle\Func,\Cons\rangle$ or of signature $\Func$, if $\Cons=\emptyset$.
\end{defn}

Thus any $\Lan$-formula in relation to any algebra of type $\Lan$ becomes a \textit{\textbf{term}}. Naturally, the sentential variables of a formula understood as term become the \textit{\textbf{individual variables}} of this term. Since each algebraic term of, say, $n$ variables can be read as an $n$-ary derivative operation in an algebra of type $\Lan$, so can be understood any $\Lan$-formula containing $n$ sentential variables. 
\begin{defn}[inessential expansion of an algebra, equalized algebras]\label{D:equalized-algebras}
	Given an algebra {\em$\alg{A}=\langle\textsf{A};\Func,\Cons\rangle$}, we obtain an \textbf{inessential expansion of} {\em$\alg{A}$} if we add to $\Func$ one or more terms regarding them as new signature operations, or if we add to $\Cons$ one or more symbols with their intended interpretation as elements of {\em\textsf{A}}. Two algebras $($or their signatures$)$ are called \textbf{equalized} if they have inessential expansions of one and the same type.
\end{defn}

The central notion of Lindenbaum method is that of formula algebra associated with a particular language. The \textbf{\textit{formula algebra}} of type $\Lan$ has $\textit{\textbf{Fm}}_{\mathcal{L}}$ as its carrier and the signature operations as defined in Definition~\ref{D:L-formulas}. This algebra is denoted by $\mathfrak{F}_{\mathcal{L}}$. Thus we have:
\[
\mathfrak{F}_{\mathcal{L}}=\langle\textit{\textbf{Fm}}_{\mathcal{L}};\Func,\Cons\rangle.
\]

Analogously,  the set of $\Lan$-metaformulas along with the set $\Cons$ regarded as nullary operations and the set
$\Func$ regarded as non-nullary operations constitute a \textit{\textbf{metaformula algebra}} $\MformAl$ of type $\Lan$.\\

A relation between the $\Lan$-formulas and the algebras of type $\Lan$ is established through the following notion.
\begin{defn}[valuation, value of a formula in an algebra]\label{D:valuation}
	Let {\em\alg{A}} be an algebra of type $\Lan$. Then any map {\em$v:\Var\longrightarrow|\alg{A}|$}	is called a \textbf{valuation} $($or an \textbf{assignment}$)$ in the algebra {\em\alg{A}}. Then, given a valuation $v$, the value of an $\Lan$-formula
	$\alpha$ with respect to $v$, in symbols $v[\alpha]$, is defined inductively as follows$\,:$
	{\em\[
		\begin{array}{cl}
		(\text{a}) &v[p]=v(p),~\textit{for any $p\in\Var$};\\
		(\text{b}) &v[c]~\textit{is the value of the nullary operation of}~\alg{A} ~\textit{corresponding to}\\ 
		&\text{the constant $c$ of $\Lan$};\\
		(\text{c}) &v[\alpha]=Fv[\alpha_1]\ldots v[\alpha_n]~\textit{if $\alpha=F\alpha_1\ldots\alpha_n$}.
		\end{array}
		\]}
	If $\Var^{\prime}\subset\Var$, then a map {\em$v:\Var^{\prime}\longrightarrow|\alg{A}|$}	is called a $\Var^{\prime}$-\textbf{valuation}, or a \textbf{valuation restricted to} $\Var^{\prime}$.
\end{defn}

It makes sense to denote the mapping 
\[
\alpha\mapsto v[\alpha]
\]
by the same letter $v$. The following observation is a justification for this usage.
\begin{prop}\label{P:valuation}
	Let $v$ be a valuation in an algebra {\em\alg{A}}. Then the mapping $v:\alpha\mapsto v[\alpha]$ defines a homomorphism {\em$\mathfrak{F}_{\mathcal{L}}\longrightarrow\alg{A}$}. Conversely, each homomorphism {\em$\mathfrak{F}_{\mathcal{L}}\longrightarrow\alg{A}$} is a valuation.
\end{prop}
\noindent\textit{Proof}~of the first part is based on the clauses (a)--(c) of Definition~\ref{D:valuation}. The second part is obvious.\\

Thus, given a valuation $v$ in \alg{A} and a formula $\alpha(p_1,\ldots,p_n)$ where $p_1,\ldots,p_n$ are all variables occurring in $\alpha$, we observe that
\[
v[\alpha]=\alpha(v[p_1],\ldots,v[p_n]).
\]

Grounding on the last observation base in turn on Proposition~\ref{P:valuation}, very often a valuation is defined as a 
homomorphism $\mathfrak{F}_{\mathcal{L}}\longrightarrow\alg{A}$.\\

Comparing (b)--(c) of Definition~\ref{D:valuation} with (a)--(b) of Definition~\ref{D:substitution}, we obtain the following.
\begin{prop}\label{P:substitution-as-homomorphism}
	Any substitution is a valuation in $\mathfrak{F}_{\mathcal{L}}$ and vice versa. Hence any substitution is an endomorphism $\mathfrak{F}_{\mathcal{L}}\longrightarrow\mathfrak{F}_{\mathcal{L}}$.
\end{prop}

The converse of the last conclusion is also true: any endomorphism $\mathfrak{F}_{\mathcal{L}}\longrightarrow\mathfrak{F}_{\mathcal{L}}$ can be regarded as a substitution. (Exercise~\ref{section:semantics}.\ref{EX:endomorphism}.)\\ 

Now we turn to a central notion of any semantics --- the notion of\textit{ validity}.\footnote{W.~and M.~Kneale note, ``[\dots] logic is not simply valid argument but the reflection upon principles of validity;''~\cite{kneales1962}, section I.1.} This notion in turn is defined in the framework of the notion of logical matrix.

\begin{defn}[logical matrix of type $\Lan$]\label{D:matrix-type-L}
	A \textbf{logical matrix} $($or simply a \textbf{matrix}$)$ of type $\Lan$ is a structure {\em$\mat{M}=\langle\alg{A},D(x)\rangle$},  where {\em\alg{A}} is a nontrivial algebra of type $\Lan$, which, if considered as part of {\em\mat{M}}, is called an {\em\mat{M}}-\textbf{algebra}, and $D(x)$ is a predicate on {\em$|\alg{A}|$}. We will also denote a matrix as a structure {\em$\langle\textsf{A};\Func,\Cons,D(x)\rangle$} or as a structure
	{\em$\langle\textsf{A};\Func,D(x)\rangle$}, if $\Cons=\emptyset$. We will also denote a matrix as {\em$\langle\alg{A},D\rangle$} or as {\em$\langle\textsf{A};\Func,D\rangle$}
	or as {\em$\langle\textsf{A};\Func,D\rangle$}, if $\Cons=\emptyset$, where {\em$D\subseteq|\alg{A}|$}. In the last three notations $D$ is called  the set of \textbf{designated elements} or a \textbf{logical filter} $($or simply a \textbf{filter}$)$ in $($or of$)$ {\em\mat{M}}. 
\end{defn}

We should remark that using ``logical'' in the last two definitions was a bit hasty, for the first connection with logic emerges, at least in the case of logical matrices, in the next definition
when we separate semantically valid $\Lan$-formulas from the other formulas.
\begin{defn}\label{D:validity-1}
	Let {\em$\alg{M}=\langle\alg{A},D\rangle$} be a logical matrix of type $\Lan$. Then a formula $\alpha$ is \textbf{satisfied} by a valuation $v$ in {\em\alg{A}} if $v[\alpha]\in D$, in which case we say that $v$ \textbf{satisfies} $\alpha$ in {\em\mat{M}}{\em;} further, $\alpha$ is \textbf{semantically valid} $($or simply \textbf{valid}$)$ \textbf{in} {\em\alg{M}}, symbolically {\em$\alg{M}\models\alpha$} or {\em$\models_{\textbf{M}}\alpha$}, if for any valuation $v$ in {\em\alg{A}},
	$v[\alpha]\in D$. If a formula is not valid in a matrix, it is called \textbf{invalid} in this matrix or is \textbf{rejected} by the matrix or one can say that this matrix \textbf{rejects} the formula. If $v[\alpha]\notin D$ in {\em\alg{M}}, $v$ is a \textbf{valuation rejecting} or \textbf{refuting} $\alpha$. The set of all formulas valid in a matrix {\em\alg{M}} is denoted by {\em$L\alg{M}$} and is called the \textbf{logic of} {\em\mat{M}}.
\end{defn}

\begin{prop}\label{P:matrix-substitution-closedness}
	Let {\em$\alg{M}=\langle\alg{A},D\rangle$} be a matrix $($of type $\Lan$$)$. Then for any substitution $\sigma$, if {\em$\alg{M}\models\alpha$}, so is {\em$\alg{M}\models\sigma(\alpha)$}.
\end{prop}
\begin{proof}
	Let $v$ be any valuation in \alg{A}. According to Propositions~\ref{P:substitution-as-homomorphism} and~\ref{P:valuation}, $\sigma:\mathfrak{F}_{\mathcal{L}}\longrightarrow\mathfrak{F}_{\mathcal{L}}$ and
	$v:\mathfrak{F}_{\mathcal{L}}\longrightarrow\alg{A}$ are homomorphisms. Then $v\circ\sigma$ is also a homomorphism $\mathfrak{F}_{\mathcal{L}}\longrightarrow\alg{A}$ and hence, by Proposition~\ref{P:valuation}, is a valuation. Therefore,
	$\alg{M}\models\alpha$ implies that $v[\sigma(\alpha)]\in D$.
\end{proof}

The last proposition induces a way of how, given a matrix, to obtain a matrix on the formula algebra so that both matrices validate the same set of formulas.

Let $\alg{M}=\langle\alg{A},D\rangle$ be a matrix. We denote
\[
D_{\textsf{M}}=\set{\alpha}{\alg{M}\models\alpha}
\]
and then define a matrix
\[
\alg{M}_{\textit{Fm}}=\langle\mathfrak{F}_{\mathcal{L}},D_{\textsf{M}}\rangle.
\]
We observe that for any formula $\alpha$,
\[
\alg{M}_{\textit{Fm}}\models\alpha\Longleftrightarrow\alg{M}\models\alpha.
\tag{\textit{Lindenbaum equivalence}}
\]

Indeed, let $\alg{M}_{\textit{Fm}}\models\alpha$. Then for any substitution $\sigma$ (which, according to Proposition~\ref{P:substitution-as-homomorphism}, is a valuation in $\mathfrak{F}_{\mathcal{L}}$), $\sigma(\alpha)\in D_{\textsf{M}}$; in particular, $\iota(\alpha)=\alpha\in D_{\textsf{M}}$, that is $\alg{M}\models\alpha$.

Now suppose that $\alg{M}\models\alpha$, that is $\alpha\in D_{\textsf{M}}$. In virtue of Proposition~\ref{P:matrix-substitution-closedness}, $D_{\textsf{M}}$ is closed under substitution. This completes the proof of the equivalence above.

Next let us consider $\alg{M}=\langle\mathfrak{F}_{\mathcal{L}},D\rangle$. We observe that if $\alpha\in D_{\textsf{M}}$, that is $\alg{M}\models\alpha$, then, since, by virtue of Proposition~\ref{P:substitution-as-homomorphism}, each substitution is a valuation in $\mathfrak{F}_{\mathcal{L}}$ and vice versa, $\alpha\in D$.

Now let $D^{\prime}\subseteq D$ and $D^{\prime}$ be closed under substitution. Assume that $\alpha\in D^{\prime}$. Then, obviously, $\alg{M}\models\alpha$ and hence $\alpha\in  D_{\textsf{M}}$.

Thus we have obtained the following.
\begin{prop}
	For any matrix {\em\alg{M}}, there is a matrix on the formula algebra, namely {\em$\alg{M}_{\textit{Fm}}$}, such that both matrices validate the same set of $\Lan$-formulas. Moreover, given a matrix {\em$\alg{M}=\langle\mathfrak{F}_{\mathcal{L}},D\rangle$}, the filter of the matrix
	{\em$\alg{M}_{\textit{Fm}}=\langle\mathfrak{F}_{\mathcal{L}},D_{\textsf{M}}\rangle$} is the largest subset of $D$ closed under substitution. Thus, if $D$ itself is closed under substitution, then {\em$D=D_{\textsf{M}}$}.
\end{prop}

The importance of the last proposition for what follows is that it shows that the matrices of the form 
$\langle\mathfrak{F}_{\mathcal{L}},D\rangle$, where a filter $D$ is closed under substitution, suffice for obtaining all sets of valid $\Lan$-formulas, which can be obtained individually  by the matrices of type $\Lan$. This observation will be clarified below as \textit{Lindenbaum's Theorem}; see Proposition~\ref{P:lindenbaum-theorem}.\\

Now we turn to some special matrices which have played their important role in the development of sentential logic. In all these matrices, the filter of designated elements consists of one element. In this connection, we note the following equivalence, which will be used throughout:
If $D=\{c\}$, then
\begin{equation}\label{E+one-element-filter}
	x\in D\Longleftrightarrow x=c.
\end{equation}

In the sequel, we find useful the following property.
\begin{prop}\label{P:validity-homomorphism}
	Let {\em$\alg{M}_{1}=\langle\alg{A}_{1},D_{1}\rangle$} and {\em$\alg{M}_{2}=\langle\alg{A}_{2},D_{2}\rangle$} be matrices of the same type.
	Also, let $h$ be an epimorphism of {\em$\alg{A}_{1}$} onto {\em$\alg{A}_{2}$} such that
	$h[D_1]\subseteq D_2$. Then {\em$L\alg{M}_{1}\subseteq L\alg{M}_{2}$}
\end{prop}
\begin{proof}
	Let $v$ be a valuation in $\alg{A}_2$. Since $h$ is an epimorphisms, for each variable $p_{\gamma}$, there is an element $x_{\gamma}\in|\alg{A}_1|$ such that $h(x_\gamma)=v[p_\gamma]$. We define a valuation $v^{\prime}$ in $\alg{A}_1$ as follows:~$v^{\prime}[p_\gamma]=x_\gamma$, for any variable $p_\gamma$. It is obvious that $h(v^{\prime}[p_\gamma])=v[p_\gamma]$. Assume that $\alpha\in L\alg{M}_{1}$; in particular, $v^{\prime}[\alpha]\in D_1$. This implies that $v[\alpha]\in D_2$.
\end{proof}

\subsection{The logic of a two-valued matrix}\label{S:two-valued}

We begin with a well-known matrix of the so-called ``two-valued logic.'' The significance of this matrix in the history of logic is difficult to overestimate. The two-valued matrix was assumed, first, to distinguish true statements from false ones; and, second, to validate only the former. As we will see below, the main logic laws such as, for example, the law of identity, the law of contradiction, the law of the excluded middle (see below) and others can be demonstrated as valid formulas in this matrix. For this reason, it is worth considering this matrix first. Later on, we will explain another significance of this matrix, purely algebraic in nature, which leads to a very important view on this matrix as a refuting and separating means. \\

The two-valued matrix assumes two main variants of formulation --- as a matrix of type $\Lan_{A}$ and as that of type $\Lan_{B}$. As will be seen below, the difference between these variants concerns only semantic rules and does not affect essentially the set of valid formulas in the corresponding matrices.
For this reason, we denote both variants by $\booleTwo$. The elements of $\booleTwo$ are denoted by $\zero$ (\textit{false}) and $\one$ (\textit{true}), which constitute a Boolean algebra with $\zero$ as the least element and $\one$ as the greatest with respect to the partial ordering $\zero\leq \one$. According to the intended interpretation, $\one$ represents a truth value \textit{true}, $\zero$ represents \textit{false}. Being viewed as a 2-element Boolean algebra of type $\Lan_{A}$, the connectives $\wedge$, $\vee$ and $\neg$ are interpreted as meet, join and complementation, respectively, while
\begin{equation}\label{E:material-implication}
	x\rightarrow y:=\neg x\vee y.\footnote{We remind the reader that, in terms where the complementation (or pseudo-complementation) occurs, it precedes both meet and join when we interpret such terms in the lattices with complementation or pseudo-complementation.}
\end{equation}
This leads to the familiar truth tables:
\begin{center}
	\begin{tabular}{c|cc} \hline
		&\multicolumn{2}{|c}{$x\wedge y$}\\
		$x\backslash y$ &$\zero$  &$\one$\\ \hline
		$\zero$ &$\zero$
		&$\zero$ \\
		$\one$ &$\zero$
		&$\one$\\ \hline
	\end{tabular}
	\quad\begin{tabular}{c|cc} \hline
		&\multicolumn{2}{|c}{$x\vee y$}\\
		$x\backslash y$ &$\zero$  &$\one$\\ \hline
		$\zero$ &$\zero$ &$\one$\\
		$\one$ &$\one$
		&$\one$ \\ \hline
	\end{tabular}
	\quad\begin{tabular}{c|c} \hline
		$x$ &$\neg x$\\ \hline
		$\zero$ &~~$\one$\\
		$\one$ &~~$\zero$\\ \hline
	\end{tabular}
	\quad\begin{tabular}{c|cc} \hline
		&\multicolumn{2}{|c}{$x\rightarrow y$}\\
		$x\backslash y$ &$\zero$  &$\one$\\ \hline
		$\zero$ &$\one$ &$\one$\\
		$\one$ &$\zero$
		&$\one$ \\ \hline
	\end{tabular}
\end{center}

We observe that for any $x,y\in|\booleTwo|$,
\begin{equation}\label{E:implication}
	\begin{rcases*}
		\begin{array}{c}
			x\rightarrow y=\one\Longleftrightarrow x\leq y;\\
			x\rightarrow\zero=\neg x;\\
			\one\rightarrow x=x.
		\end{array}
	\end{rcases*}
\end{equation}

If one regards $\booleTwo$ (more exactly, its counterpart) as an algebra of type $\Lan_{B}$, two nullary operations, $\zero$ and $\one$, should be added to the signature, along with the intended interpretation of $\bot$ and $\top$ as $\zero$ and $\one$, respectively. Making of $\booleTwo$ a logical matrix, we add a logical filter $\{\one\}$. Thus a formula $\alpha$ of type $\Lan_{A}$ (or of type $\Lan_{B}$) is valid in $\booleTwo$ if for any valuation $v$,
\[
v[\alpha]=\one.
\] 
We denote both, the algebra and the matrix, grounded on this algebra, by $\booleTwo$. It must be clear that any $\Lan_{A}$-formula is valid in the matrix $\booleTwo$ of type $\Lan_{B}$ if and only if it is valid in $\booleTwo$ of type $\Lan_{A}$.

A formula $\alpha$ valid in $\booleTwo$ is called a \textit{\textbf{classical tautology}}. It is not difficult to see that the following formulas are classical tautologies:
\[
\begin{array}{cl}
p\rightarrow p, &(\textit{law of identity})\\
\neg(p\wedge\neg p), &(\textit{law of contradiction})\\
p\vee\neg p; &(\textit{law of the excluded middle})
\end{array}
\]
as well as the following:
\[\label{E:paradoxes-implicaion}
\begin{array}{cl}
\begin{rcases*}
p\rightarrow(q\rightarrow p);\\
\neg p\rightarrow(p\rightarrow q).
\end{rcases*} &(\text{``paradoxes of material implication''})
\end{array}
\]
The above tautologies had been known from ancient times. In modern times, such tautologies as the following have been in focus:
\[
\begin{array}{cl}
((p\rightarrow q)\rightarrow p)\rightarrow p; &(\textit{Peirce's law})\\
\neg\neg p\leftrightarrow p; &(\textit{double negation law})\\
\begin{rcases}
\neg(p\wedge q)\leftrightarrow\neg p\vee\neg q;\\
\neg(p\vee q)\leftrightarrow\neg p\wedge\neg q.
\end{rcases} &(\textit{De Morgan's laws})
\end{array}
\]

By the \textit{\textbf{two-valued logic}}, one usually means the set of all classical tautologies.

\subsection{The {\L}ukasiewicz logic of a three-valued matrix}\label{S:lukasiewicz}

Perhaps, the ``{\L}ukasiewicz 3-valued logic'' is the best-known of all multiple-valued systems alternative to the set of classical tautologies.
The matrix which defines this ``logic'' has three values, $\zero$ (\textit{false}), $\tb$ (\textit{indeterminate}) and
$\one$ (\textit{true}), which are arranged by a linear ordering as follows: $\zero\leq\tb\leq\one$. On this carrier, an algebra of type $\Lan_{A}$ is defined in such a way that $\wedge$ and $\vee$ are interpreted as meet and join in the 3-element distributive lattice; then, $\neg$ and $\rightarrow$ are defined according to the following tables:

\begin{center}
	\begin{tabular}{c|c} \hline
		$x$ &$\neg x$\\ \hline
		$\zero$ &~~$\one$\\
		$\tb$  &~~$\tb$\\
		$\one$ &~~$\zero$\\ \hline
	\end{tabular}
	\quad\begin{tabular}{c|ccc} \hline
		&\multicolumn{3}{|c}{$x\rightarrow y$}\\
		$x\backslash y$ &$\zero$  &$\tb$ &$\one$\\ \hline
		$\zero$ &$\one$ &$\one$ &$\one$\\
		$\tb$  &$\tb$ &$\one$ &$\one$\\
		$\one$ &$\zero$ &$\tb$	&$\one$ \\ \hline
	\end{tabular}
\end{center}

We denote this algebra by $\lukasThree$.
Actually, when we have defined $\neg$ in $\lukasThree$ as above, the operation $\rightarrow$ can be defined according to \eqref{E:implication}. Moreover, the following identities can be easily checked:
\begin{equation}\label{E:lukasiewicz}
	\begin{array}{c}
		\begin{rcases*}
			x\vee y=(x\rightarrow y)\rightarrow y;\\
			x\wedge y=\neg(\neg x\vee\neg y).
		\end{rcases*}
	\end{array}
\end{equation}
(Exercise~\ref{section:semantics}.\ref{EX:lukasiewicz}.)
Also, we note that, as algebras of type $\Lan_{A}$, $\booleTwo$ is a subalgebra $\lukasThree$. 
(See Exercise~\ref{section:semantics}.\ref{EX:subalgebra}.)

Now taking $\{\one\}$ as a logical filter, we obtain the {\L}ukaciewicz 3-valued matrix which is also denoted by $\lukasThree$.

We note that the law of the excluded middle and Peirce's law, which are tautologies of the two-valued logic, fail to be tautologous in $\lukasThree$. (Exercise~\ref{section:semantics}.\ref{EX:failure}.)
The law of contradiction is also invalid in $\lukasThree$, for the set $\{\tb\}\subseteq|\lukasThree|$ is closed under $\wedge$ and $\neg$.

By the {\L}ukasiewicz 3-valued logic one means the set of formulas which are validated by $\lukasThree$. These formulas are called $\lukasThree$-\textit{\textbf{tautologies}}. It should be clear that the set 
$\lukasThree$-tautologies is a proper subset of the set of classical tautologies.
(Exercise~\ref{section:semantics}.\ref{EX:luk-three-included-boole-two}.) One time, it had been assumed that the matrix $\langle\lukasThree,\{\tb,\one\}\rangle$ validates the classical tautologies and only them. If the first part of the last statement is obviously true, the second is false. The following counterexample is due to A.~R. Turquette: the formula $\neg(p\rightarrow\neg p)\vee\neg(\neg p\rightarrow p)$ is a classical tautology but is refuted in the last matrix by the assignment $v[p]=\tb$.

\subsection{The {\L}ukasiewicz modal logic of a three-valued matrix}\label{S:lukasiewicz-modal}

Both the algebra and matrix $\lukasThree$ can be expanded to type $\Lan_{C}$ as follows:
\begin{center}
	\begin{tabular}{c|c|c} \hline
		$x$ &$\Box x$ &$\Diamond x$\\ \hline
		$\zero$ &~$\zero$ &~$\zero$\\
		$\tb$ &~$\zero$ &$\one$\\
		$\one$ &~$\one$ &$\one$\\ \hline
	\end{tabular}
\end{center}

The last definitions of the modalities $\Box$ and $\Diamond$ are not independent of the other assertoric connectives, for in the expanded algebra $\lukasThree$ the following identities hold:
\begin{equation}\label{E:lukasiewicz-modal}
	\begin{array}{c}
		\begin{rcases*}
			\Diamond x=\neg x\rightarrow x;\\
			\Box x=\neg\Diamond\neg x.
		\end{rcases*}
	\end{array}
\end{equation}
We leave for the reader to check this. (See Exercise~\ref{section:semantics}.\ref{EX:lukasiewicz-modal}.)

We observe that the formulas
\[
\Box p\rightarrow p,~p\rightarrow\Diamond p~\text{and}~\Box p\rightarrow\Diamond p
\]
are valid in the expanded matrix, while their converses,
\[
p\rightarrow\Box p,~\Diamond p\rightarrow p~\text{and}~\Diamond p\rightarrow\Box p,
\]
are rejected by it. (See Exercise~\ref{section:semantics}.\ref{EX:lukasiewicz-modal-2}.)

\subsection{The G\"{o}del $n$-valued logics}\label{section:goedel}
As in the case $\lukasThree$, let us consider the three distinct ``truth values,'' $\zero$, $\tb$ and
$\one$, arranged by a linear ordering as follows: $\zero\leq\tb\leq\one$. Regarding this chain as a Heyting algebra with $\wedge$, $\vee$, $\rightarrow$ and $\neg$ as meet, join, relative pseudo-complementation and pseudo-complementation, respectively, we get an algebra $\godelThree$ of type $\Lan_{A}$. Thus we arrive at the following truth tables:
\begin{center}
	\begin{tabular}{c|ccc} \hline
		&\multicolumn{3}{|c}{$x\wedge y$}\\
		$x\backslash y$ &$\zero$ &$\tb$  &$\one$\\ \hline
		$\zero$ &$\zero$
		&$\zero$ &$\zero$\\
		$\tb$ &$\zero$ &$\tb$ &$\tb$\\
		$\one$ &$\zero$ &$\tb$
		&$\one$\\ \hline
	\end{tabular}
	\quad\begin{tabular}{c|ccc} \hline
		&\multicolumn{3}{|c}{$x\vee y$}\\
		$x\backslash y$ &$\zero$ &$\tb$ &$\one$\\ \hline
		$\zero$ &$\zero$ &$\tb$ &$\one$\\
		$\tb$ &$\tb$ &$\tb$ &$\one$\\
		$\one$ &$\one$
		&$\one$ &$\one$\\ \hline
	\end{tabular}
	\quad\begin{tabular}{c|c} \hline
		$x$ &$\neg x$\\ \hline
		$\zero$ &~~$\one$\\
		$\tb$ &~~$\zero$\\
		$\one$ &~~$\zero$\\ \hline
	\end{tabular}
	\quad\begin{tabular}{c|ccc} \hline
		&\multicolumn{3}{|c}{$x\rightarrow y$}\\
		$x\backslash y$ &$\zero$ &$\tb$ &$\one$\\ \hline
		$\zero$ &$\one$ &$\one$ &$\one$\\
		$\tb$ &$\zero$ &$\one$ &$\one$\\
		$\one$ &$\zero$ &$\tb$
		&$\one$ \\ \hline
	\end{tabular}
\end{center}

We notice that all properties \eqref{E:implication} are satisfied in $\godelThree$.

Taking into account the relation $\leq$, the above truth tables can be read through the following equalities.
\begin{equation}\label{E:LC-algebra-operations}
	\begin{rcases}
		\begin{array}{l}
			x\wedge y=\min(x,y)\\
			x\vee y=\max(x,y)\\
			x\rightarrow y=\begin{cases}
				\begin{array}{cl}
					\one &\text{if $x\le y$}\\
					y &\text{otherwise}
				\end{array}
			\end{cases}\\
			\neg x=\begin{cases}
				\begin{array}{cl}
					\zero &\text{if $x\neq\zero$}\\
					\one &\text{otherwise},
				\end{array}
			\end{cases}
		\end{array}
	\end{rcases}
\end{equation}

Taking $\{\one\}$ as a logical filter, we obtain a logical matrix $\godelThree$.
The last matrix can be generalized if instead of 3 elements we take 4 and more elements and use \eqref{E:LC-algebra-operations} to define the same signature operations as in $\godelThree$. Thus we obtain G\"{o}del's $n$-valued logics $\textbf{G}_n$, where $n\geq 3$. We note that \eqref{E:LC-algebra-operations} are also valid in $\booleTwo$; therefore, it is sometimes convenient to rename  the algebra $\booleTwo$ by $\textbf{G}_2$.

The matrices $\textbf{G}_n$, $n\ge 2$, have been in particular important in investigation of the lattice of intermediate logics, that is the logics the valid formulas of which contain all \textit{intuitionistic tautologies} and are included into the set of classical tautologies. In particular, all the matrices $\textbf{G}_n$, $n\ge 3$, fail Peirce's law, the double negation law, as well as both De Morgan's laws. (Exercise~\ref{section:semantics}.\ref{EX:Peirce-double-negation}.)

\subsection{The Dummett denumerable matrix}\label{section:dummett}
An denumerable generalization of $\textbf{G}_n$, $n\ge 2$, can be obtained if we take take a denumerable set and arrange it by a linear ordering $\le$ according to type $1+\omega^{\ast}$. We denote the least element by $\zero$ and the greatest by $\one$. (We note: the ordinal notation of the former is $\omega$ and of the latter is $0$.) Then, introducing the operations $\wedge$, $\vee$, $\rightarrow$ and $\neg$ as meet, join, relative pseudo-complementation and pseudo-complementation, respectively, of a Heyting algebra and equipping the latter algebra with a logical filter $\{\one\}$, we receive the \textit{Dummett denumerable matrix} \textbf{LC}. As for $\godelThree$, the algebra of the Dummett matrix satisfies the properties~\eqref{E:implication} and also~\eqref{E:LC-algebra-operations}. We note that \textbf{LC} validates the formula
\[
(p\rightarrow q)\vee(q\rightarrow p).
\]

The set of all formulas valid in this matrix, $L\mat{LC}$, is called the \textit{\textbf{Dummett logic}}. 

M. Dummett shows in~\cite{dummett1959}, theorem 2, that
\begin{equation}\label{E:LC-characterization}
	L\mat{LC}=\bigcap_{n\ge 2}L\mat{G}_{n}.
\end{equation}

In the sequel, we will need another logical matrix whose logic coincides with the Dummett logic $\LC$. 

Let us consider a denumerable set arranged linearly according to type $\omega+1$.
We treat this structure as a Heyting algebra with a least element $\zero$, greatest element $\one$ and operations $\wedge$, $\vee$, $\rightarrow$ and $\neg$ that are defined according to~\eqref{E:LC-algebra-operations}. Announcing $\lbrace\one\rbrace$ a logical filter, we define the logical matrix $\textbf{LC}^{\ast}$.
\begin{prop}\label{P:LC=LC-ast}
	{\em$L\textbf{LC}=L\textbf{LC}^{\ast}$}.
\end{prop}
\begin{proof}
	We show that
	\[
	L\mat{LC}^{\ast}=\bigcap_{n\ge 2}L\mat{G}_{n}.
	\]
	It is clear that for any $n\ge 2$, the algebra of $\mat{G}_n$ is a homomorphic image of the algebra of $\mat{LC}^{\ast}$; see Section~\ref{section:heyting-algebra}. In virtue of Proposition~\ref{P:validity-homomorphism}, $L\mat{LC}^{\ast}\subseteq\bigcap_{n\ge 2}L\mat{G}_{n}$.
	
	Now, assume that a valuation $v$ in $\mat{LC}^{\ast}$ refutes a formula $\alpha$, that is $\alpha\notin L\mat{LC}^{\ast}$. Let $\beta_1,\ldots,\beta_n$ be all subformulas of $\alpha$. Then, according to~\eqref{E:LC-algebra-operations}, the set $\lbrace v[\beta_1],\ldots,v[\beta_n],\zero,\one\rbrace$ is the carrier of a subalgebra of the $\mat{LC}^{\ast}$-algebra, which in turn is isomorphic to a $\mat{G}_{n}$-algebra, for some $n\ge 2$. This implies that $\alpha\notin\bigcap_{n\ge 2}L\mat{G}_{n}$. 
\end{proof}

\subsection{Two infinite generalizations of {\L}ukasiewicz' $\lukasThree$}

The {\L}ukasiewicz algebra $\lukasThree$ can be easily generalized if we take for a carrier the interval $[0,1]$ or the rational scale $\mathbb{Q}$ or on the real scale $\mathbb{R}$ arranged in both cases by the ordinary less-than-or-equal relation $\le$. The set of designated elements in both cases is $\{1\}$.
The operations $\neg$ and $\rightarrow$ are defined as follows:
\[
\begin{array}{l}
\neg x:=1-x,\\
x\rightarrow y:=\min(1,1-x+y);
\end{array}
\] 
then, the other operations are defined according to \eqref{E:lukasiewicz}. This leads to the identities
\[
\begin{array}{c}
x\wedge y=\min(x,y)~\text{and}~x\vee y=\max(x,y),
\end{array}
\]
which can be taken as definitions of $\wedge$ and $\vee$ as well.

The set of formulas valid in the first matrix, that is when $[0,1]$ is considered in the rational scale, is denoted by $\mat{\L}_{\aleph_{0}}$ and the formulas valid in the second matrix by $\mat{\L}_{\aleph_{1}}$.\\

So far, we have discussed the semantic rules for the language $\Lan$. As we introduce other schematic languages later on, we will have to adjust these rules for them. However, although semantic rules allow us to distinguish valid formulas from refuted ones, they, taken alone, do not suffice to
analyze reasoning. This echoes the remark due to Aristotle that ``it is absurd to discuss refutation (\textgreek{>'elegqoc}) without discussing reasoning (\textgreek{snllogism'os});'' cf.~\cite{kneales1962}, section I.4. We discuss the relation of formal deduction between $\Lan$-formulas  in the form of \textit{consequence relation}. This is a topic of the next chapter.\\

\noindent{\textbf{Exercises~\ref{section:semantics}}}
\begin{enumerate}
	\item\label{EX:endomorphism}Show that any endomorphism $f:\mathfrak{F}_{\mathcal{L}}\longrightarrow\mathfrak{F}_{\mathcal{L}}$ can be regarded as a substitution; that is $f$ is a unique extension of $v:\Var\longrightarrow\mathfrak{F}_{\mathcal{L}}$ such that $v(p)=f(p)$, for any $p\in\Var$.
	\item Prove the second part of Proposition~\ref{P:valuation}.
	\item Prove that the laws of identity, of contradiction, of the excluded middle, as well as Peirce's and double negation laws are tautologies in $\booleTwo$.
	\item Prove the properties~\eqref{E:implication}.
	\item Show that for any two $\Lan_{A}$-formulas $\alpha$ and $\beta$, $\alpha\leftrightarrow\beta$ is a classical tautology if and only if for any $\booleTwo$-valuation $v$, $v[\alpha]=v[\beta]$.
	\item\label{EX:lukasiewicz}Prove the identities \eqref{E:lukasiewicz}.
	\item\label{EX:subalgebra} Prove that $\booleTwo$ is a subalgebra of $\lukasThree$.
	\item\label{EX:failure} Show that the law of the excluded middle and Peirce's law are not tautologous in $\lukasThree$.
	\item\label{EX:luk-three-included-boole-two}Prove that the set of $\lukasThree$-tautologies is properly included in the set of classical tautologies. 
	\item Prove that the ``paradoxes of material implications'' (p.~\pageref{E:implication}) are $\lukasThree$-tautologous.
	\item\label{EX:wajsberg-axioms} Show that the following formulas are $\lukasThree$-tautologies.
	\[
	\begin{array}{cl}
	(a) &p\rightarrow(q\rightarrow p);\\
	(b) &(p\rightarrow q)\rightarrow((q\rightarrow r)\rightarrow(p\rightarrow r));\\
	(c) &(\neg p\rightarrow\neg q)\rightarrow(q\rightarrow p);\\
	(d) &((p\rightarrow\neg p)\rightarrow p)\rightarrow p.
	\end{array}
	\]
	\item\label{EX:lukasiewicz-modal}Prove that the identities~\eqref{E:lukasiewicz-modal} hold the algebra $\lukasThree$ expanded by the modalities $\Box$ and $\Diamond$.
	\item\label{EX:lukasiewicz-modal-2}Show that the logical matrix of type $\Lan_{C}$ of Section~\ref{S:lukasiewicz-modal} validates the formulas $\Box p\rightarrow p,~p\rightarrow\Diamond p~\text{and}~\Box p\rightarrow\Diamond p$ and rejects their converses $p\rightarrow\Box p,~\Diamond p\rightarrow p~\text{and}~\Diamond p\rightarrow\Box p$.
	\item\label{EX:Peirce-double-negation}Show that Peirce's law, the double negation law and both De Morgan's laws can be refuted in $\godelThree$.
	\item Let us designate on the algebra $\godelThree$ the logical filter $\{\tb,\one\}$. We denote the new matrix by $\godelThree^{\prime}$. Show that $\godelThree^{\prime}$ validates a formula $\alpha$ if and only if $\alpha$ is a classical tautology.
\end{enumerate}

\section{Historical notes}\label{languages-historical-notes}
The idea that logic, as the science of reasoning, deals with the universal rather than the particular can be traced back to Plato and Aristotle. According to the Platonic doctrine of Forms, a single Form is associated with each group of things to which a common name can be applied. Plato believed that each Form differs from the particulars it represents but, at the same time, is a particular of its own kind.
Applying this doctrine to reasoning, he believed that correct reasoning is based on the connections of the involved Forms. This is how W. Kneale and M. Kneale summarize Plato's view on reasoning.
\begin{quote}
	``For Plato necessary connexions hold between Forms, and inference is presumably valid when we follow in thought the connexions between Forms as they are. How is this view related to the theory which assigns truth and falsity to sentences made up of nouns and verbs? Plato seems to hold that a sentence is true if the arrangement of its parts reflects or corresponds to connexion between Forms.''
	(\cite{kneales1962}, section I.5.)
\end{quote}

However, what Plato wrote about correct thinking can be regarded as philosophy of logic rather than logic itself. The first construction of formal logic we encounter in Aristotle. Although Aristotle was influenced by Plato's doctrine of Forms, he rejected it. Nevertheless, we dare to assert that Aristotle's doctrine of the syllogism is \textit{formal} in the Platonic sense. Namely Aristotle was the first to introduce variables into logic. In his influential book on Aristotle's syllogistic J. {\L}ukasiewicz makes the following comment about this fact.
\begin{quote}
	``It seems that Aristotle regarded his invention as entirely plain and requiring no explanation, for there is nowhere in his logical works any mention of variables. It was Alexander who first said explicitly that Aristotle presents his doctrine in letters, \textgreek{stoiqe{\^{i}}a}, in order to show that we get the conclusion not in consequence of the matter of the premisses, but in consequence of their form and combination; the letters are marks of universality and show that such a conclusion will follow always and for any term we may choose.'' (\cite{luk1957}, {\S} 4.)
\end{quote}

Although in the course of the development of logic, when after Aristotle the generalized nature of formal judgments was reflected in the use of letters, and the latter were perceived, at least gradually, as sentential variables, and therefore it would be natural to consider the uniform substitution, as defined above, as a legitimate \textit{operation} that would apply to all formal judgments, however, the uniform substitution was first proposed as an \textit{inference rule}, that is, only in conjunction with (and for constructing) logical inferences. 

Nonetheless, the realization, though implicit and at least for propositional languages such as $\Lan$-languages above, of that the set of formulas is closed under substitution and that the notion of validity should be defined in such a way to maintain this closedness for the formulas identified as valid one can find in George Boole. As to the first part of this observation, Boole puts it as follows.
\begin{quote}
	``Let us imagine any known or existing language freed from idioms and divested of superfluity and let us express in that language any given proposition in a manner the most simple and literal [\ldots] The transition from such a language to the notation of analysis would consist of no more than the substitution of one set of signs for another, without essential change either in form or character.'' (\cite{boole1854}, chapter 11, section 15.)
\end{quote}
As to the second, he writes:
\begin{quote}
	``I shall in some instances modify the premises by [\ldots] substitution of some new proposition, and shall determine how by such change the ultimate conclusions are affected.'' (\cite{boole1854}, chapter 13, section 1.)
\end{quote}

It appears that this idea has become fully realized by the 1920s in the Lvov-Warsaw School, for without this, Lindenbaum's Theorem (Proposition~\ref{P:lindenbaum-theorem}) would be impossible. Although we can only guess, but it seems plausible to suggest that the confusion between the concepts of predicate and propositional term was an obstacle for such a long delay.\\

With merely some reservation, one can say that George Boole was the first who interpreted
symbolic expressions representing ``laws of thought'' as algebraic terms. However,
his algebraic treatment of logic was merely the beginning of a long way. This is how T. Hailperin estimates Boole's contribution.   
\begin{quote}
	``Boole wrote his \textit{Laws of Thought} before the notion of an abstract formal system, expressed within a precise language, was fully developed. And, at that time, still far in the future was the contemporary view which makes a clear distinction between a formal system and its realization or models. Boole's work contributed to bringing these ideas to fruition.'' (\cite{hailperin1986}, chapter 1, {\S} 1.0.)
\end{quote}

We find interpretation of formal judgments in an abstract logical matrix in, namely $\booleTwo$,  Wittgenstein's \textit{Tractatus}; cf.~\cite{wittgenstein2001}, 4.441 and further. Also, Wittgenstein used the term \textit{tautology}, applying it to any classical tautology. Later on, logical matrices turned to be a natural way to define numerous deductive systems. We owe to Alfred Tarski a systematic approach to deductive systems through logical matrices.  

At first matrices were employed in the search for an independent axiomatic system for the set of classical tautologies; cf.~\cite{bernays1926} and~\cite{lukasiewicz1929}, section 6. Also, J. C. C. McKinsey~\cite{mckinsey1939} used matrices to prove independence of logical connectives in intuitionistic propositional logic.\\

The idea of the use of a 3-valued matrix was expressed by Jan {\L}ukasiewicz in his  inaugural address as Chancellor of the University of Warsaw in 1922. In 1961, a revised version of this speech was published, where one can read:
\begin{quote}
	``I maintain that there are propositions which are neither true nor false but \textit{indeterminate}. All sentences about future facts which are not yet decided belong to this category. Such sentences are neither true at the present moment, for they have no real correlate, nor are they false, for their denials too have no real correlate. If we make use of philosophical terminology which is not particularly clear, we could say that ontologically there corresponds to these sentences neither being nor non-being but possibility. Indeterminate sentences, which ontologically have possibility as their correlate, take the third truth-value.'' (\cite{lukasiewicz1970}, p. 126.)
\end{quote}

Technically, however, {\L}ukasiewicz proposed $\lukasThree$ already in 1920; cf.~\cite{lukasiewicz1920}. In~\cite{lukasiewicz1930} he expands $\lukasThree$ to include modalities and thus obtain the matrix of Section~\ref{S:lukasiewicz-modal}. In fact, he defines only $\Diamond$, but it is easy to check that the definition
\[
\Box x:=\neg\Diamond\neg x
\] 
gives operation $\Box$ as introduced in Section~\ref{S:lukasiewicz-modal} directly.

It is worth noting that Lukasiewicz's interest in "multivalued logic" (in the sense of logical matrices having more than two values) was motivated by his interest in modal contexts, and not vice versa. Indeed, he wrote:
\begin{quote}
	``I can assume without contradiction that my presence in Warsaw at a certain moment of next year, e.g. at noon on 21 December, is at the present time determined neither positively nor negatively, Hence it is \textit{possible}, but not \textit{necessary}, that I shall be present in Warsaw at the given time. On this assumption the proposition `I shall be in Warsaw at noon on 21 December of next year', can at the present time be neither true nor false. For if it were true now, my future presence in Warsaw would have to be necessary, which is contradictory to the assumption. If it were false now, on the other hand, my future presence in Warsaw would have to be impossible, which is also contradictory to the assumption. Therefore the proposition considered is at the moment \textit{neither true nor false} and must possess a third value, different from `0' or falsity and `1' or truth. This value we can designate by `$\frac{1}{2}$'. It represents `the possible', and joins `the true' and `the false' as a third value.'' (\cite{lukasiewicz1930}, {\S} 6; cf. English translation in~\cite{mccall1967}, pp. 40--65.)
\end{quote}

The logical matrix $\textbf{L}_{\aleph_{1}}$ and its logic for the first time were discussed by {\L}ukasiewicz in his address delivered at a meeting of the Polish Philosophical Society in Lvov on October 14, 1922. The report of this lecture was published in {\em Ruch Filozoficzny}, vol. 7 (1923), no. 6,  pp. 92--93; see English translation in \cite{lukasiewicz1970}, pp. 129--130. According to this report, {\L}ukasiwicz's motivation for this matrix was as follows.
\begin{quote}
	``The application of this interpretation is twofold: 1) It can be demonstrated that if those verbal rules, or directives, which are accepted by the authors of {\em Principia Mathematica} (the rule of deduction and the rule of substitution) are adopted, then no set of the logical laws that have the numerical value 1 can yield any law that would have a lesser numerical value. 2) If 0 is interpreted as falsehood, 1 as truth, and other numbers in the interval 0--1 as the degrees of probability corresponding to various possibilities, a many-valued logic is obtained, which is an expansion of three-valued logic and differs from the latter in certain details.'' (\cite{lukasiewicz1970}, p. 130.)
\end{quote}

It is clear that $\textbf{L}_{\aleph_{0}}$ is an obvious simplification of $\textbf{L}_{\aleph_{1}}$.\\

Answering a question of Hahn, Kurt G\"{o}del considered~\cite{godel1932} (see English translation in~\cite{godel1986}) an infinite series of finite logical matrices and their logics, where the matrix $\booleTwo$ was the second and $\godelThree$ the third in the series. All these matrices were partially ordered chains with the maximal element designated. G\"{o}del's goal was twofold: to show that 1) the intuitionistic propositional calculus cannot be determined by a single matrix; and that 2) there are at least countably many logics between the \textit{intuitionistic} and \textit{classical logics}. (These logics will be descussed in more detail in Chapter~\ref{chapter:consequence}.)\\

The aforementioned series of G\"{o}del's matrices produces a descending chain of logics all containing the intuitionistic propositional logic. Michael Dummett showed~\cite{dummett1959} that \textbf{LC} equals the intersection of all G\"{o}del's logics. He also gave an axiomatization of \textbf{LC}, which will be discussed in Chapter~\ref{chapter:consequence}.

\chapter[Logical Consequence]{Logical Consequence}\label{chapter:consequence}	
The conception of logical consequence which we discuss in this chapter can be traced back to
Aristotle's distinction between demonstrative and dialectical argument. The essence of their difference lies in the status of the premises of the argument in question. In a demonstrative argument, its premises are regarded as true. Aristotle called the argument with such premises ``didactic'' and compared them with what a teacher lays down as a starting point of the development of his subject. We would call such premises the axioms of a theory. In a dialectical argument, according to Aristotle, any assumptions can be employed for premises. More than that, in the \textit{reductio ad absurdum} argument one starts with a premise which is aimed to be proven false.

Thus, the focus in the dialectical reasoning is transferred on the argument itself. This involves considering such an argumentation as a binary relation between a set of premises and what can be obtained from them in the course of this argumentation. And the first question raised by this point of view is this: What characteristics of the dialectical argument could qualify it as sound?

\section{Consequence relation}\label{section:consequence-relation}
From the outset, we are faced with a choice between two paths each addressing one of the questions: Should consequence relation state determination between a set of premises (which can be empty) and a set of conclusions (which can also be empty), or should it do it between a set of premises (possibly the empty one) and a single formula?  In this text, we address the first option.
 The framework of this option is known as \textit{single-conclusion}, while that of the second as \textit{multiple-conclusion} or \textit{poly-conclusion} consequence relation.\\

Talking about consequence relation in general, we denote it by ``$\vdash$'' and equip this sign with subscript when referring to a specific consequence relation.

In the context of consequence relation, we use the following notation:
\begin{equation}\label{E:agreement}
	\begin{rcases}
		\begin{array}{c}
			X,Y,\ldots,Z:= X\cup Y\cup\ldots\cup Z;\\
			X,\alpha,\ldots,\beta:=X\cup\{\alpha,\ldots,\beta\},
		\end{array}
	\end{rcases}
\end{equation}
for any sets $X, Y,\ldots, Z$ of $\Lan$-formulas and any $\Lan$-formulas $\alpha,\ldots,\beta$.
\begin{defn}[single-conclusion consequence relation]\label{D:consequnce-relation-single}
	A relation $\vdash\subseteq\mathcal{P}(\Forms)\times\Forms$ is an $\Lan$-\textbf{single-conclusion consequence relation} $($or simply an $\Lan$-\textbf{consequence relation} or just a \textbf{consequence relation} when  consideration occurs in a fixed, perhaps unspecified, formal language$)$ if the following conditions are satisfied$\,:$
	{\em\[
		\begin{array}{cll}
		(\text{a}) &\alpha\in X~\textit{implies}~X\vdash\alpha; &(\textit{reflexivity})\\
		(\text{b}) &X\vdash\alpha~\textit{and}~X\subseteq Y~\textit{imply}~Y\vdash\alpha;
		&(\textit{monotonicity})\\
		(\text{c}) &\textit{if $X\vdash\beta$, for all $\beta\in Y$, and
			$Y,Z\vdash\alpha$, then $X,Z\vdash\alpha$}, &(\textit{transitivity})\\
		\end{array} 
		\]}
	for any sets $X$, $Y$  and $Z$ of $\Lan$-formulas.
\end{defn}

In case $Y=\{\beta\}$, the property  (c) reads:
\[
\begin{array}{cll}
(\text{c}^{\ast}) &\textit{$X\vdash\beta$ and $Z,\beta\vdash\alpha$ imply $X,Z\vdash\alpha$}. &\quad\quad~~~~~~(\textit{cut})
\end{array}
\]

A consequence relation $\vdash$ is called \textit{\textbf{nontrivial}} if
\[
\vdash~\subset\,\mathcal{P}(\Forms)\times\Forms;
\]
otherwise it is \textit{\textbf{trivial}}. 

A consequence relation $\vdash$ is called \textit{\textbf{finitarity}} if
\[
\textit{$X\vdash\alpha$ implies that there is $Y\Subset X$ such that $Y\vdash\alpha$},
\tag{\textit{finitariness}}
\]
where here and on $Y\Subset X$ means that $Y$ is a finite (maybe empty) subset of $X$.

The last notion can be generalized as follows. Let $\kappa$ be an infinite cardinal. A consequence relation $\vdash$ is called $\kappa$-\textit{\textbf{compact}} if
\[
\textit{$X\vdash\alpha$ implies that there is $Y\subseteq X$ such that $Y\vdash\alpha$
	and $\card{Y}<\kappa$}.
\tag{$\kappa$-\textit{compactness}}
\]
Thus finitariness is simply $\aleph_{0}$-compactness.

A consequence relation $\vdash$ is \textit{\textbf{structural}} if
{\em\[
	\textit{$X\vdash\alpha$ implies $\sigma(X)\vdash\sigma(\alpha)$}, \tag{\textit{structurality}}\\ 
	\]}
\text{for any set $X\cup\{\alpha\}\subseteq\Forms$ and any $\Lan$-substitution $\sigma$}.

We will be using the term \textit{\textbf{sequent}} for a statement ``$X\vdash\alpha$''.

\begin{quote}\textit{Note on monotonicity of $\vdash$}.
	The property (b) of Definition~\ref{D:consequnce-relation-single} characterizes the category of \textit{monotone consequence relations} as opposed to that of \textit{non-monotone consequence relations} when the property (b) does not hold. In this book we consider only the first category of consequence relations. (\textit{End of note}.)
\end{quote}

Other, specific, properties in addition to (a)--(c) have been considered in literature. For instance, for the language $\Lan_{B}$ and its expansions the following property is typical for many consequence relations:
\[
X,\alpha\vdash\bot~\textit{implies}~X\vdash\neg\alpha. \tag{\textit{reductio ad absurdum}}
\]

We will be using indices to distinguish in the same formal language either two different consequence relations or two equal consequence relation that are defined differently. The following three propositions will be used in the sequel.
\begin{prop}\label{P:con-relation-intersection}
	Let $\lbrace\vdash_i\rbrace_{i\in I}$ be a family of consequence relations in a language $\Lan$. Then the relation $\vdash$ defined by $\bigcap_{i\in I}\lbrace\vdash_i\rbrace$ is also a consequence relation. Moreover, if each $\vdash_i$ is structural, so is $\vdash$. In addition, if $I$ is finite and each $\vdash_i$ is finitary, then $\vdash$ is also finitary.
\end{prop}
\noindent\textit{Proof}~is left to the reader. (See Exercise~\ref{section:consequence-relation}.\ref{EX:con-relation-intersection}.)\\

Given a consequence relation $\vdash$ and a set $X_0$ of $\Lan$-formulas. We define a relation as follows:
\[
X\vdash_{X_0}\alpha\stackrel{\text{df}}{\Longleftrightarrow}X, X_0\vdash\alpha.
\]
\begin{prop}\label{P:relative-consequence}
	For any consequence relation $\vdash$ and arbitrary fixed set $X_0$, $\vdash_{X_0}$ is a consequence relation such that
	\[
	\emptyset\vdash_{X_0}\alpha~\Longleftrightarrow~ X_0\vdash\alpha.
	\]
\end{prop}
\noindent\textit{Proof}~is left to the reader. (Exercise~\ref{section:consequence-relation}.\ref{EX:relative-consequence})\\

For any consequence relation $\vdash$, we define:
\begin{equation}\label{E:nonempty-consequence}
	X\vdash^{\circ}\alpha~\stackrel{\text{df}}{\Longleftrightarrow}~X\vdash\alpha~~\text{and}~~X\neq\emptyset.
\end{equation}

\begin{prop}\label{P:nonempty-consequence}
	Let $\vdash$ be a consequence relation. Then $\vdash^{\circ}$ is also a consequence relation. In addition, if $\vdash$ is structural, $\vdash^{\circ}$ is also structural; further, if $\vdash$ is $\kappa$-compact, so is $\vdash^{\circ}$.
\end{prop}
\noindent\textit{Proof}~is left to the reader. (Exercise~\ref{section:consequence-relation}.\ref{EX:nonempty-consequence})\\

As is seen from the Proposition~\ref{P:con-relation-intersection}, Definition~\ref{D:consequnce-relation-single} is essentially extensional, because it does not show how the transition from premises to conclusion can be carried out.
However, before addressing this issue, we will discuss how the concept of single-conclusion consequence relation can be represented purely in terms of sets of $\Lan$-formulas.

\paragraph{Exercises~\ref{section:consequence-relation}}
\begin{enumerate}
	\item Consider the properties (a), (b) and (c) of Definition~\ref{D:consequnce-relation-single}. Show that (a) and (c) imply (b).
	\item\label{EX:con-relation-intersection} Prove Proposition~\ref{P:con-relation-intersection}.
	\item \label{EX:relative-consequence}Prove Proposition~\ref{P:relative-consequence}.
	\item \label{EX:nonempty-consequence}Prove Proposition~\ref{P:nonempty-consequence}.
\end{enumerate}

\section{Consequence operator}\label{section:consequence-operator}

Let $\vdash$ be a consequence relation for $\Lan$-formulas. Then, we define:
\begin{equation}\label{E:consequence-definition}
	\Con{X}:=\set{\alpha}{X\vdash\alpha},
\end{equation}
for any set $X$ of $\Lan$-formulas. We observe the following connection: \begin{equation}\label{E:consequence-interconnection}
	X\vdash\alpha~\textit{if and only if}~\alpha\in\Con{X}.
\end{equation}

The last equivalence induces the following definition.
\begin{defn}\label{D:consequence-operator}
	A map {\em$\textbf{Cn}:\mathcal{P}(\Forms)\longrightarrow\mathcal{P}(\Forms)$} is called a \textbf{consequence operator} if it satisfies the following three conditions:
	{\em\[
		\begin{array}{cll}
		(\text{a}^{\dag}) &X\subseteq\Con{X}; &(\textit{reflexivity})\\
		(\text{b}^{\dag}) &X\subseteq Y~\textit{implies}~\Con{X}\subseteq\Con{Y}; &(\textit{monotonicity})\\
		(\text{c}^{\dag}) &\Con{\Con{X}}\subseteq\Con{X}, &(\textit{closedness})\\
		\end{array}
		\]}
	for any sets $X$ and $Y$ of $\Lan$-formulas. 
\end{defn}

A consequence operator \textbf{Cn} is called \textit{\textbf{finitary}} (or \textit{\textbf{compact}}) if
\[
\begin{array}{cl}
(\text{d}^{\dag}) &\Con{X}\subseteq\bigcup\set{\Con{Y}}{Y\Subset X}. \tag{\textit{finitariness}~\text{or}~\textit{compactness}}
\end{array}
\]

As in case of consequence relation, we generalize the last notion for any infinite cardinal $\kappa$: an consequence operator \textbf{Cn} is $\kappa$-\textit{\textbf{compact}} if
\[
\Con{X}\subseteq\bigcup\set{\Con{Y}}{Y\subseteq X~\textit{and}~\card{Y}<\kappa}. \tag{$\kappa$-\textit{compactness}}
\]

And a consequence operator \textbf{Cn} is called \textit{\textbf{structural}}, if for any set $X$ and  any substitution $\sigma$,
\[
\sigma(\Con{X})\subseteq\Con{\sigma(X)}. \tag{\textit{structurality}}
\]
\begin{quote}\textit{Note on monotonicity of} $\textbf{Cn}$.
	A consequence operator as defined in Definition~\ref{D:consequence-operator} is called \textit{monotone} as opposed to \textit{non-monotone} consequence operators. (\textit{End of note}.)
\end{quote}

A set $X$ of formulas is called \textit{\textbf{closed}}, or is a \textit{\textbf{theory}}, (w.r.t. $\textbf{Cn}$) if $X=\Con{X}$. This, in view of $(\text{a}^{\dag})$ and $(\text{c}^{\dag})$, any set $\Con{X}$ is a theory; we call it the \textit{\textbf{theory generated by}} $X$ (w.r.t. \textbf{Cn}). 

We observe the following.
\begin{prop}\label{P:intersection-con}
	Given a consequence operator {\em\textbf{Cn}}, any set {\em$\bigcap_{i\in I}\lbrace\Con{X_i}\rbrace$} is closed.
\end{prop}
\noindent\textit{Proof}~is left to the reader. (See Exercise~\ref{section:consequence-operator}.\ref{EX:intersection-con}.) Compare this proposition with Proposition~\ref{P:con-relation-intersection}.\\

The next proposition establishes connection between the notion of consequence relation and that of consequence operator.
\begin{prop}\label{P:con-operation=con-relation}
	Let $\vdash$ be a consequence relation. The operator {\em\textbf{Cn}} defined by~\eqref{E:consequence-definition} is a consequence operator. Conversely, for a given consequence operator {\em\textbf{Cn}}, the corresponding relation defined by~\eqref{E:consequence-interconnection} is a consequence relation. Moreover, {\em\textbf{Cn}} is finitary $($$\kappa$-compact$)$ if, and only if, so is $\vdash$.
	In addition, {\em\textbf{Cn}} is structural if, and only if, so is $\vdash$.
\end{prop}
\noindent\textit{Proof}~can be obtained by routine check and is left to the reader. (See Exercise~\ref{section:consequence-operator}.\ref{EX:con-operation=con-relation}.)
\begin{quote}
	{\em Proposition~\ref{P:con-operation=con-relation} allows us to use the notion of consequence relation and that of consequence operator interchangeably.}	
\end{quote}

The next observation confirms the last remark; it just mirrors Proposition~\ref{P:con-relation-intersection} 
\begin{prop}\label{P:con-operator-intersection}
	Let {\em$\lbrace\textbf{Cn}_i\rbrace_{i\in I}$} be a family of consequence operators in a language $\Lan$. Then the operator {\em$\textbf{Cn}$} defined by 
	the equality
	{\em\[
		\textbf{Cn}(X):=\bigcap_{i\in I}\textbf{Cn}_{i}(X)
		\]}is also a consequence operator. Moreover, if each {\em$\textbf{Cn}_i$} is structural, so is {\em $\textbf{Cn}$}. In addition, if $I$ is finite and each {\em$\textbf{Cn}_i$} is finitary, then {\em$\textbf{Cn}$} is also finitary.
\end{prop}
\noindent\textit{Proof}~is left to the reader. (Exercise~\ref{section:consequence-operator}.\ref{EX:con-operator-intersection})\\

Besides the properties $(\text{a}^{\dag})$--$(\text{c}^{\dag})$, finitariness and $\kappa$-compactness, other properties of a map $\textbf{Cn}:\mathcal{P}(\Forms)\longrightarrow\mathcal{P}(\Forms)$ related to logical consequence have been considered in literature. These are some of them.
\[
\begin{array}{cl}
(\text{e}^{\dag}) &X\subseteq \Con{Y}~\textit{implies}~\Con{X}\subseteq\Con{Y}; \quad(\textit{comulative transitivity})\\
(\text{f}^{\dag}) &X\subseteq \Con{Y}~\textit{implies}~\Con{X\cup Y}\subseteq\Con{Y}; \quad(\textit{strong comulative transitivity})\\
(\text{g}^{\dag}) &X\subseteq Y\subseteq\Con{X}~\textit{implies}~\Con{Y}\subseteq\Con{X}; 
\quad(\textit{weak comulative transitivity})\\
(\text{h}^{\dag}) &\bigcup\set{\Con{Y}}{Y\Subset X}\subseteq\Con{X};
\quad(\textit{finitary inclusion})\\
(\text{i}^{\dag}) &\textit{if}~X\neq\emptyset,~\textit{then}~\Con{X}\subseteq
\bigcup\set{\Con{Y}}{Y\neq\emptyset~\mbox{and}~Y\Subset X};
\quad(\textit{strong finitariness})\\
(\text{j}^{\dag}) &\textit{if {\em$\alpha\notin\Con{X}$}, then there is
	a maximal $X^{\ast}$ such that $X\subseteq X^{\ast}$}\\
&\textit{and}~\alpha\notin\Con{X^{\ast}}.
\quad(\textit{maximalizability})
\end{array}
\]
\begin{prop}\label{P:con-connections}
	Let {\em\textbf{Cn}} be a map from $\mathcal{P}(\Forms)$ into $\mathcal{P}(\Forms)$, Then the following implications hold$\,:$
	{\em\[
		\begin{array}{rl}
		i) &(\text{a}^{\dag})~\textit{and}~(\text{g}^{\dag})~\textit{imply}~(\text{c}^{\dag});
		\\
		ii) &(\text{b}^{\dag})~\textit{and}~(\text{c}^{\dag})~\textit{imply}~(\text{e}^{\dag});
		\\
		iii) &(\text{e}^{\dag})~\textit{implies}~(\text{g}^{\dag});
		\\
		iv) &(\text{f}^{\dag})~\textit{implies}~(\text{g}^{\dag});
		\\
		v) &(\text{a}^{\dag})~\textit{and}~(\text{g}^{\dag})~\textit{imply}~(\text{f}^{\dag});
		\\
		vi) &(\text{a}^{\dag})~\textit{and}~(\text{f}^{\dag})~\textit{imply}~(\text{c}^{\dag});
		\\
		vii) &(\text{a}^{\dag})~\textit{and}~(\text{b}^{\dag})~\textit{and}~(\text{c}^{\dag})~
		\textit{imply}~(\text{f}^{\dag});
		\\
		viii) &(\text{i}^{\dag})~\textit{implies}~(\text{d}^{\dag});
		\\
		ix) &(\text{b}^{\dag})~\textit{and}~(\text{d}^{\dag})~\textit{imply}~(\text{i}^{\dag});
		\\
		x) &(\text{b}^{\dag})~\textit{implies}~(\text{h}^{\dag});
		\\
		xi) &(\text{b}^{\dag})~\textit{and}~(\text{d}^{\dag})~\textit{imply}~(\text{j}^{\dag});
		\\
		xii) &(\text{d}^{\dag})~\textit{and}~(\text{h}^{\dag})~\textit{imply}~(\text{b}^{\dag}).
		\\
		\end{array}	
		\]}
\end{prop}
\begin{proof}
	We prove $(vi)$, $(xi)$ and $(xii)$, in this order, and leave the rest to the reader. (See Exercise~\ref{section:consequence-operator}.\ref{EX:con-connections}.)
	
	Suppose $X\subseteq\Con{Y}$, Then, by virtue of $(\text{a}^\dag)$, $X\cup Y\subseteq\Con{Y}$. According to $(\text{b}^\dag)$ and $(\text{c}^\dag)$, we have: 
	$\Con{X\cup Y}\subseteq\Con{\Con{Y}}\subseteq\Con{Y}$.
	
	To prove $(xi)$, we assume that $\alpha\notin\Con{X}$, for some set $X$.
	
	Let us denote
	\[
	C_{\alpha}:=\set{Y}{X\subseteq Y~\text{and}~\alpha\notin\Con{Y}}.
	\]
	We aim to show that $C_\alpha$ contains a maximal set.
	
	First, we observe that $C_\alpha$ is nonempty, for $X\in C_\alpha$. Regarding $C_\alpha$ as a partially ordered set with respect to $\subseteq$, we consider a chain $\lbrace Y_i\rbrace_{i\in I}$ in $C_\alpha$. That is, either $Y_i\subseteq Y_j$ or $Y_i\subseteq Y_j$, and $\alpha\notin Y_i$, for any $i\in I$.  Next we denote
	\[
	Y_0:=\bigcup_{i\in I}\lbrace Y_i\rbrace.
	\]
	We aim to show that $\alpha\notin\Con{Y_0}$. For contradiction, assume that $\alpha\in\Con{Y_0}$. Then, by $(\text{d}^{\dag})$, there is a set $Y^{\prime}\Subset Y_0$ such that $\alpha\in\Con{Y^{\prime}}$. We note that $Y^{\prime}\neq\emptyset$, for otherwise, by $(\text{b}^{\dag})$, we would have that $\alpha\in\Con{X}$. Since $Y^\prime$ is finite, there is ${i_0}\in I$ such that $Y^\prime\subseteq Y_{i_0}$. Then, by virtue of $(\text{b}^{\dag})$, $\alpha\in\Con{Y_{i_0}}$. We have obtained a contradiction. Thus $C_\alpha$ satisfies the conditions of Zorn's lemma, according to which $C_\alpha$ contains a maximal set.
	
	Finally, to prove ($xii$), we assume that $X\subseteq Y$ and $\alpha\in\Con{X}$. In virtue of $(\text{d}^\dag)$, there is a finite $Y^{\prime}\subseteq X$ such that $\alpha\in\Con{Y^\prime}$. We note that, by premise, $Y^{\prime}\Subset Y$. Then, by $(\text{h}^\dag)$, $\alpha\in\Con{Y}$.
\end{proof}

\paragraph{Exercises~\ref{section:consequence-operator}}
\begin{enumerate}
	\item \label{EX:intersection-con}Prove Proposition~\ref{P:intersection-con}.
	\item \label{EX:con-operation=con-relation}Prove Proposition~\ref{P:con-operation=con-relation}.
	\item Prove that a consequence operator \textbf{Cn} is structural if, and only if, $\Con{\sigma(\Con{X})}=\Con{\sigma(X)}$, for any set $X$ and any substitution $\sigma$.
	\item  Let \textbf{Cn} be a structural operator. Prove that a set $\sigma(\Con{X})$ is closed (or is a theory) if, and only if, 
	$\sigma(\Con{X})=\Con{\sigma(X)}$, for any set $X$ and any substitution $\sigma$.
	\item \label{EX:con-operator-intersection}Prove Proposition~\ref{P:con-operator-intersection}.
	\item \label{EX:con-connections}Complete the proof of Proposition~\ref{P:con-connections}.
\end{enumerate}
\section{Defining consequence relation}\label{section:consequence-defining}
In order to distinguish consequence relations or consequence operators we will be using the symbols $\vdash$ and \textbf{Cn} with subscripts.  Quite often, we employ a letter  `$\mathcal{S}$', calling a consequence relation $\vdash_{\mathcal{S}}$ and the corresponding consequence operator $\textbf{Cn}_{\mathcal{S}}$ an \textit{\textbf{abstract logic}} $\mathcal{S}$.
Thus, given an abstract logic $\mathcal{S}$, for any sets $X$ of $\Lan$-formulas and any $\Lan$-formula $\alpha$, we have:
\begin{equation}\label{E:Cn-S-connection}
	X\vdash_{\mathcal{S}}\alpha\Longleftrightarrow\alpha\in\textbf{Cn}_{\mathcal{S}}(X).
\end{equation}

Further, we denote 
\[
\bm{T}_{\mathcal{S}}:=\textbf{Cn}_{\mathcal{S}}(\emptyset)
\]
and call the elements of the latter set \textbf{\textit{theorems}}, or \textit{\textbf{theses}}, \textit{\textbf{of}} $\mathcal{S}$, or simply $\mathcal{S}$-\textit{\textbf{theorems}}, or $\mathcal{S}$-\textit{\textbf{theses}}.
Any set $\textbf{Cn}_{\mathcal{S}}(X)$ is call the $\mathcal{S}$-\textit{\textbf{theory generated by a set}} $X$.  Thus $\bm{T}_{\mathcal{S}}$ is the least $\mathcal{S}$-theory with respect to $\subseteq$. 

A set $X$, as well as the theory $\textbf{Cn}_{\mathcal{S}}(X)$, is called \textit{\textbf{inconsistent}} (relative to $\mathcal{S}$) if
$\textbf{Cn}_{\mathcal{S}}(X)=\Forms$; otherwise both are \textit{\textbf{consistent}}.
An abstract logic $\mathcal{S}$ is \textit{\textbf{trivial}} if $\vdash_{\mathcal{S}}$ is trivial, and otherwise is \textit{\textbf{nontrivial}}. We observe that any structural $\mathcal{S}$ is trivial if, and only if, for any arbitrary variable $p$,
$p\in\ConS{\emptyset}$. (Exercise~\ref{section:consequence-defining}.\ref{EX:trivial-1})
We note that for any nontrivial abstract logic $\mathcal{S}$, $\bm{T}_{\mathcal{S}}\subset\Forms$. (Exercise~\ref{section:consequence-defining}.\ref{EX:trivial-2})
) We also note that a theory is not required to be closed under substitution. A bit more we can say about substitution in the structural abstract logics: every theory of a structural abstract logic is closed under substitution if, and only if, this theory is generated by a set itself closed under substitution. (Exercise~\ref{section:consequence-defining}.\ref{EX:theory})

We denote
\[
\theory:=\set{X}{\ConS{X}=X}\tag{\textit{the set of $\mathcal{S}$-theories}}
\]
and
\[
\theoryC:=\set{X}{\ConS{X}=X\subset\Forms}\tag{\textit{the set of consistent $\mathcal{S}$-theories}}
\]

\subsection{Consequence operator (relation) via a closure system} 
A set $\mathcal{A}\subseteq\mathcal{P}(\Forms)$ which contains $\Forms$ and is closed under any nonempty intersections in $\mathcal{P}(\Forms)$ is called a \textit{\textbf{closure system over}} $\Forms$. That is, for any closure system $\mathcal{A}$, $\Forms\in\mathcal{A}$ and for any $I\neq\emptyset$,
\[
\lbrace X_i\rbrace_{i\in I}\subseteq\mathcal{A}\Longrightarrow\bigcap_{i\in I}
\lbrace X_i\rbrace\in\mathcal{A}.
\]

All operators in this subsection are considered over a fixed language $\Lan$.
\begin{prop}\label{P:closure-system}
	Let $\mathcal{A}$ be a closure system. Then the operator
	{\em\[
		\textbf{Cn}_{\mathcal{A}}(X):=\bigcap\set{Y}{X\subseteq Y~\textit{and}\,~Y\in\mathcal{A}}
		\]}is a consequence operator and {\em$\mathcal{A}=\set{X}{X=\textbf{Cn}_{\mathcal{A}}(X)}$}. Conversely, if {\em\textbf{Cn}} is a consequence operator and {\em$\mathcal{A}_{\textbf{Cn}}$} is the family of its closed sets, then
	{\em$\textbf{Cn}=\textbf{Cn}_{\mathcal{A}_{\textbf{Cn}}}$}.
\end{prop}
\begin{proof}
	Assume that $\mathcal{A}$ is a closure system. It is easy to check that the operator $\textbf{Cn}_{\mathcal{A}}$ satisfies the properties $(\text{a}^{\dag})$ and $(\text{b}^{\dag})$ of Definition~\ref{D:consequence-operator}. We prove that it also satisfies $(\text{c}^{\dag})$. For this, given a set $X$, we denote
	\[
	X_0:=\textbf{Cn}_{\mathcal{A}}(X)
	\]
	and suppose that $\alpha\in \textbf{Cn}_{\mathcal{A}}(X_0)$. Aiming to show that $\alpha\in X_0$, we note that
	\[
	\alpha\in X_0\Longleftrightarrow\forall Y\in\mathcal{A}.~X\subseteq Y\Rightarrow
	\alpha\in Y.
	\]
	Thus, we assume that $X\subseteq Y$, for an arbitrary $Y\in\mathcal{A}$. In virtue of 
	$(\text{b}^\dag)$, $X_0\subseteq\textbf{Cn}_{\mathcal{A}}(Y)$ and hence $\Con{X_0}\subseteq Y$, for $Y\in\mathcal{A}$. This implies that $\alpha\in Y$ and, using the equivalence above, we obtain that $\alpha\in X_0$.
	
	Finally, one can easily observe that
	\[
	X\in\mathcal{A}\Longleftrightarrow\textbf{Cn}_{\mathcal{A}}(X)=X.
	\]
	
	Conversely, suppose \textbf{Cn} is a consequence operator and
	\[
	\mathcal{A}_{\textbf{Cn}}:=\set{X}{\textbf{Cn}(X)=X}.
	\]
	Assume that $\lbrace X_i\rbrace_{i\in I}\subseteq\mathcal{A}_{\textbf{Cn}}$ and denote
	\begin{equation}\label{E:conjunction}
		X_0:=\bigcap_{i\in I}\lbrace X_i\rbrace.
	\end{equation}
	Since $X_0\subseteq X_i$, for each $i\in I$, $\Con{X_0}\subseteq\Con{X_i}=X_i$, also for each $i\in I$. Hence $\Con{X_0}\subseteq X_0$, that is $X_0\in\mathcal{A}_{\textbf{Cn}}$.
	
	Finally, we show that for any formula $\alpha$ and any set $X$ of formulas,
	\[
	\alpha\in\textbf{Cn}(X)\Longleftrightarrow\alpha\in
	\textbf{Cn}_{\mathcal{A}_{\textbf{Cn}}}(X).
	\]
	
	First, we notice that
	\[
	\alpha\in
	\textbf{Cn}_{\mathcal{A}_{\textbf{Cn}}}(X)\Longleftrightarrow
	\forall Y.~(X\subseteq Y~\text{and}~\Con{Y}=Y)\Rightarrow\alpha\in Y. \tag{$\ast$}
	\]
	Then, in virtue of $(\ast)$, we have:
	\[
	\begin{array}{rl}
	\alpha\in
	\textbf{Cn}_{\mathcal{A}_{\textbf{Cn}}}(X)\!\!\!\! &\Longrightarrow
	[(X\subseteq\Con{X}~\text{and}~\Con{\Con{X}}=\Con{X})\Rightarrow\alpha\in\Con{X}]\\
	&\Longrightarrow\alpha\in\Con{X}.
	\end{array}
	\]
	
	Now we suppose that $\alpha\in\Con{X}$, $X\subseteq Y$ and $\Con{Y}=Y$. From the second premise, we derive that $\Con{X}\subseteq\Con{Y}$ and hence $\Con{X}\subseteq
	Y$. Thus $\alpha\in Y$. Since $Y$ is an arbitrary set satisfying the aforementioned conditions, using $(\ast)$, we conclude that $\alpha\in\textbf{Cn}_{\mathcal{A}_{\textbf{Cn}}}(X)$.
\end{proof}

Given an abstract logic $\mathcal{S}$, we remind the notation:
\[
\Sigma_{\mathcal{S}}:=\set{X}{X=\ConS{X}}.
\]

We conclude with the following observation.
\begin{prop}[Brown-Suszko theorem]\label{P:brown-suszko-theorem}
	An abstract logic $\mathcal{S}$ in a language $\Lan$ is structural if, and only if, for any set $X\subseteq\Forms$ and any substitution $\sigma$, the following holds:
	\[
	X\in\Sigma_{\mathcal{S}}\Longrightarrow\sigma^{-1}(X)\in\Sigma_{\mathcal{S}}.
	\]
\end{prop}
\begin{proof}
	Suppose $\mathcal{S}$ is structural. Let $\sigma$ be a substitution and $X\in\Sigma_{\mathcal{S}}$. Assume that $\sigma^{-1}(X)\vdash_{\mathcal{S}}\alpha$.
	Since $\mathcal{S}$ is structural, we have: $\sigma(\sigma^{-1}(X))\vdash_{\mathcal{S}}\sigma(\alpha)$. In virtue of (\ref{E:substitution-inequalities}--$i$), $X\vdash_{\mathcal{S}}\sigma(\alpha)$, that is $\sigma(\alpha)\in X$. This means that $\alpha\in\sigma^{-1}(X)$.
	
	To prove the converse, we assume that the conditional above holds. Then we, with help of~(\ref{E:substitution-inequalities}--$ii$), obtain that $X\subseteq\sigma^{-1}(\sigma(X))\subseteq\sigma^{-1}(\ConS{\sigma(X)})$. This implies that $\ConS{X}\subseteq\ConS{\sigma^{-1}(\ConS{\sigma(X)})}$. Applying the assumption, we get that $\ConS{X}\subseteq\sigma^{-1}(\ConS{\sigma(X)})$. Thus, if $X\vdash_{\mathcal{S}}\alpha$, that is $\alpha\in\ConS{X}$, then $\alpha\in\sigma^{-1}(\ConS{\sigma(X)})$, that is $\sigma(X)\vdash_{\mathcal{S}}\sigma(\alpha)$.
\end{proof}

\subsection{Consequence relation via logical matrices}\label{section:con-via-matrices}
Given a (logical) matrix $\mat{M}=\langle\alg{A},D\rangle$, we define a relation
$\models_{\textbf{M}}$ on $\mathcal{P}(\Forms)\times\Forms$ as follows:
\[
X\models_{\textbf{M}}\alpha\stackrel{\text{df}}{\Longleftrightarrow}(v[X]\subseteq D
\Rightarrow v[\alpha]\in D,~\text{for any valuation $v$ in \alg{A}}),
\]
where
\[
v[X]:=\set{v[\beta]}{\beta\in X}.
\]

Let $\mathcal{M}$ be a nonempty class of matrices. We define the relation of \textit{\textbf{matrix consequence}} (w.r.t. a class $\mathcal{M}$) as follows:
\begin{equation}\label{E:matrix-consequence}
	X\models_{\mathcal{M}}\alpha\stackrel{\text{df}}{\Longleftrightarrow}
	(X\models_{\textbf{M}}\alpha,~\text{for all $\mat{M}\in\mathcal{M}$}).
\end{equation}
Accordingly, the relation $\models_{\textbf{M}}$ is called a \textit{\textbf{single-matrix consequence}}, or $\mat{M}$-\textit{\textbf{consequence}} for short.
If a class $\mathcal{M}$ is known, we call a matrix consequence associated with a class $\mathcal{M}$ an $\mathcal{M}$-\textit{\textbf{consequence}}. Given a nonempty set of matrices $\mathcal{M}$, we denote the $\mathcal{M}$-consequence by $\mathcal{S}_{\mathcal{M}}$ and say that an abstract logic $\mathcal{S}$ is \textit{\textbf{determined}} by $\mathcal{M}$ if $\mathcal{S}=\mathcal{S}_{\mathcal{M}}$.

The justification for the terminology `$\mathcal{M}$-consequence' and `\mat{M}-consequence' is the following. 

\begin{prop}\label{P:matrix-con-is-con-relation}
	Any matrix consequence is a structural consequence relation.
\end{prop}
\begin{proof}
	Let $\mathcal{M}=\lbrace\textbf{M}_i\rbrace_{i\in I}$ be a nonempty set of matrices. We have to check that the conditions (a)--(c) of Definition~\ref{D:consequnce-relation-single} are satisfied for the relation $\models_{\mathcal{M}}$. In addition, we have to show that this consequence relation is structural. We leave this routine check to the reader. (Exercise~\ref{section:consequence-defining}.\ref{EX:matrix-con-is-con-relation})
\end{proof}

Let $\mathcal{S}$ be an abstract logic and $\mat{M}=\langle\alg{A},D\rangle$ be a matrix. Accordingly, we have the consequence relation $\vdash_{\mathcal{S}}$ and a single-matrix consequence $\models_{\textbf{M}}$. If $\vdash_{\mathcal{S}}\subseteq\models_{\textbf{M}}$, we call the filter $D$ an $\mathcal{S}$-\textit{\textbf{filter}}, in which case the matrix \mat{M} is called an $\mathcal{S}$-\textit{\textbf{matrix}} (or an $\mathcal{S}$-\textit{\textbf{model}}).

\begin{defn}[Lindenbaum matrix]\label{D:lindebaum-matrix}
	Let $\mathcal{S}$ be a structural abstract logic and $D_{\mathcal{S}}$ be an $\mathcal{S}$-theory. The matrix $\langle\FormAl,D_{\mathcal{S}}\rangle$ is called a Lindenbaum matrix $($relative to $\mathcal{S}$$)$. For an arbitrary $X\subseteq\Forms$, we denote:
	{\em\[
		\Lin_{\,\mathcal{S}}[X]:=\langle\FormAl,\ConS{X}\rangle
		\]}
	and
	{\em\[
		\Lin_{\,\mathcal{S}}:=\langle\FormAl,\ConS{\emptyset}\rangle
		\]}
\end{defn}

The reader is advised to remember that (Proposition~\ref{P:substitution-as-homomorphism})
\begin{center}
	\textit{any $\Lan$-substitution is a valuation in $\mathfrak{F}_{\mathcal{L}}$ and vice versa.}	
\end{center}

\begin{prop}[Lindenbaum's theorem]\label{P:lindenbaum-theorem}
	Let $\mathcal{S}$ be a structural abstract logic. Then {\em$\Lin_{\,\mathcal{S}}$} is an $\mathcal{S}$-model. Also, if a set $X\subseteq\Forms$ is closed under arbitrary $\Lan$-substitution, then for any $\Lan$-formula $\alpha$,
	{\em\[
		X\vdash_{\mathcal{S}}\alpha~\Longleftrightarrow~\models_{\textsf{\textbf{Lin}}_{\,\mathcal{S}}[X]}\alpha;
		\]}
	in particular,
	{\em\[
		\alpha\in\bm{T}_{\mathcal{S}}~\Longleftrightarrow~\models_{\textsf{\textbf{Lin}}_{\,\mathcal{S}}}\alpha.
		\]}
\end{prop}
\begin{proof}
	Let us take any set $X\cup\lbrace\alpha\rbrace\subseteq\Forms$ such that $X\vdash_{\mathcal{S}}\alpha$. Then, for an arbitrary $\Lan$-substitution $\sigma$, $\sigma(X)\vdash_{\mathcal{S}}\sigma(\alpha)$. Therefore, if $\sigma(X)\subseteq\ConS{\emptyset}$, then $\sigma(\alpha)\in\ConS{\emptyset}$. Thus
	$\Lin_{\,\mathcal{S}}$ is an $\mathcal{S}$-model. 
	
	Now, if, in addition, $\sigma(X)\subseteq X$, then $\sigma(\alpha)\in\ConS{X}$. Since $\sigma$ is an arbitrary $\Lan$-substitution, we have that
	$\models_{\textsf{\textbf{Lin}}_{\,\mathcal{S}}[X]}\alpha$.
	
	Conversely, if $\models_{\textsf{\textbf{Lin}}_{\,\mathcal{S}}[X]}\alpha$, then, in particular, $\iota(\alpha)\in\ConS{X}$, that is $X\vdash_{\mathcal{S}}\alpha$.
	
	The second equivalence is a specification of the first.
\end{proof}

We say that a matrix is \textit{\textbf{weakly adequate}} for an abstract logic $\mathcal{S}$ if this matrix validates all $\mathcal{S}$-theorems and only them. Thus, by virtue of the last equivalence of Proposition~\ref{P:lindenbaum-theorem},
we obtain the following.
\begin{cor}\label{C:lindenbaum-matrix}
	Let $\mathcal{S}$ be structural abstract logic. Then any Lindenbaum matrix relative to $\mathcal{S}$
	is an $\mathcal{S}$-model.  Moreover, $\Lin_{\,\mathcal{S}}$ is weakly adequate for $\mathcal{S}$.
\end{cor}
\noindent{\em Proof}~is left to the reader. (Exercise~\ref{section:consequence-defining}.\ref{EX:lindenbaum-matrix})\\

Proposition~\ref{P:lindenbaum-theorem} can be generalized as follows.
\begin{prop}
	Given a structural abstract logic $\mathcal{S}$, let $\mathcal{M}_{\mathcal{S}}$ be the set of all Lindenbaum matrices relative to $\mathcal{S}$. Then $\vdash_{\mathcal{S}}\,=\,\models_{\mathcal{M}_{\mathcal{S}}}$; that is for any set $X$ of formulas and any formula $\alpha$,
	\[
	X\vdash_{\mathcal{S}}\alpha\Longleftrightarrow X\models_{\mathcal{M}_{\mathcal{S}}}
	\alpha.
	\]
\end{prop}
\begin{proof}
	First we suppose that $X\vdash_{\mathcal{S}}\alpha$. Let us take any $\langle\FormAl,D_{\mathcal{S}}\rangle$. Assume that for some substitution $\sigma$, $\sigma(X)\subseteq D_{\mathcal{S}}$.
	Since $\mathcal{S}$ is structural, $\sigma(X)\vdash_{\mathcal{S}}\sigma(\alpha)$.
	This implies that $\sigma(\alpha)\in\textbf{Cn}_{\mathcal{S}}(\sigma(X))$ and hence
	$\sigma(\alpha)\in D_{\mathcal{S}}$. 
	
	Conversely, assume that $X\not\vdash_{\mathcal{S}}\alpha$. This means (see~\eqref{E:Cn-S-connection}) that $\alpha\notin\textbf{Cn}_{\mathcal{S}}(X)$.
	Let us denote $\textbf{M}:=\langle\FormAl,\textbf{Cn}_{\mathcal{S}}(X)\rangle$ and take the identity substitution $\iota$. It is obvious that $\iota(X)\subseteq
	\textbf{Cn}_{\mathcal{S}}(X)$ but $\iota(\alpha)\notin\textbf{Cn}_{\mathcal{S}}(X)$; that is $X\not\models_{\textbf{M}}\alpha$.
\end{proof}

The last proposition inspires us for the following (complex) definition.
\begin{defn}
	A $($nonempty$)$ family $\mathcal{B}=\lbrace\langle\alg{A}_{i},D_i\rangle\rbrace_{i\in I}$ of matrices of type $\Lan$ is called a \textbf{bundle} if for any $i,j\in I$, the algebras $\alg{A}_i$ and $\alg{A}_j$ are isomorphic. A pair $\langle\alg{A},\lbrace D_i\rbrace_{i\in I}\rangle$, where each $D_i$ is a logical filter in $($or of$)$ $\alg{A}$, is called an \textbf{atlas}. By definition, the \textbf{matrix consequence relative to an atlas}
	$\langle\alg{A},\lbrace D_i\rbrace_{i\in I}\rangle$ is the same as the matrix consequence related the bundle $\lbrace\langle\alg{A},D_i\rangle\rbrace_{i\in I}$, which in turn is defined in the sense of~{\em \eqref{E:matrix-consequence}}.
	Given an abstract logic $\mathcal{S}$, the bundle  $\lbrace\langle\FormAl,D\rangle\rbrace_{D\in\Sigma_{\mathcal{S}}}$, where $\Sigma_{\mathcal{S}}$ is the set of all $\mathcal{S}$-theories, is called the \textbf{Lindenbaum bundle} relative to $\mathcal{S}$ and the atlas $\mat{Lin}[\Sigma_{\mathcal{S}}]:=\langle\FormAl,\Sigma_{\mathcal{S}}\rangle$, is called the \textbf{Lindenbaum atlas} relative to $\mathcal{S}$.
\end{defn}

We say that a given matrix (bundle or atlas) is \textit{\textbf{adequate}} for an abstract logic $\mathcal{S}$ if $\vdash_{\mathcal{S}}$ and the corresponding matrix relation coincide. 

We denote:
\[
\begin{array}{rcl}
X\models_{\textsf{\textbf{Lin}}[\Sigma_{\mathcal{S}}]}\alpha &\stackrel{\text{df}}{\Longleftrightarrow}
&\textit{for any $\Lan$-substitution $\sigma$ and any $D_{\mathcal{S}}\in\Sigma_{\mathcal{S}}$},\\
&&\textit{$\sigma(X)\subseteq D_{\mathcal{S}}$ implies $\sigma(\alpha)\in D_{\mathcal{S}}$},
\end{array}
\]
for any $X\cup\lbrace\alpha\rbrace\subseteq\Forms$.

Thus we conclude with the following.
\begin{cor}\label{C:Lindenbaum-completeness}
	The Lindenbaum atlas relative to a structural abstract logic is adequate for this logic.
	That is, for any $X\cup\lbrace\alpha\rbrace\subseteq\Forms$,
	\[
	X\vdash_{\mathcal{S}}\alpha~\Longleftrightarrow~X\models_{\textsf{\textbf{Lin}}[\Sigma_{\mathcal{S}}]}\alpha.
	\]
\end{cor}
\noindent\textit{Proof}~is left to the reader. (Exercise~\ref{section:consequence-defining}.\ref{EX:Lindenbaum-completeness})

\begin{cor}\label{C:S-matrix-completeness}
	Let $\mathcal{S}$ be a structural abstract logic and $\mathcal{M}$ be the set of all $\mathcal{S}$-models. Then $\mathcal{S}=\mathcal{S}_{\mathcal{M}}$.
\end{cor}
\noindent\textit{Proof}~is left to the reader. (Exercise~\ref{section:consequence-defining}.\ref{EX:S-matrix-completeness})

\subsection{Consequence relation via inference rules}\label{section:inference-rules}
Since we perceive the sentential formulas as forms of judgments (Section 2.1), our understanding of a rule of inference, at least in one  important case, will be a relation in the class of forms of formulas. As in the first event, where substitution plays a key role in seeing any sentential formula as an abstract representation of infinitely many judgments, it will play a similar role in the concept of inference rule, at least in this important class of inference rules. It is a way, by which we create an abstraction of infinitely many instances of one element that falls under the name of \textit{structural inference rule}. 

An \textit{\textbf{inference rule}} is a set $R\subseteq\mathcal{P}(\Forms)\times\Forms$. An inference rule receives its significance, when it is a subset of some consequence relation, say $\vdash$. We say that a rule $R$ is \textit{\textbf{sound w.r.t.}} $\vdash$ if for any $\langle\Gamma,\alpha\rangle\in R$, $\Gamma\vdash\alpha$.

Let $\mathcal{R}$ be a set of inference rules. Taking into account that $\mathcal{P}(\Forms)\times\Forms$ is the trivial consequence relation (relative to a language $\Lan$), we define
\begin{equation}\label{E:consequence-by-rules}
	\vdash_{[\mathcal{R}]}:=\bigcap\set{\vdash_{\mathcal{S}}}
	{\forall R\in\mathcal{R}.~\text{$R$ is sound w.r.t. $\mathcal{S}$}}.
\end{equation}
(We note that the right-hand set is nonempty.)

It is not difficult to see that $\vdash_{[\mathcal{R}]}$ is a consequence relation. (Exercise~\ref{section:consequence-defining}.\ref{EX:S_R-consequence}) We say that an abstract logic $\mathcal{S}$ is \textit{\textbf{determined by}} a set $\mathcal{R}$ of inference rules if $\vdash_{\mathcal{S}}\,=\,\vdash_{[\mathcal{R}]}$. Given an abstract logic $\mathcal{S}$, one can treat the consequence relation $\vdash_{\mathcal{S}}$ as an inference rule that determines $\mathcal{S}$.\\

Now we turn to structural rules. 

Let us consider a pair $\langle\Gamma,\phi\rangle\in\mathcal{P}(\Mform)\times\Mform$. We call the set
\[
R_{\langle\Gamma,\phi\rangle}:=\set{\langle\bm{\sigma}(\Gamma),\bm{\sigma}(\phi)\rangle}{\bm{\sigma}~\text{is an instantiation of}~\Mvar}
\]
a \textit{\textbf{structural rule}} (in $\Lan$) with the premises $\Gamma$ and conclusion $\phi$. Each pair $\langle\bm{\sigma}(\Gamma),\bm{\sigma}(\phi)\rangle$ is an \textit{\textbf{instance of the rule}} $R_{\langle\Gamma,\phi\rangle}$ (with the premises $\bm{\sigma}(\Gamma)$ and conclusion $\bm{\sigma}(\phi)$). It is convenient (and we follow this practice) to identify the rule $R_{\langle\Gamma,\phi\rangle}$ with its label $\langle\Gamma,\phi\rangle$, using them interchangeably. It is customary to write a structural rule $\langle\Gamma,\phi\rangle$ in the forms $\Gamma\slash\phi$ or
$\frac{\Gamma}{\phi}$. (It is customary to omit the set-builder notation ``$\{\ldots\}$'', when a set of premises is finite, and leave the space blank if $\Gamma=\emptyset$.)

Given an abstract logic $\mathcal{S}$, specifically, a structural rule $\langle\Gamma,\phi\rangle$ is \textit{\textbf{sound w.r.t.}} $\mathcal{S}$ if for any instantiation $\bm{\sigma}$, $\bm{\sigma}(\Gamma)\vdash_{\mathcal{S}}\bm{\sigma}(\phi)$.

A structural rule is called \textit{\textbf{finitary}} (or a \textit{\textbf{modus}} or \textit{\textbf{modus rule}}) if $\Gamma\Subset\Mform$. If an abstract logic $\mathcal{S}$ is determined by a finite set of modus rules, one says that set of modus rules is a \textit{\textbf{calculus}} for $\mathcal{S}$. An abstract logic may be defined by more than one calculus, if any. An example of a modus rule is \textit{modus ponens} which is the rule
$\bm{\alpha},\bm{\alpha}\rightarrow\bm{\beta}\slash\bm{\beta}$.

Not all useful inference rules are structural.

Let $\sigma$ be a fixed $\Lan$-substitution. Then we define
\[
R_{\sigma}:=\set{\langle\lbrace\alpha\rbrace,\sigma(\alpha)\rangle}{\alpha\in\Forms}.
\]
Then the \textit{\textbf{rule of substitution}} is the set
\[
R_{\text{sub}}:=\bigcup\set{R_{\sigma}}{\sigma~\text{is an $\Lan$-substitution}}.
\]
Each $\langle\alpha,\sigma(\alpha)\rangle$ is called an \textit{\textbf{application of the substitution}} $\sigma$ \textit{\textbf{to}} $\alpha$. A customary utterance is: $\sigma(\alpha)$ \textit{is obtained from $\alpha$ by substitution} $\sigma$. 

Although substitution is an important component of any structural rule (see Proposition~\ref{P:metaformula-instantiations}), the rule of substitution itself is not structural, providing that $\Lan$ contains sentential connectives.

Indeed, for contradiction, we assume that 
\[
R_{\text{sub}}=R_{\langle\Gamma,\phi\rangle},
\]
for some $\Gamma\cup\lbrace\phi\rbrace\subseteq\Mform$.
Let us fix an arbitrary $p\in\Var_{\Lan}$. By definition, it is clear that for any $\alpha\in\Forms$, $\langle\lbrace p\rbrace,\alpha\rangle\in R_{\text{sub}}$.
This implies that for any $\alpha\in\Forms$, there is an instantiation $\bm{\xi}_{\alpha}$ such that $\bm{\xi}_{\alpha}(\Gamma)=\lbrace p\rbrace$ and $\bm{\xi}_{\alpha}(\phi)=\alpha$. Thus, we conclude that $\Gamma\cup\lbrace\phi\rbrace\subseteq\Mvar$ and $\Gamma\cap\lbrace\phi\rbrace=\emptyset$; moreover, it must be clear that $\Gamma$ cannot have more than one metavariables, because if it were otherwise, the claim 
$\langle\lbrace p\rbrace,\alpha\rangle\in R_{\langle\Gamma,\phi\rangle}$ would not be true. Thus, $\langle\Gamma,\phi\rangle=\langle\lbrace\gamma\rbrace,\beta\rangle$, for two distinct metavariables $\gamma$ and $\beta$. Then, providing that $\Lan$ contains sentential connectives, there is an instantiation $\bm{\xi}$ such that the degree of $\bm{\xi}(\gamma)$ is greater than the degree of $\bm{\xi}(\beta)$. A contradiction.\\

We will not discuss here non-structural rules, except substitution.\footnote{However, we note that the hyperrules introduced below, strictly speaking, are  also not structural. }

\begin{prop}\label{P:structural-rules-for-structural-logic}
	Let $\mathcal{S}$ be a structural abstract logic in $\Lan$. Then there is a set of structural rules that determines $\mathcal{S}$.
\end{prop}
\begin{proof}
	Since $\card{\Var_{\Lan}}\le\card{\Mvar}$, there is a one-one map $\Var_{\Lan}\longrightarrow\Mvar$ which can, obviously, be extended to $f:\FormAl\longrightarrow\MformAl$. (We employ the same notation for both maps.)
	
	Next, we define:
	\[
	\mathcal{R}:=\set{\langle f(X),f(\alpha)\rangle}{X\vdash_{\mathcal{S}}\alpha}.
	\]
	
	It must be clear that any rule $R\in\mathcal{R}$ is structural. And if $R=\langle f(X),f(\alpha)\rangle$, then there is an instantiation $\bm{\xi_0}$ such that $\bm{\xi_0}(f(X))=X$
	and $\bm{\xi_0}(f(\alpha))=\alpha$. This implies that any rule $R\in\mathcal{R}$ is sound w.r.t. $\mathcal{S}$. Hence $\vdash_{[\mathcal{R}]}\subseteq\vdash_{\mathcal{S}}$.
	
	Now we assume that any rule $R\in\mathcal{R}$ is sound w.r.t. some $\mathcal{S}^{\prime}$. Suppose $X\vdash_{\mathcal{S}}\alpha$. Then 
	$\langle f(X),f(\alpha)\rangle\in\mathcal{R}$. Therefore, for any instantiation $\bm{\eta}$, $\bm{\eta}(f(X))\vdash_{\mathcal{S}^{\prime}}\bm{\eta}(f(\alpha))$.
	In particular, $\bm{\xi_0}(f(X))\vdash_{\mathcal{S}^{\prime}}\bm{\xi_0}(f(\alpha))$.
	Hence $X\vdash_{\mathcal{S}^{\prime}}\alpha$.
\end{proof}

For defining consequence relations, especially structural consequence relations,  sometimes even more powerful machinery is convenient. We mean the use of hyperrules.

An (n+1)-tuple
\[
\Theta:=\langle\langle\Gamma_1,\phi_1\rangle,\ldots,\langle\Gamma_n,\phi_n\rangle,
\langle\Delta,\psi\rangle\rangle,
\]
where $\Gamma_{1}\cup\ldots\cup\Gamma_{n}\cup\Delta\cup\lbrace\phi_1,\ldots\phi_n,\psi
\rbrace\subseteq\Mform$, is called a (\textit{\textbf{structural}}) \textit{\textbf{hyperrule}}. 

A hyperrule
$\Theta$ is \textit{\textbf{sound w.r.t.}} an abstract logic $\mathcal{S}$ if for any $\Lan$-instantiation $\bm{\xi}$ and any set $X\subseteq\Forms$,
\[
[X\cup\,\bm{\xi}(\Gamma_{1})\vdash_{\mathcal{S}}\bm{\xi}(\phi_1),\ldots,X\cup\,\bm{\xi}(\Gamma_{n})\vdash_{\mathcal{S}}\bm{\xi}(\phi_n)]\Longrightarrow
X\cup\,\bm{\xi}(\Delta)\vdash_{\mathcal{S}}\bm{\xi}(\psi).
\]

If a hyperrule $\Theta$ is employed in relation to an abstract logic $\mathcal{S}$, it is customary to write this hyperrule in the following form:
\[
\dfrac{X,\Gamma_{1}\vdash_{\mathcal{S}}\phi_1;\ldots; X,\Gamma_{n}\vdash_{\mathcal{S}}\phi_n}{X,\Delta\vdash_{\mathcal{S}}\psi}.
\]

Next, using  rules and hyperrules (not all of them for each example but selectively), we define several consequence relations in language $\Lan_{A}$ (Section~\ref{section:languages}) and discuss their properties.

\textsf{Inference rules:}
\[
\begin{array}{cllll}
(\text{a}) &i)~\dfrac{\bm{\alpha},\bm{\beta}}{\bm{\alpha}\wedge\bm{\beta}} &ii)~\dfrac{\bm{\alpha}}{\bm{\alpha}\vee\bm{\beta}}
&iii)~\dfrac{\bm{\beta}}{\bm{\alpha}\vee\bm{\beta}} 
&iv)~\dfrac{\bm{\alpha}}{\neg\neg\bm{\alpha}}\\\\
(\text{b}) &i)~\dfrac{\bm{\alpha}\wedge\bm{\beta}}{\bm{\alpha}}
&ii)~\dfrac{\bm{\alpha}\wedge\bm{\beta}}{\bm{\beta}}
&iii)~\dfrac{\bm{\alpha},\bm{\alpha}\rightarrow\bm{\beta}}{\bm{\beta}}
&iv)~\dfrac{\neg\neg\bm{\alpha}}{\bm{\alpha}}.
\end{array}
\]

\textsf{Hyperrules:}
\[
\begin{array}{cccc}
(\text{c}) &i)~\dfrac{X,\bm{\alpha}\vdash\bm{\beta}}{X\vdash\bm{\alpha}\rightarrow\bm{\beta}}
&ii)~\dfrac{X,\bm{\alpha}\vdash\bm{\beta};~X,\bm{\alpha}\vdash\neg\bm{\beta}}{X\vdash\neg\bm{\alpha}}  &iii)~\dfrac{X,\bm{\alpha}\vdash\bm{\gamma};~X,\bm{\beta}\vdash\bm{\gamma}}
{X,\bm{\alpha}\vee\bm{\beta}\vdash\bm{\gamma}},
\end{array}
\]
where $\vdash$ in the hyperrules above denotes a fixed binary relation on
$\mathcal{P}(\Forms)\times\Forms$.\\

First we define a relation $\vdash_1$ on $\mathcal{P}(\Forms)\times\Forms$ and then prove that it is a consequence relation. Then, we do the same for relations $\vdash_2$ and $\vdash_3$. In all three definitions, we use the concept of (\textit{\textbf{formal}}) \textit{\textbf{derivation}} (or \textit{\textbf{formal proof}}). The relations being defined are written as $X\vdash_i\alpha$, where $i=1,2,3$. (For convenience we call $X\vdash_i\alpha$ a \textit{sequent}, even when we may not know yet that $\vdash_i$ is a consequent relation.) If $X\vdash_i\alpha$ is justified according to the clauses below, we say that \textit{$X$ derives $\alpha$}, or \textit{$\alpha$ is derivable from $X$} (all w.r.t. $\vdash_i$). It is for such a justification we use the notion of derivation.

\paragraph{Relation $\vdash_1$} Given a set $X$ of $\Lan_{A}$-formulas and an $\Lan_{A}$-formula $\alpha$, derivation confirming (or establishing or justifying) a sequent $X\vdash_1\alpha$ is a finite (nonempty) list of $\Lan_{A}$-formulas
\begin{equation}\label{E:list}
	\alpha_1,\ldots,\alpha_n
\end{equation}
such that
\[
\begin{array}{cl}
(\text{d}_1) &\alpha_n=\alpha;~\text{and}\\
(\text{d}_2) &\text{each $\alpha_i$ is either in $X$ or  obtained from one or more formulas}\\ 
&\text{preceeding it on the list \eqref{E:list} by one of the rules (a)--(b)}.
\end{array}
\]

It is easy to see that $\vdash_1$ satisfies the properties (a)--(c) of Definition~\ref{D:consequnce-relation-single} and thus is a consequence relation. Moreover, this relation is obviously finitary and structural. One peculiarity of it is that the set of the theses of this relation is empty. Another property of $\vdash_1$ is that
\begin{equation}\label{E:classical-1}
	X\vdash_1\alpha\Longrightarrow X\models_{\textbf{B}_{2}}\alpha,
\end{equation}
where $\models_{\textbf{B}_{2}}$ is a single-matrix consequence w.r.t. matrix $\textbf{B}_{2}$ (Section~\ref{S:two-valued}).  To prove~\eqref{E:classical-1}, one needs to show that each rule of (a)--(b) preserves validity in $\textbf{B}_2$.
We leave this task to the reader. (Exercise~\ref{section:consequence-defining}.\ref{EX:classical-1})

Clearly, the converse of \eqref{E:classical-1} is not true, for if $\alpha$ is a classical tautology, then 
$\emptyset\models_{\textbf{B}_{2}}\alpha$ but $\emptyset\not\vdash_1\alpha$.

\paragraph{Relation $\vdash_{2}$} Our next example, relation $\vdash_2$, is a modification of $\vdash_1$. Namely we define:
\[
X\vdash_2\alpha\stackrel{\text{df}}{\Longleftrightarrow} X,L\textbf{B}_{2}\vdash_1\alpha,
\]
where $L\textbf{B}_{2}$ is the logic of matrix $\textbf{B}_{2}$ (Definition~\ref{D:validity-1}), that is the set of classical tautologies. According to Proposition~\ref{P:relative-consequence}, $\vdash_2$ is a consequence relation whose set of theses is $L\textbf{B}_{2}$. 

In addition, we claim that
\begin{equation}\label{E:classical-2}
	X\vdash_2\alpha\Longrightarrow X\models_{\textbf{B}_{2}}\alpha.
\end{equation}

The proof of \eqref{E:classical-2} follows from the proof of \eqref{E:classical-1}, where we have to add one more (obvious) case --- when $\alpha_i$ in \eqref{E:list} is a classical tautology, that is $\alpha_i\in L\textbf{B}_2$. 

The converse of \eqref{E:classical-2} is also true, that is
\begin{equation}\label{E:classical-2-converse}
	X\models_{\textbf{B}_{2}}\alpha\Longrightarrow X\vdash_2\alpha.
\end{equation}

To prove \eqref{E:classical-2-converse}, we first prove that the consequence relation $\models_{\textbf{B}_{2}}$ is finitary. Indeed, the premise $X\models_{\textbf{B}_{2}}\alpha$ is equivalent to that the set $X\cup\lbrace\neg\alpha\rbrace$ is not \textit{satisfiable} in $\textbf{B}_2$, that is for any valuation $v$ in $\textbf{B}_2$, $v[X\cup\lbrace\neg\alpha\rbrace]\not\subseteq\lbrace\one\rbrace$. 
Then, by compactness theorem,\footnote{Cf.~\cite{barwise1977}, J. Barwise, An introduction to first-order logic,  theorem 4.2. See also~\cite{chang-keisler1990}, corollary 1.2.12.} $X\cup\lbrace\neg\alpha\rbrace$ is not \textit{finitely satisfiable} in $\textbf{B}_2$, that is, there is $X_{0}\Subset X$ such that $X_0\cup\lbrace\neg\alpha\rbrace$ is not satisfiable in $\textbf{B}_2$.
The latter is equivalent to $X_{0}\models_{\textbf{B}_{2}}$.\footnote{We will consider this argument in a more general setting in Section~\ref{section:finitary-matrix-consequence}.}

In the next step, having $X_{0}\models_{\textbf{B}_{2}}\alpha$ with $X_{0}\Subset X$, we show that $X_{0}\vdash_2\alpha$.

If $X_0=\emptyset$, then $\alpha$ is a classical tautology and the conclusion is obvious. Assume that
\[
X_0:=\lbrace\beta_1,\ldots,\beta_n\rbrace.
\]
We denote
\[
\wedge X_0:=(\ldots(\beta_1\wedge\beta_2)\wedge\ldots\wedge\beta_n)
\]
and observe that $\wedge X_0\rightarrow\alpha$ is a classical tautology, that is
$\wedge X_0\rightarrow\alpha\in L\textbf{B}_2$. Further, it is easy to see that
$X_0\vdash_2 \wedge X_0$. The two last observations imply that $X_0\vdash_2\alpha$ and hence $X\vdash_2\alpha$.

According to Proposition~\ref{P:matrix-con-is-con-relation} and the consideration above, $\models_{\textbf{B}_{2}}$, as well as $\vdash_2$, is a finitary structural consequence relation.\\

\paragraph{Relation $\vdash_3$} For the next example, we employ the hyperrules (c). We state that $X\vdash_3\alpha$ when the latter can be confirmed by one of the following types of \textit{confirmation} or their combination.  A formal definition is the following.
\begin{quote}
	$1^{\circ}$~The first type is a derivation of $\alpha$ from $X$ by using steps determined by ($\text{d}_1$)--($\text{d}_2$), that is a derivation confirming $X\vdash_1\alpha$.\\
	$2^{\circ}$~The second type is when confirmed premises of one of the hyperrules (c) are satisfied for $\vdash_3$, we can confirm the conclusion of this hyperrule. (This means an application of the hyperrule in question.)\\ 
	$3^{\circ}$~The third type consists in the following.
	When we have confirmed $X\vdash_3\alpha_1,\ldots,X\vdash_3\alpha_n$, we confirm
	$X\vdash_3\alpha$ if $X,\alpha_1,\ldots,\alpha_n\vdash_1\alpha$.
\end{quote}

Thus, for any set $X$ of $\Lan_{A}$-formulas and any $\Lan_{A}$-formula $\alpha$, $X\vdash_3\alpha$ is true if, and only if, there is a consecutive finite number of confirmations of the types of $1^{\circ}$--$3^{\circ}$. In particular,
\begin{equation}\label{E:classic-1-implies-classic-3}
	X\vdash_1\alpha\Longrightarrow X\vdash_3\alpha.
\end{equation}

We illustrate the last definition on the following example. Using $X=\lbrace\alpha,\beta\rbrace$ (for concrete formulas $\alpha$ and $\beta$), we confirm $X\vdash_3\alpha\wedge\beta$ by the derivation
\[
\alpha,\beta,\alpha\wedge\beta.
\]
Then, we rewrite the last sequent as $\lbrace\alpha\rbrace,\beta\vdash_3\alpha\wedge\beta$ and apply (c-$i$) to obtain $\lbrace\alpha\rbrace\vdash_3\beta\rightarrow(\alpha\wedge\beta)$ (confirmation of the second type). Now, we rewrite the last sequent as $\emptyset,\alpha\vdash_3\beta\rightarrow(\alpha\wedge\beta)$ and apply this hyperrule one more time. The resulting sequent is $\emptyset\vdash_3\alpha\rightarrow(\beta\rightarrow(\alpha\wedge\beta))$, or
simply $\vdash_3\alpha\rightarrow(\beta\rightarrow(\alpha\wedge\beta))$.\\

It should be clear that the relation $\vdash_3$ is monotone; that is
\begin{equation}\label{E:classical-3-is-monotone}
	X\vdash_3\alpha~\textit{and}~X\subseteq Y~\textit{imply}~Y\vdash_3\alpha.
\end{equation}
At the point, we do not know yet if $\vdash_3$ is a consequence relation. Our aim is to prove this. More than that, we shall show that $\vdash_3$ and $\vdash_2$ are equal as relations, though defined differently. We discuss advantage and disadvantage of $\vdash_3$ in relation to $\vdash_2$ at the end of this subsection.

Aiming to prove $\vdash_3\,=\,\vdash_2$, we notice that the relation $\vdash_3$ enjoys the following variant of cut (the property $(\text{c}^{\ast})$ from Definition~\ref{D:consequnce-relation-single}).
\begin{equation}\label{E:calculus-3-has-cut}
	\textit{If $X,\beta\vdash_3\alpha$ and $\vdash_3\beta$, then $X\vdash_3\alpha$}.
\end{equation}

Indeed, applying the hyperrule of (c-$i$), we confirm that $X\vdash_3\beta\rightarrow\alpha$. Using \eqref{E:classical-3-is-monotone}, we obtain that $X\vdash_3\beta$. Finally, we use a confirmation of type $3^{\circ}$ to conclude that $X\vdash_3\alpha$.

On the next step, we prove that
\begin{equation}\label{E:cal-2-implies-cal-3}
	X\vdash_2\alpha\Longrightarrow X\vdash_3\alpha.
\end{equation}

Our proof of~\eqref{E:cal-2-implies-cal-3} is based on the following observation.
\begin{equation}\label{E:classical-tautology-in-cal-3}
	\textit{If {\em$\alpha\in L\textbf{B}_2$}, then $\vdash_3\alpha$}.
\end{equation}

To prove~\eqref{E:classical-tautology-in-cal-3}, first, we show that for any $\Lan_{A}$-formula $\alpha$,
\begin{equation}\label{E:excluded-middle}
	\vdash_3\alpha\vee\neg\alpha.
\end{equation}

Indeed,\footnote{The following steps 1--6 are the adaptation of the proof of $51^{\ast}$ from~\cite{kleene1952}, {\S} 27.}
\[
\begin{array}{cl}
1. &\lbrace\neg(\alpha\vee\neg\alpha)\rbrace,\alpha\vdash_3 \alpha\vee\neg\alpha\quad[\text{for $\neg(\alpha\vee\neg\alpha),\alpha\vdash_1 \alpha\vee\neg\alpha$ and~\eqref{E:classic-1-implies-classic-3}}]\\
2. &\lbrace\neg(\alpha\vee\neg\alpha)\rbrace,\alpha\vdash_3 \neg(\alpha\vee\neg\alpha)\quad[\text{for $\neg(\alpha\vee\neg\alpha),\alpha\vdash_1 \neg(\alpha\vee\neg\alpha)$ and~\eqref{E:classic-1-implies-classic-3}}]\\
3. &\neg(\alpha\vee\neg\alpha)\vdash_3\neg\alpha\quad[\text{by (c-$ii$) with 1 and 2 as premises}]\\
4. &\neg(\alpha\vee\neg\alpha)\vdash_3\neg\neg\alpha\quad[\text{analogously to 1--3}]\\
5. &\vdash_3\neg\neg(\alpha\vee\neg\alpha)\quad[\text{by (c-$ii$) with 3 and 4 as premises}]\\
6. &\vdash_3\alpha\vee\neg\alpha\quad[\text{by applying \eqref{E:calculus-3-has-cut} to 5 and $\neg\neg(\alpha\vee\neg\alpha)\vdash_1\alpha\vee\neg\alpha$}]
\end{array}
\]

Also, we will use lemma 13 of~\cite{kleene1952}, {\S} 29, which we formulate in a paraphrased form as follows.
\begin{quote}
	Let $\alpha$ be an $\Lan_{A}$-formula in the distinct propositional variables $p_1,\ldots,p_n$; let $v$ be a valuation in $\textbf{B}_2$; and let a sequence of $\Lan_{A}$-formulas is produced according to the rule: if $v(p_j)=\one$, then
	$\beta_j=p_j$, and if $v(p_j)=\zero$, then $\beta_j=\neg p_j$. Then 
	$\beta_1,\ldots,\beta_n\vdash_3\alpha$ or $\beta_1,\ldots,\beta_n\vdash_3\neg\alpha$, according as for the given valuation $v$ $\alpha$ takes the value $\one$ or the value $\zero$.
\end{quote}

Now, to prove ~\eqref{E:classical-tautology-in-cal-3}, we assume that $\alpha\in L\textbf{B}_2$ and $p_1,\ldots,p_n$ are all variables that occur in $\alpha$. Then, according to the last property, $p_1,\ldots,p_n\vdash_3\alpha$ and
$p_1,\ldots,\neg p_n\vdash_3\alpha$. This, in virtue of (c-$iii$), implies that
$p_1,\ldots,p_{n-1},p_{n}\vee\neg p_n\vdash_3\alpha$. Since, by~\eqref{E:calculus-3-has-cut}, $\vdash_3 p_{n}\vee\neg p_n$, we obtain that
$p_1,\ldots,p_{n-1}\vdash_3\alpha$. Continuing the elimination of the variables $p_j$, we come to the conclusion that $\vdash_3\alpha$.  

Now we obtain the proof of~\eqref{E:cal-2-implies-cal-3} as follows.

Assume that $X\vdash_2\alpha$; that is for some $\alpha_1,\ldots,\alpha_n\in L\textbf{B}_2$, $X,\alpha_1,\ldots,\alpha_n\vdash_1\alpha$. According to \eqref{E:classic-1-implies-classic-3}, $X,\alpha_1,\ldots,\alpha_n\vdash_3\alpha$,
where for each $\alpha_j$, in virtue of~\eqref{E:classical-tautology-in-cal-3},
$\vdash_3\alpha_j$. Using \eqref{E:calculus-3-has-cut} $n$ times, we obtain that
$X\vdash_3\alpha$.

Finally, we show that
\begin{equation}\label{E:cal-3-implies-cal-2}
	X\vdash_3\alpha\Longrightarrow X\vdash_2\alpha.
\end{equation}
The last implication will be proven if we show that
\begin{equation}\label{E:cal-3-implies-B-2}
	X\vdash_3\alpha\Longrightarrow X\models_{\textbf{B}_{2}}\alpha
\end{equation}
and then apply~\eqref{E:classical-2-converse}. 

We note  that~\eqref{E:cal-3-implies-B-2} will be proven if we establish that every rule of (a)--(b) and every hyperrule of (c) preserves the truth value $\one$. Let us check this property for the hyperrules.

We begin with (c-$i$). Assume that for any valuation $v^{\prime}$ in $\textbf{B}_2$, 
\[
v^{\prime}[X,\alpha]\subseteq\lbrace\one\rbrace\Longrightarrow v^{\prime}[\beta]=\one.
\]
Suppose for some valuation $v$, $v[X]\subseteq\lbrace\one\rbrace$. We observe that if $v[\alpha]=\zero$, then $v[\alpha\rightarrow\beta]=\one$, regardless of the value of $v[\beta]$. However, if $v[\alpha]=\one$, then $v[\alpha\rightarrow\beta]=\one$, by premise. 

Now we turn to (c-$ii$). Suppose that for any valuation $v^{\prime}$ in $\textbf{B}_2$, 
\[
v^{\prime}[X,\alpha]\subseteq\lbrace\one\rbrace\Longrightarrow
\text{both $v^{\prime}[\beta]=\one$ and $v^{\prime}[\neg\beta]=\one$}.
\]
It is clear that this premise is equivalent to that for any valuation $v^{\prime}$ in $\textbf{B}_2$, 
\[
v^{\prime}[X,\alpha]\not\subseteq\lbrace\one\rbrace.
\]
Thus, if for some valuation $v$, $v[X]\subseteq\lbrace\one\rbrace$, then obviously
$v[\alpha]=\zero$, that is $v[\neg\alpha]=\one$.

To prove this preservation property for (c-$iii$), we suppose that for any valuation $v^{\prime}$ in $\textbf{B}_2$, both
\[
\begin{array}{l}
v^{\prime}[X,\alpha]\subseteq\lbrace\one\rbrace\Longrightarrow v^{\prime}[\gamma]=\one~\text{and}\\
v^{\prime}[X,\beta]\subseteq\lbrace\one\rbrace\Longrightarrow v^{\prime}[\gamma]=\one
\end{array}
\]
hold. Now assume that for some valuation $v$, $v[X,\alpha\vee\beta]\subseteq\lbrace\one\rbrace$. It is clear that either $v[\alpha]=\one$ or $v[\beta]=\one$ (or both). In each case, we obtain that $v[\gamma]=\one$.

Thus we have obtained the equality $\vdash_2\,=\,\vdash_3$. As a byproduct of this equality, we conclude that $\vdash_3$ is a structural consequence relation.\\

\paragraph{Comparison of $\vdash_2$ and $\vdash_3$} We begin with the remark that neither $\vdash_2$ nor $\vdash_3$ is a calculus; the later is not because of the presence of hyperrules, the former because of the factual presence of the infinite set $L\textbf{B}_2$ of classical tautologies. However, $\vdash_2$ can be reformulated in such a way that in a new formulation, let us call it $\vdash_{2^{\circ}}$, the latter is a calculus. First of all, we remove from the definition of $\vdash_2$ all rules except $\alpha,\alpha\rightarrow\beta\slash\beta$ (known as \textit{modus ponens} or \textit{detachment}). The removed rules are actually replaced with rules without premises, that is rules of type $\emptyset\slash\varphi$, where $\varphi$ is one of the following forms of formulas (which are called  \textit{axiom schemata}):
\[
\begin{array}{cl}
\text{ax1} &\bm{\alpha}\rightarrow(\bm{\beta}\rightarrow\bm{\alpha}),\\
\text{ax2} &(\bm{\alpha}\rightarrow\bm{\beta})
\rightarrow((\bm{\alpha}\rightarrow(\bm{\beta}\rightarrow\bm{\gamma}))\rightarrow
(\bm{\alpha}\rightarrow\bm{\gamma})),\\
\text{ax3} &\bm{\alpha}\rightarrow(\bm{\beta}\rightarrow(\bm{\alpha}\wedge\bm{\beta})),\\
\text{ax4} &(\bm{\alpha}\wedge\bm{\beta})\rightarrow\bm{\alpha},\\
\text{ax5} &(\bm{\alpha}\wedge\bm{\beta})\rightarrow\bm{\beta},\\
\text{ax6} &\bm{\alpha}\rightarrow(\bm{\alpha}\vee\bm{\beta}),\\
\text{ax7} &\bm{\beta}\rightarrow(\bm{\alpha}\vee\bm{\beta}),\\
\text{ax8} &(\bm{\alpha}\rightarrow\bm{\gamma})\rightarrow((\bm{\beta}\rightarrow\bm{\gamma})
\rightarrow((\bm{\alpha}\vee\bm{\beta})\rightarrow\bm{\gamma})),\\
\text{ax9} &(\bm{\alpha}\rightarrow\bm{\beta})
\rightarrow((\bm{\alpha}\rightarrow\neg\bm{\beta})\rightarrow\neg\bm{\alpha}), \\
\text{ax10} &\neg\neg\bm{\alpha}\rightarrow\bm{\alpha}.
\end{array}
\]

We will not prove that $\vdash_{2^{\circ}}\,=\,\vdash_2$, but we state that $\vdash_{2^{\circ}}$ is a calculus.

However, the aforementioned removal of rules is not necessary; it would suffice to add to the rules (a)--(b) nine new rules of type $\emptyset\slash\varphi$, where $\varphi$ is one of the schemata ax1--ax10; in this case, the notion of derivation should be modified from $\vdash_1$ (namely the clause ($\text{d}_{2}$)) by including the new rules. We treat the forms ax1--ax10 as \textit{\textbf{premiseless modus rules}}. 

Although $\vdash_3$ is not a calculus, it is very convenient in confirming sequents that are true w.r.t. $\vdash_{2^{\circ}}$. For instance, we show that the schema ax8 above is a thesis of $\vdash_3$. Indeed, we have:
\[
\begin{array}{cl}
1. &\lbrace\alpha\rightarrow\gamma,\beta\rightarrow\gamma\rbrace,\alpha
\vdash_3\gamma;~\lbrace\alpha\rightarrow\gamma,\beta\rightarrow\gamma\rbrace,\beta
\vdash_3\gamma\quad[\text{two $\vdash_1$-derivations}]\\
2. &\lbrace\alpha\rightarrow\gamma,\beta\rightarrow\gamma\rbrace,
\alpha\vee\beta\vdash_3\gamma\quad[\text{from 1 by (c-$iii$)}]\\
3. &\lbrace\alpha\rightarrow\gamma\rbrace,\beta\rightarrow\gamma
\vdash_3(\alpha\vee\beta)\rightarrow\gamma\quad[\text{from 2 by (c-$i$)}]\\
4. &\alpha\rightarrow\gamma\vdash_3(\beta\rightarrow\gamma)\rightarrow((\alpha\vee\beta)\rightarrow\gamma)\quad[\text{from 3 by (c-$i$)}]\\
5. &\vdash_3(\alpha\rightarrow\gamma)\rightarrow((\beta\rightarrow\gamma)\rightarrow((\alpha\vee\beta)\rightarrow\gamma)).\quad[\text{from 4 by (c-$i$)}]
\end{array}
\]
However, there is a price for such a convenience. While in proving $X\vdash_2\alpha$ (or in proving $X\vdash_{2^{\circ}}\alpha$) a set $X$ stays the same in each application of a rule; in proving $X\vdash_3\alpha$, when we need to apply a hyperrule, the set ``$X$'' in the premise(s) of a hyperrule may differ from the original $X$ in the sequent we are proving. In addition, in each application of a hyperrule, the $X$ in its premise(s) may vary from hyperrule to hyperrule or from one application to another of the same hyperrule.

\paragraph{The idea of semantic completeness}\label{paragraph:idea-completeness}
The equalities like $\vdash_2\,=\,\models_{\textbf{B}_{2}}$ or $\vdash_{2^{\circ}}\,=\,\models_{\textbf{B}_{2}}$ or $\vdash_3\,=\,\models_{\textbf{B}_{2}}$, but not $\vdash_2\,=\,\vdash_{2^{\circ}}$ or $\vdash_2\,=\,\vdash_3$, (see Section~\ref{section:consequence-defining}) are typical examples of semantic completeness. That is, a completeness result is claimed when one and the same consequence relation admits two different definitions --- one purely syntactic and the other with the help of semantic concepts, such as, e.g., logical matrices. The former is regarded as a convenient way to define an abstract logic, say $\mathcal{S}$, the latter as an effective way to show that for a particular set $X$ and formula $\alpha$, the sequent $X\vdash_{\mathcal{S}}\alpha$ does not hold. This issue will be further discussed in the next chapter (Section~\ref{section:separating-means}).

To demonstrate, e.g., the validity of $\vdash_2\,=\,\models_{\textbf{B}_{2}}$, we proved the properties~\eqref{E:classical-2} and~\eqref{E:classical-2-converse}. The former is known in literature as \textit{soundness}, the latter as (properly) \textit{semantic completeness}. Thus we can state that the abstract logic $\vdash_2$ is \textit{complete w.r.t.} the logical matrix $\mathbf{B}_2$, or that
$\vdash_2$ is \textit{determined by} $\mathbf{B}_2$.

Although the main concept of this book is consequence relation, the equalities
like 
\[
\begin{array}{c}
\bm{T}_{\textsf{LC}}=L\mat{LC}~~\text{and}~~\bm{T}_{\textsf{LC}}=L\mat{LC}^{\ast}\
\end{array}
\]
(compare with the $\LC$ completeness above) will attract our attention too.
We call them also completeness results but in a weaker sense of the term ``completeness''.

\paragraph{Exercises~\ref{section:consequence-defining}}
\begin{enumerate}
	\item \label{EX:trivial-1}Prove that a structural logic $\mathcal{S}$ is trivial if, and only if, for any arbitrary variable $p$, $p\in\ConS{\emptyset}$.
	\item\label{EX:trivial-2} Prove that for any abstract logic $\mathcal{S}$, $\mathcal{S}$ is trivial if, and only if, $\bm{T}_{\mathcal{S}}$ is inconsistent.
	\item \label{EX:theory}Let $\mathcal{S}$ be a structural abstract logic. Prove that an $\mathcal{S}$-theory is closed under substitution if, and only if, this theory is generated by a set closed under substitution. Prove that $\bm{T}_{\mathcal{S}}$ is the least $\mathcal{S}$-theory closed under any substitution.
	\item\label{EX:matrix-con-is-con-relation} Prove Proposition~\ref{P:matrix-con-is-con-relation}.
	\item\label{EX:lindenbaum-matrix} Prove Corollary~\ref{C:lindenbaum-matrix}.
	\item \label{EX:Lindenbaum-completeness}
	Prove Corollary~\ref{C:Lindenbaum-completeness}
	\item \label{EX:S-matrix-completeness}
	Prove Corollary~\ref{C:S-matrix-completeness}
	\item \label{EX:S_R-consequence} Show that $\vdash_{[\mathcal{R}]}$ is a consequence relation.
	\item \label{EX:S-as-a-rule}Prove that for any abstract logic $\mathcal{S}$, the consequence relation $\vdash_{\mathcal{S}}$ is determined by $\lbrace\vdash_{\mathcal{S}}\rbrace$.
	\item \label{EX:classical-1}Prove \eqref{E:classical-1}.
	\item Show that the schemata ax1--ax7 and ax9--ax10 are theses of $\vdash_3$.
	\item Although for any consequence operator $\textbf{Cn}$, $\Con{X\cap Y}\subseteq\Con{X}\cap\Con{Y}$ for any sets $X$ and $Y$, the converse inclusion is not necessarily true. Define such an abstract logic $\mathcal{S}$ that $\ConS{X}\cap\ConS{Y}\not\subseteq\ConS{X\cap Y}$ for some sets $X$ and $Y$.
\end{enumerate}

\section{Abstract logics defined by modus rules}\label{section:modus-rules}
When we use or investigate calculi, we inevitably deal with modus rules.
As we will see below (Corollary~\ref{C:logic-by-modus-rules-criterion}), for a structural logic to be finitary  and defined as calculus, are equivalent properties. On the other hand, the modus rules are naturally modeled in first order logic.

\subsection{General characterization}\label{section:modus-rules-general}
If $\mathcal{R}$ is a set of modus rules, there is an effective way to determine the consequence relation $\vdash_{[\mathcal{R}]}$ defined in~\eqref{E:consequence-by-rules} (Section~\ref{section:inference-rules}). For this, we define a relation $\vdash_{\mathcal{R}}$ on $\mathcal{P}(\Forms)\times\Forms$, by generalizing (for $\Lan$-formulas) the notion of derivation which earlier was defined for $\Lan_{A}$-formulas via ($\text{d}_1$)--($\text{d}_2$).

The relation $X\vdash_{\mathcal{R}}\alpha$ holds, by definition, if there is a list
\begin{equation}\label{E:list-2}
	\alpha_1,\ldots,\alpha_n,\alpha
\end{equation}
(the list $\alpha_1,\ldots,\alpha_n$ can be empty), where  each formula on \eqref{E:list-2} is 
\[
\begin{array}{cl}
(\text{d}^{\ast}_{1}) &\text{either in $X$ or}\\
(\text{d}^{\ast}_{2}) &\text{obtained from the formulas preceding it on \eqref{E:list-2} by one of the rules $\mathcal{R}$.}
\end{array}
\]
We note that so defined $\vdash_{\mathcal{R}}$ is finitary. About~\eqref{E:list-2}, we say that it is an $\mathcal{R}$-\textit{derivation} of $\alpha$ from $X$.
\begin{prop}\label{P:modus-rules}
	Let $\mathcal{R}$ be a set of modus rules. Then $\vdash_{\mathcal{R}}\,=\,\vdash_{[\mathcal{R}]}$. Hence $\vdash_{\mathcal{R}}$ is a consequence relation.
\end{prop}
\begin{proof}
	We shall prove successively the following three claims.
	\[
	\begin{array}{cl}
	1^{\circ} &\vdash_{\mathcal{R}}~\text{is a structural consequence relation};\\
	2^{\circ} &\text{for any $R\in\mathcal{R}$, $R$ is sound w.r.t. $\vdash_{\mathcal{R}}$};\\
	3^{\circ} &\text{if an abstract logic $\mathcal{S}$ is such that if for any $R\in\mathcal{R}$,}\\
	&\text{$R$ is sound w.r.t. $\mathcal{S}$, then $\vdash_{\mathcal{R}}\subseteq\vdash_{\mathcal{S}}$.}
	\end{array}
	\]
	
	We notice that the properties (a)--(b) of Definition~\ref{D:consequnce-relation-single} are obviously true for $\vdash_{\mathcal{R}}$; and we leave for the reader to check the property (c) and  structurality property. (Exercise~\ref{section:modus-rules}.\ref{EX:modus-rules})
	
	Next, suppose that $\langle\Gamma,\phi\rangle\in\mathcal{R}$. Let $\bm{\sigma}$ be an arbitrary instantiation in $\Lan$. Since $\Gamma$ is finite, we denote
	\[
	\bm{\sigma}(\Gamma):=\lbrace\alpha_1,\ldots,\alpha_n\rbrace.
	\]
	According to the clauses $(\text{d}^{\ast}_{1})$--$(\text{d}^{\ast}_{2})$, the list
	\[
	\alpha_1,\ldots,\alpha_n,\bm{\sigma}(\phi)
	\]
	is an $\mathcal{R}$-derivation of $\bm{\sigma}(\phi)$ from $\bm{\sigma}(\Gamma)$. Hence $\bm{\sigma}(\Gamma)\vdash_{\mathcal{R}}\bm{\sigma}(\phi)$.
	
	Finally, let $\mathcal{S}$ be an abstract logic and for any $R\in\mathcal{R}$, $R$ is sound w.r.t. $\mathcal{S}$. We show that then
	\[
	X\vdash_{\mathcal{R}}\alpha\Longrightarrow X\vdash_{\mathcal{S}}\alpha.
	\]
	
	Suppose~\eqref{E:list-2} is a derivation w.r.t. $\mathcal{R}$. We prove that for each formula $\alpha_i$ of~\eqref{E:list-2}, $X\vdash_{\mathcal{S}}\alpha_i$.
	We employ induction on the number of the formulas preceding $\alpha$ in~\eqref{E:list-2}.
	Indeed, we see that, if $\alpha=\alpha_{1}$, then $X\vdash_{\mathcal{S}}\alpha_1$, for either $\alpha\in X$ or $\alpha$ is a conclusion of a premiseless rule.
	
	Now, suppose $\alpha$ is obtained by application of a rule  $R_{\langle\Gamma,\bm{\beta}\rangle}$. That is, for some formulas $\alpha_{i_1},\ldots,
	\alpha_{i_k}$ preceding $\alpha$ and some instantiation $\bm{\sigma}$,
	$\bm{\sigma}(\Gamma)=\lbrace\alpha_{i_1},\ldots,
	\alpha_{i_k}\rbrace$ and $\bm{\sigma}(\bm{\beta})=\alpha$. Since, by induction, $X\vdash_{\mathcal{S}}\alpha_{i_j}$, $1\le j\le k$, and, by soundness of $R_{\langle\Gamma,\bm{\beta}\rangle}$ w.r.t. $\mathcal{S}$, $\bm{\sigma}(\Gamma)\vdash_{\mathcal{S}}\alpha$, we conclude that
	$X\vdash_{\mathcal{S}}\alpha$.
\end{proof}

Given a set $\mathcal{R}$ of modus rules, we denote by $\mathcal{S}_{\mathcal{R}}$ the abstract logic corresponding to $\vdash_{[\mathcal{R}]}$ (see the definition~\eqref{E:consequence-by-rules}) or, equivalently (due to Proposition~\ref{P:modus-rules}), to $\vdash_{\mathcal{R}}$.
\begin{cor}\label{C:modus-rules-finitary}
	If $\mathcal{R}$ is a set of modus rules, then the abstract logic $\mathcal{S}_{\mathcal{R}}$ is structural and finitary.
\end{cor}
\noindent\textit{Proof}~is left to the reader. (Exercise~\ref{section:modus-rules}.\ref{EX:modus-rules-finitary})

\begin{cor}\label{C:logic-by-modus-rules-criterion}
	Let $\mathcal{S}$ be a structural abstract logic. Then $\mathcal{S}$ is finitary if, and only if, there is a set $\mathcal{R}$ of modus rules such that $\mathcal{S}=\mathcal{S}_{\mathcal{R}}$.
\end{cor}
\begin{proof}
	The if-part follows from Corollary~\ref{C:modus-rules-finitary}.
	
	To prove the only-if part, we assume that $\mathcal{S}$ is finitary. Next we define:
	\[
	\mathcal{R}:=\set{\langle\Gamma,\phi\rangle}{\Gamma\cup\lbrace\phi\rbrace\Subset
		\Mform~\text{and $\bm{\xi}(\Gamma)\vdash_{\mathcal{S}}\bm{\xi}(\phi)$, for any instantiation $\bm{\xi}$}}.
	\]
	
	We aim to prove that $\mathcal{S}=\mathcal{S}_{\mathcal{R}}$.
	
	First, assume that $X\vdash_{\mathcal{S}}\beta$, for a particular set $X\cup\lbrace\beta\rbrace\subseteq\Forms$.
	Since $\mathcal{S}$ is finitary, there is a set
	$Y\Subset X$ such that $Y\vdash_{\mathcal{S}}\beta$. Now, replacing the variables $\Var(Y\cup\lbrace\beta\rbrace)$ with metavariables according to arbitrary one-one correspondence, we obtain a finite set $\Gamma\cup\lbrace\phi\rbrace$ of metaformulas such that for any instantiation $\bm{\xi}$,
	$\bm{\xi}(\Gamma)\vdash_{\mathcal{S}}\bm{\xi}(\phi)$ and for a particular instantiation $\bm{\xi}_{0}$, $\bm{\xi}_{0}(\Gamma)=Y$ and 
	$\bm{\xi}_{0}(\phi)=\beta$. This implies that
	$\langle\Gamma,\phi\rangle\in\mathcal{R}$. The latter in turn allows us to conclude that $Y\vdash_{\mathcal{R}}\beta$ and hence $X\vdash_{\mathcal{R}}\beta$. 
	
	Now, we suppose that $X\vdash_{\mathcal{R}}\beta$. That is, there is an $\mathcal{R}$--derivation, determined by $(\text{d}^{\ast}_{1})$--$(\text{d}^{\ast}_{2})$:
	\[
	\alpha_1,\ldots,\alpha_n,\beta. \tag{$\ast$}
	\]
	
	It is obvious that the list $(\ast)$ is not empty. We prove (by induction) that for each formula $\gamma$ of $(\ast)$, $X\vdash_{\mathcal{S}}\gamma$.
	
	It is true for the first formula $\gamma$ in $(\ast)$, in which case $\gamma=\alpha_1=\beta$, for, then, $\gamma\in X$ and hence $X\vdash_{\mathcal{S}}\gamma$.
	
	Suppose we proved that for each $\alpha_{i}$, $X\vdash_{\mathcal{S}}\alpha_{i}$. If $\beta\in X$, we are done. If $\beta\notin X$, then $\beta$ is obtained from $\alpha_{i_1},\ldots,\alpha_{i_k}$ by a rule
	$\langle\Gamma,\phi\rangle\in\mathcal{R}$. By induction, this implies that $\lbrace\alpha_{i_1},\ldots,\alpha_{i_k}\rbrace
	\vdash_{\mathcal{S}}\beta$. By the transitivity property, we conclude that $X\vdash_{\mathcal{S}}\beta$.
\end{proof}

\subsection{The class of $\mathcal{S}$-models}\label{section:S-models}
In this subsection we focus on the class of $\mathcal{S}$-models, when an abstract logic $\mathcal{S}$ can be determined by a set of modus rules. Given a structural abstract logic
$\mathcal{S}$ in a language $\Lan$, we employ for our characterization of $\mathcal{S}$-models a special language of first order. This language depends on $\Lan$ and we denote it by $\textbf{FO}_{\Lan}$.

The denumerable set of the individual variables of $\textbf{FO}_{\Lan}$ consists of the
symbols $\bm{x_{\alpha}}$, for each metavariable $\bm{\alpha}$. The functional
connectives coincide with the connectives of $\Lan$. Also, there is one unary predicate
symbol, $D$, two logical connectives, $\&$ (conjunction) and $\Rightarrow$
(implication), universal quantifier $\forall$, as well the parentheses, `(' and `)'. The rules of formula formation are usual, as well as is usual the notion of an $\textbf{FO}_{\Lan}$-structure (including the notions of \textit{satisfactory} and \textit{validity}). If $\Theta$ is an $\textbf{FO}_{\Lan}$-formula containing only free variables, its universal closure is denoted by $\forall\ldots\forall\Theta$.
Main reference here is~\cite{chang-keisler1990}.\\

It must be clear that if $\phi$ is a metaformula relevant to $\Lan$, then, replacing each metavariable $\bm{\alpha}$ in $\phi$ with the individual variable $\bm{x_\alpha}$, we obtain a term
of $\textbf{FO}_{\Lan}$. We denote this term (obtained from $\phi$) 
by $\phi^{\ast}$.

Further, any modus rule $R:=\dfrac{\psi_1,\ldots,\psi_n}{\phi}$ is translated to the formula
\begin{equation}\label{E:horn-formula-1}
	R^{\ast}:=\forall\ldots\forall((D(\psi_{1}^{\ast})\&\ldots\& D(\psi_{n}^{\ast}))\Rightarrow D(\phi^{\ast}))
\end{equation}
and any premiseless rule $T:=\emptyset\slash\phi$ to the formula
\begin{equation}\label{E:horn-formula-2}
	T^{\ast}:=\forall\ldots\forall D(\phi^{\ast}).
\end{equation}

The formulas of types~\eqref{E:horn-formula-1} and~\eqref{E:horn-formula-2} belong to the class of \textit{basic Horn sentences}; cf.~\cite{chang-keisler1990}, section 6.2.\\

It is clear that any matrix $\mat{M}=\langle\alg{A},D\rangle$ of type $\Lan$ can be regarded as a model for $\textbf{FO}_{\Lan}$-formulas.  We denote the \textit{validity} of a closed formula $\Theta$ in $\mat{M}$ by $\mat{M}\Vdash\Theta$.

Let $\psi$ be a metaformula of $\Mform$ containing only metavariables $\alpha_1,\ldots,\alpha_n$. We note that there is a strong correlation between the standard valuation of the term $\psi^{\ast}$ (containing only individual variables $x_{\alpha_1}, \ldots,x_{\alpha_n}$) in a first order structure $\mat{M}$ and a valuation of a formula $\bm{\xi}(\psi)$, obtained from $\psi$ by a simple instantiation $\bm{\xi}$, in a logical matrix $\mat{M}$. For instance, let us denote the value of $\psi^{\ast}$ on a sequence $a_1,\ldots,a_n$ in $\mat{M}$ (as a first order structure) by $\psi^{\ast}[a_1,\ldots,a_n]$. Now let $\bm{\xi}$ be a simple instantiation with
$\bm{\xi}(\alpha_i)=p_i$. Then, if $v$ is a valuation in $\alg{A}$ with $v[p_i]=a_i$,
then $\psi^{\ast}[a_1,\ldots,a_n]=v[\bm{\xi}(\psi)]$. This can be proven by induction on the degree of $\psi$.

Grounding on the last conclusion, we obtain The first order formula $D(\psi^{\ast})$ is satisfied on the sequence $\alpha_1,\ldots,\alpha_n$ in the first order structure $\mat{M}$ if, and only if, $v[\bm{\xi}(\psi)]\in D$ in the logical matrix $\mat{M}$.

This argument leads to the following lemma.

\begin{lem}\label{L:R-and-R*-in-matrix}
	Let $\mat{M}=\langle\alg{A},D\rangle$ be a matrix of type $\Lan$. For any modus rule $R$, $R$ is sound w.r.t. {\em$\models_{\textbf{M}}$} if, and only if, $\mat{M}\Vdash R^{\ast}$.
\end{lem}
\noindent\textit{Proof}~is left to the reader. (Exercise~\ref{section:modus-rules}.\ref{EX:R-and-R*-in-matrix}) 

\begin{prop}\label{P:S-models-modus-rules}
	Let $\mathcal{R}$ be a set of modus rules in $\Lan$. Then for any $\Lan$-matrix $\mat{M}$, $\mat{M}$ is an $\mathcal{S}_{\mathcal{R}}$-model if, and only if, 
	$\mat{M}\Vdash R^{\ast}$, for any $R\in\mathcal{R}$.
\end{prop}
\begin{proof}
	We fix an $\Lan$-matrix $\mat{M}=\langle\alg{A},D\rangle$. In view of Proposition~\ref{P:modus-rules}, we can associate with the abstract logic $\mathcal{S}_{\mathcal{R}}$ the consequence relation $\vdash_{\mathcal{R}}$.
	
	We begin with the ``only-if'' part. Thus we assume that the inclusion $\vdash_{\mathcal{R}}\,\subseteq\,\models_{\textbf{M}}$ holds. This implies that any $R\in\mathcal{R}$ is sound w.r.t. $\models_{\textbf{M}}$. By Lemma~\ref{L:R-and-R*-in-matrix}, this in turn implies that $\mat{M}\Vdash R^{\ast}$, for any $R\in\mathcal{R}$.
	
	To prove the ``if'' part, we assume that the last conclusion holds. By Lemma~\ref{L:R-and-R*-in-matrix}, we derive that any $R\in\mathcal{R}$ is sound w.r.t.
	$\models_{\textbf{M}}$.
	
	Now, suppose that $X\vdash_{\mathcal{R}}\alpha$. That is, there is an $\mathcal{R}$-derivation~\eqref{E:list-2}. By induction on the length of~\eqref{E:list-2},
	one can prove that for any valuation $v$ in $\alg{A}$, if $v[X]\subseteq D$, then $v[\alpha]\in D$. We leave this task to the reader. (Exercise~\ref{section:modus-rules}.\ref{EX:S-models-modus-rules})
\end{proof}
\begin{cor}\label{C:S-models-modus-rules}
	Let $\mathcal{R}$ be a set of modus rules in $\Lan$. Then the class of all 
	$\mathcal{S}_{\mathcal{R}}$-models is closed under reduced products, in particular under ultraproducts.
\end{cor}
\begin{proof}
	Let $\mathcal{M}$ be the class of all $\mathcal{S}_{\mathcal{R}}$-models. According to Proposition~\ref{P:S-models-modus-rules}, an $\Lan$-matrix $\mat{M}\in\mathcal{M}$ if, and only if, for any $R\in\mathcal{R}$, $R\in\mathcal{R}$. This means that $\mathcal{M}$ is determined by a set of Horn formulas. It is well known that every Horn sentence is preserved under reduced products, in particular under ultraproducts; cf.~\cite{chang-keisler1990}, proposition 6.2.2.
\end{proof}

\subsection{Modus rules vs. non-finitary rules}\label{section:modus-rules-vs-rules}

Proposition~\ref{P:modus-rules} and Corollary~\ref{C:modus-rules-finitary} show limitation of modus rules in defining consequent relations. It should not be surprising that there are consequence relations that cannot be defined by any set of modus rules. Below we discuss an example of the consequence relation of that kind; moreover, the theorems of that abstract logic, as we will show, coincide with the theorems of a calculus, \textsf{LC}. (See \textsf{LC} completeness below.)

In the language $\Lan_{A}$, we formulate two more premiseless structural rules: 
\[
\begin{array}{cl}
\text{ax11} &\beta\rightarrow(\neg\beta\rightarrow\alpha),\\
\text{ax12} &(\alpha\rightarrow\beta)\vee(\beta\rightarrow\alpha).
\end{array}
\]

A calculus \textsf{LC} is defined through derivations according to $(\text{d}_1)$--$(\text{d}_2)$ where in $(\text{d}_2)$, however, we use only the rules 
ax1--ax9, ax11-ax12 and (\text{b}-$iii$) (modus ponens).

We note two properties of \textsf{LC}. The first is that for any set $X_{0}\cup\lbrace\alpha\rbrace\Subset\Forms_{\Lan_{A}}$,
\begin{equation}\label{E:LC-property-1}
	X_0\vdash_{\textsf{LC}}\alpha~\Longleftrightarrow~\wedge X_0\vdash_{\textsf{LC}}\alpha.
\end{equation}
(See Exercise~\ref{section:modus-rules}.\ref{EX:LC-property-1}.)

The second property is known as the ``deduction theorem'' (for \textsf{LC}), which in its $\Rightarrow$-part, in essence, asserts that the hyperrule (c--$i$) is sound w.r.t. $\vdash_{\textsf{LC}}$. Namely, for any set $X\cup\lbrace\alpha,\beta\rbrace\subseteq\Forms_{\Lan_{A}}$,
\begin{equation}\label{E:LC-deduction-theorem}
	X,\alpha\vdash_{\textsf{LC}}\beta~\Longleftrightarrow~X\vdash_{\textsf{LC}}\alpha\rightarrow\beta.
\end{equation}
(See Exercise~\ref{section:modus-rules}.\ref{EX:LC-deduction-theorem}.)

In our consideration, we employ the single matrix relation $\models_{\textbf{LC}}$  which is grounded on the logical matrix \textbf{LC}; see Section~\ref{section:dummett}. 

We will need the following property.

For any $\Lan_{A}$-formulas $\alpha$ and $\beta$,
\[
\vdash_{\textsf{LC}}\alpha~\Longleftrightarrow~\models_{\textbf{LC}}\alpha.
\tag{$\LC$ completeness}
\]

The last equivalence is due to M. Dummett~\cite{dummett1959}, theorem 1.\\

Next, we define a consequence relation, $\vdash_{\textsf{LC}}$ (or abstract logic $\LC$), analogously to the relation $\vdash_{2}^{\circ}$ (Section~\ref{section:consequence-defining}), but by using the rules ax1--ax9, ax11--ax12 and modus ponens (the rule (b-$iii$)) for the notion of \textsf{LC}-derivation defined by the clauses ($\text{d}_{1}^{\ast}$)--($\text{d}_{2}^{\ast}$) above. 

In addition, we consider the non-modus structural rule $R^{\ast}:=R_{\langle\Gamma^{\ast},\bm{\alpha}_0\rangle}$ with
\[
\Gamma^{\ast}:=\set{(\bm{\alpha}_{i}\leftrightarrow \bm{\alpha}_{j})\rightarrow \bm{\alpha}_0}{0<i<j<\omega},
\]
where $\bm{\alpha}_i$, $0\leq i<\omega$, are pairwise distinct metavariables. Also, we define an infinite set of modus rules
\[
\mathcal{R}:=\set{\Gamma\slash\bm{\alpha}_0}{\emptyset\subset\Gamma\Subset\Gamma^{\ast}}.
\]

Then, two consequence relations, $\vdash^{\ast}$ and $ \vdash^{\star} $, are defined in $\mathcal{P}(\Forms_{\Lan_{A}})\times\Forms_{\Lan_{A}}$
similarly to the notion of ${\textbf{LC}}$-derivation but also for $\vdash^{\ast}$ with $R^{\ast}$ as an additional inference rule and for $\vdash^{\star}$ with the rules of $\mathcal{R}$ as additional inference rules. 

We aim to show that, although for any set $X_{0}\cup\lbrace\alpha\rbrace\Subset\Forms_{\Lan_{A}}$,
\begin{equation}\label{E:finitary-equivalence}
	X_0\vdash^{\star}\alpha~\Longleftrightarrow~X_0\vdash^{\ast}\alpha,
\end{equation}
the consequence relation $\vdash^{\ast}$ cannot be defined by any set of structural modus rules, for, as we will show, it is not finitary; see Corollary~\ref{C:modus-rules-finitary}.

As concerns~\eqref{E:finitary-equivalence}, it is obvious, for in any $\vdash^{\ast}$-derivation with a finite set $X_0$ as a set of premises,
the rule $R^{\ast}$ can be applied only under such an instantiation $\bm{\xi}$, for which
there is an instantiation $\bm{\zeta}$ such that \mbox{$\bm{\xi}(\Gamma^\ast)\slash\bm{\xi}(\bm{\alpha_{0}})=\bm{\zeta}(\Gamma)
	\slash\bm{\zeta}(\bm{\alpha})$}, for some $\Gamma\Subset\Gamma^{\ast}$; and vice versa.

Finally, we demonstrate that the consequence relation $\vdash^{\ast}$ differs from $\vdash_{\textsf{LC}}$, for the latter is finitary (Corollary~\ref{C:modus-rules-finitary}), while the former cannot be defined by any set of modus rules. 

For contradiction, assume that $\vdash^{\ast}$ is finitary. Let $\bm{\sigma}$ be an instantiation such that $\bm{\sigma}(\bm{\alpha}_{i})=p_i$, $0\le i<\omega$. We note that $\bm{\sigma}(\Gamma^{\ast})\vdash^{\ast}p_0$. By assumption, for some finite $X_0\Subset\bm{\sigma}(\Gamma^{\ast})$, $X_{0}\vdash^{\ast}p_0$. In virtue of~\eqref{E:finitary-equivalence}, $X_0\vdash_{\textsf{LC}}p_0$. This in turn,
according to~\eqref{E:LC-property-1} and~\eqref{E:LC-deduction-theorem}, implies
that $\vdash_{\textsf{LC}}\wedge X_{0}\rightarrow p_0$. By \textsf{LC} completeness,
we obtain that $\models_{\textbf{LC}}\wedge X_{0}\rightarrow p_0$ and hence
$X_{0}\models_{\textbf{LC}}p_0$. Thus, by monotonicity, we have that
$\bm{\sigma}(\Gamma^{\ast})\models_{\textbf{LC}} p_0$. However, let $v$ be such a valuation that $v[p_0]$ is a pre-last element $a$ of $\mat{LC}$ w.r.t. $\leq$ and $v[p_i]\neq v[p_j]$, whenever $i<j$. Then, since
$v[p_i\leftrightarrow p_j]\leq a$, we have: $v[\bm{\sigma}(\Gamma^{\ast})]\subseteq\lbrace\one\rbrace$, but $v[p_0]=a\neq\one$.

\paragraph{Exercises~\ref{section:modus-rules}}
\begin{enumerate}
	\item \label{EX:modus-rules-finitary}Prove Corollary~\ref{C:modus-rules-finitary}.
	\item \label{EX:modus-rules}Show that if $\mathcal{R}$ is a set of structural modus rules, then the relation $\vdash_{[\mathcal{R}]}$ satisfies the property (c) and structurality property of Definition~\ref{D:consequnce-relation-single}.
	\item \label{EX:LC-property-1}Prove the equivalence~\eqref{E:LC-property-1}.
	\item\label{EX:LC-deduction-theorem} Prove the equivalence~\eqref{E:LC-deduction-theorem} (the deduction theorem for \textsf{LC}).
	\item \label{EX:R-and-R*-in-matrix}Prove Lemma~\ref{L:R-and-R*-in-matrix}.
	\item \label{EX:S-models-modus-rules}Finish the proof of Proposition~\ref{P:S-models-modus-rules}.
\end{enumerate}

\section{Extensions of abstract logics}\label{section:extensions}
In this section we consider interaction of abstract logics defined in different formal languages. We begin with the following two definitions.

\begin{defn}[extension, conservative extension]\label{D:conservative-extension}
	Let a language $\Lan^{+}$ be an extension of a language $\Lan$. Also, let $\mathcal{S}^{+}$ and $\mathcal{S}$ be abstract logics in the languages $\Lan^{+}$ and
	$\Lan$, respectively. The abstract logic $\mathcal{S}^{+}$ is an \textbf{extension of}  $($or \textbf{over}$)$ 
	$\mathcal{S}$ if for any set $X$ of  $\Lan$\!-formulas,
	{\em\[
		\textbf{Cn}_{\mathcal{S}}(X)\subseteq\textbf{Cn}_{\mathcal{S}^{+}}(X).
		\]}
	We call $\mathcal{S}^{+}$ a \textbf{conservative extension of} $($or \textbf{over}$)$ $\mathcal{S}$ if for any set $X$ of $\Lan$\!-formulas,
	{\em\begin{equation}\label{E:conservative-extension}
			\textbf{Cn}_{\mathcal{S}^{+}}(X)\cap\Forms=\textbf{Cn}_{\mathcal{S}}(X).
	\end{equation}}
\end{defn}

The following two observation are obvious.
\begin{prop}\label{P:trivial-extension}
	Let a language $\Lan^{+}$ be an extension of a language $\Lan$. Assume that $\mathcal{S}$ is a structural logic in $\Lan$ and $\mathcal{S}^{+}$ is its structural extension in $\Lan^{+}$. Then if $\mathcal{S}$ is trivial, then so is $\mathcal{S}^{+}$.
\end{prop}
\begin{proof}
	If $\mathcal{S}$ is trivial, then $p\in\ConS{\emptyset}$, for an arbitrary $p\in\Var_{\Lan}$. Since $\mathcal{S}^{+}$ is an extension of $\mathcal{S}$, $p\in\textbf{Cn}_{\mathcal{S}^{+}}(\emptyset)$. And since $\mathcal{S}^{+}$ is structural, the latter implies that $\mathcal{S}^{+}$ is trivial.
\end{proof}
\begin{prop}\label{P:intersection-conservative-extensions}
	Let a language $\Lan^{+}$ be an extension of a language $\Lan$ and let $\mathcal{S}$ be an abstract logic in $\Lan$.
	The intersection of any nonempty set of conservative extensions of $\mathcal{S}$ in $\Lan^{+}$ is a conservative extensions of $\mathcal{S}$ in $\Lan^{+}$.
\end{prop}
\noindent\textit{Proof}~is left to the reader. (Exercise~\ref{section:extensions}.\ref{EX:intersection-conservative-extensions})\\

A simple way to obtain a (structural) conservative extension of a given abstract logic is the following.

Indeed, assume that $\mathcal{S}$ is a structural logic in $\Lan$ and $\Lan^{+}$ is an extension of $\Lan$. According to Corollary~\ref{C:Lindenbaum-completeness}, the atlas $\Lin[\Sigma_{\mathcal{S}}]=\left\langle \FormAl,\Sigma_{\mathcal{S}}\right\rangle $ is adequate for $\mathcal{S}$.
Let us select an arbitrary formula $\alpha\in\Forms$ and expand the signature of $\FormAl$ to $\Lan^{+}$, thereby obtaining an expansion $\FormAl[\alpha]^{+}$, as follows. For any $c\in\Cons_{\Lan^{+}}\setminus\Cons_{\Lan}$ we interpret as the element $\alpha\in |\FormAl|$. Similarly, for any new connective $F\in\Func_{\Lan^{+}}\setminus\Func_{\Lan}$, we interpret any term $F(\beta_1,\ldots,\beta_n)$ by $\alpha$. Then, we define an abstract logic $\mathcal{S}^{+}$ as the one determined by the atlas $\left\langle \FormAl[\alpha]^{+},\Sigma_{\mathcal{S}}\right\rangle $.

According to Proposition~\ref{P:matrix-con-is-con-relation}, $\mathcal{S}^{+}$ is structural. Further, $\mathcal{S}^{+}$ is a conservative extension of $\mathcal{S}$, for for any set $X\cup\lbrace\beta\rbrace\subseteq\Forms$,
\begin{equation}\label{E:S^+-conservativity}
	X\vdash_{\mathcal{S}^{+}}\beta~\Longleftrightarrow~X\models_{\textsf{\textbf{Lin}}[\Sigma_{\mathcal{S}}]}\beta.
\end{equation}

We leave the reader to prove~\eqref{E:S^+-conservativity}; see Exercise~\ref{section:extensions}.\ref{EX:S^+-conservativity}.

In virtue of Proposition~\ref{P:trivial-extension}, $\mathcal{S}^{+}$ is trivial if $\mathcal{S}$ is trivial. \\

We summarize as follows.
\begin{prop}\label{P:S^+-conservativity}
	Let $\mathcal{S}$ be an structural abstract logic in $\Lan$ and $\Lan^{+}$ be an extension of $\Lan$. Then there is a conservative structural extension $\mathcal{S}^{+}$ of $\mathcal{S}$. Moreover, if $\mathcal{S}$ is trivial, $\mathcal{S}^{+}$ is also trivial.
\end{prop}

\begin{cor}
	Let a language $\Lan^{+}$ be an extension of a language $\Lan$. Any structural abstract logic in $\Lan$ has a least structural conservative extension in $\Lan^{+}$.
\end{cor}
\begin{proof}
	We use Proposition~\ref{P:S^+-conservativity} and  Proposition~\ref{P:intersection-conservative-extensions}. 	
\end{proof}

Our next definition will play a key part in the sequel.
\begin{defn}[relation $\vdash_{(\mathcal{S})^{+}}$]\label{D:(S)^+}
	Let a language $\Lan^{+}$ be an extension of a language $\Lan$. Also, let $\mathcal{S}$ be an abstract logic in $\Lan$. We define a relation $\vdash_{(\mathcal{S})^{+}}$ in $\mathcal{P}(\Forms_{\Lan^{+}})\times\Forms_{\Lan^{+}}$ as follows:
	{\em\[
		\begin{array}{rl}
		X\vdash_{(\mathcal{S})^{+}}\alpha\stackrel{\text{df}}{\Longleftrightarrow}
		\!\!\!&\textit{there is a set $Y\cup\lbrace\beta\rbrace\subseteq\Forms$ and an $\Lan^{+}$\!-substitution}\\
		&\textit{$\sigma$ such that $\sigma(Y)\subseteq X$, $\alpha=\sigma(\beta)$ and $Y\vdash_{\mathcal{S}}\beta$}.
		\end{array}
		\]}
\end{defn}

We find a justification for the last definition in the following proposition.
\begin{prop}\label{P:(S)^+}
	Let a language $\Lan^{+}$ be an extension of a language $\Lan$ and $\mathcal{S}$ be a structural abstract logic in $\Lan$. If the relation $\vdash_{(\mathcal{S})^{+}}$ is a consequence relation, then $(\mathcal{S})^{+}$ is the least structural conservative extension of $\mathcal{S}$ in $\Lan^{+}$. In addition, $(\mathcal{S})^{+}$ is nontrivial if, and only if, $\mathcal{S}$ is nontrivial.
\end{prop}
\begin{proof}
	It is obvious that $(\mathcal{S})^{+}$ is a conservative extension of $\mathcal{S}$.
	(The structurality of $\mathcal{S}$ should be used.)
	
	Next, we prove the structurality of $(\mathcal{S})^{+}$.
	For this, assume that $X\vdash_{(\mathcal{S})^{+}}\alpha$. By definition, there are a set $Y\cup\lbrace\beta\rbrace\subseteq\Forms$ and an $\Lan^{+}$-substitution $\sigma$ such that $Y\vdash_{\mathcal{S}}\beta$, $\sigma(Y)\subseteq X$ and $\sigma(\beta)=\alpha$.
	Let $\xi$ be any $\Lan^{+}$-substitution. We observe: $\xi\circ\sigma(Y)\subseteq\xi(X)$ and $\xi\circ\sigma(\beta)=\xi(\alpha)$. Then, by definition, $\xi(X)\vdash_{(\mathcal{S})^{+}}\xi(\alpha)$.
	
	Now, let $\mathcal{S}^{+}$ be an arbitrary structural conservative extension of $\mathcal{S}$ in $\Lan^{+}$. Suppose that $X\vdash_{(\mathcal{S})^{+}}\alpha$. Then there is a set
	$Y\cup\lbrace\beta\rbrace\subseteq\Forms$ and an $\Lan^{+}$-substitution $\sigma$ such that $Y\vdash_{\mathcal{S}}\beta$, $\sigma(Y)\subseteq X$ and $\sigma(\beta)=\alpha$. By premise, $Y\vdash_{\mathcal{S}^{+}}\beta$. Since 
	$\mathcal{S}^{+}$ is structural, we also have that $\sigma(Y)\vdash_{\mathcal{S}^{+}}\sigma(\beta)$. This implies that $X\vdash_{\mathcal{S}^{+}}\alpha$.
	
	Finally, the last claim of the proposition follows from the conservativeness of $(\mathcal{S})^{+}$ over $\mathcal{S}$.
\end{proof}

In the next proposition, we give a sufficient condition for $(\mathcal{S})^{+}$ to be a consequence relation.

\begin{prop}\label{P:(S)^{+}-existence}
	Let $\Lan$ be a language with $\kappa=\card{\Var_{\Lan}}\ge\aleph_{0}$ and $\Lan^{+}$ be a primitive extension of $\Lan$. Assume that $\mathcal{S}$ is a structural abstract logic in $\Lan$. Then the relation $\vdash_{(\mathcal{S})^{+}}$ is a consequence relation in $\Lan^{+}$. Moreover, the abstract logic $(\mathcal{S})^{+}$ is $\kappa$-compact; further, if $\mathcal{S}$ is finitary, then  $(\mathcal{S})^{+}$ is also finitary. 
\end{prop}
\begin{proof}
	First, we have to show that $\vdash_{(\mathcal{S})^{+}}$ satisfies the properties of (a)--(c) of Definition~\ref{D:consequnce-relation-single}.
	Leaving the check of the properties (a)--(b) to the reader, we turn to the property (c). (Exercise~\ref{section:extensions}.\ref{EX:(S)_plus})
	
	Assume that for a set $X\cup Y\cup Z\cup\lbrace\alpha\rbrace\subseteq\Forms_{\Lan^{+}}$,
	\[
	\begin{array}{cl}
	(\text{P1}) &X\vdash_{(\mathcal{S})^{+}}\gamma,~\text{for all $\gamma\in Y$};\\
	(\text{P2}) &Y,Z\vdash_{(\mathcal{S})^{+}}\alpha.
	\end{array}
	\]
	
	In virtue of (P2) and  Definition~\ref{D:(S)^+} (further in this proof, simply definition),  there are sets $Y_0$ and $Z_0$ with $Y_{0}\cup Z_{0}\cup\lbrace\beta_{0}\rbrace\subseteq\Forms$ and an $\Lan^{+}$-substitution $\sigma_0$ such that $\sigma_0(Y_0)\subseteq Y$, $\sigma_0(Z_0)\subseteq Z$, $\sigma_0(\beta_{0})=\alpha$ and $Y_0, Z_0\vdash_{\mathcal{S}}\beta_{0}$. 
	
	We denote:
	\[
	\lbrace\gamma_i\rbrace_{i\in I}:=\sigma_0(Y_0).
	\]
	
	In virtue of (P1) and definition, for each $i\in I$, there is a set $X_{i}\cup\lbrace\beta_i\rbrace\subseteq\Forms$ and an $\Lan^{+}$-substitution $\sigma_i$ such that $\sigma_i(X_i)\subseteq X$, $\sigma_i(\beta_i)=\gamma_i$ and
	$X_i\vdash_{\mathcal{S}}\beta_i$.
	
	Now the assumptions (P1)--(P2) can be refined as follows.
	\[
	\begin{array}{cl}
	(\text{P1}^{\ast}) &\bigcup_{i\in I}\sigma_i(X_i)\vdash_{(\mathcal{S})^{+}}\gamma_i,
	~\text{for all $i\in I$};\\
	(\text{P2}^{\ast}) &\lbrace\gamma_i\rbrace_{i\in I},\sigma_0(Z_0)\vdash_{(\mathcal{S})^{+}}\alpha.
	\end{array}
	\]
	
	We denote:
	\[
	X^{\ast}:=\bigcup_{i\in I}\sigma_i(X_i)~\text{and}~Z^{\ast}:=\sigma_0(Z_0).
	\]
	
	We aim to show that $X^{\ast}, Z^{\ast}\vdash_{(\mathcal{S})^{+}}\alpha$. Since $X^{\ast}\subseteq X$ and $Z^{\ast}\subseteq Z$, we will thus obtain that $X,Z\vdash_{(\mathcal{S})^{+}}\alpha$, for we have already proved that $\vdash_{(\mathcal{S})^{+}}$ is monotone.
	
	Let us make notes about the cardinalities of the $\Lan^+$-variables of some formula sets. (This is where we use that the extension $\Lan^{+}$ is primitive.)
	\begin{itemize}
		\item $\card{\Var(\lbrace\gamma_i\rbrace_{i\in I})}\le\card{\Var_{\Lan}}$;
		\item $\card{\Var(Z^\ast)}=\card{\Var(\sigma_0(Z_0))}\le\card{\Var_{\Lan}}$;
		\item $\card{\Var(X^\ast)}=\card{\Var(\bigcup_{i\in I}\sigma_i(X_i))}\le\card{\Var_{\Lan}}$.
	\end{itemize}
	
	Thus we have: $\card{\Var(\lbrace\gamma_i\rbrace_{i\in I}\cup X^{\ast}\cup Z^{\ast}\cup\lbrace\alpha\rbrace)}\le\card{\Var_{\Lan}}$. This allows us to claim that there is a one-one map
	\[
	f:\Var(\lbrace\gamma_i\rbrace_{i\in I}\cup X^{\ast}\cup Z^{\ast}\cup\lbrace\alpha\rbrace)
	\stackrel{\text{1--1}}{\longrightarrow}\Var_{\Lan}.
	\]
	
	For simplicity, we denote:
	\[
	\Var^{\ast}:=\Var(\lbrace\gamma_i\rbrace_{i\in I}\cup X^{\ast}\cup Z^{\ast}\cup\lbrace\alpha\rbrace).
	\]
	
	Next we define two $\Lan^+$-substitutions. Let $q$ be an arbitrary fixed $\Lan$-variable.
	Then, we define:
	\[
	\mu(p):= \begin{cases}
	\begin{array}{cl}
	f(p) &\text{if $p\in\Var^\ast$}\\
	q &\text{otherwise};
	\end{array}
	\end{cases}
	\]
	\[
	\nu(p):=\begin{cases}
	\begin{array}{cl}
	f^{-1}(p) &\text{if $p\in f(\Var^\ast)$}\\
	p &\text{otherwise}.
	\end{array}
	\end{cases}
	\]
	
	We observe the following:
	\begin{itemize}
		\item $\mu\circ\sigma_i$, for all $i\in I$, and $\mu\circ\sigma_0$ are $\Lan$-substitutions;
		\item $\nu\circ\mu\circ\sigma_i(X_i)=\sigma_i(X_i)$ and $\nu\circ\mu(\gamma_i)=\gamma_i$, for all $i\in I$, $\nu\circ\mu\circ\sigma_0(Z_0)=\sigma_0(Z_0)$ and 
		$\nu\circ\mu\circ\sigma_0(\beta_0)=\sigma_0(\beta_0)$;
		\item thus $\nu(\bigcup_{i\in I}\mu\circ\sigma_i(X_i))=\nu(\mu(X^\ast))=X^\ast$, $\nu(\mu\circ\sigma_0(Z_0))\nu(\mu(Z^\ast))=Z^\ast$ and $\nu(\mu\circ\sigma_0(\beta_0))=\nu(\mu(\alpha))=\alpha$. And, by definition, we conclude that $X^{\ast},Z^{\ast}\vdash_{(\mathcal{S})^{+}}\alpha$.
	\end{itemize}
	
	From this, since $\mathcal{S}$ is structural, we obtain:
	\begin{itemize}
		\item $\mu\circ\sigma_i(X_i)\vdash_{\mathcal{S}}\mu\circ\sigma_i(\beta_i)$, for each $i\in I$, which implies $\mu(\bigcup_{i\in I}\sigma_i(X_i))\vdash_{\mathcal{S}}\mu(\gamma_i)$, for each $i\in I$;
		\item $\lbrace\mu(\gamma_i)\rbrace_{i\in I},\mu\circ\sigma_0(Z_0)\vdash_{\mathcal{S}}\mu\circ\sigma_0(\beta_{0})$.
	\end{itemize}
	Then, from the two last assertions and the transitivity of $\mathcal{S}$, we derive:
	$\mu(X^{\ast}),\mu(Z^{\ast})\vdash_{\mathcal{S}}\mu(\alpha)$. 
\end{proof}

\paragraph{Exercises~\ref{section:extensions}}
\begin{enumerate}
	\item \label{EX:intersection-conservative-extensions} Prove Proposition~\ref{P:intersection-conservative-extensions}
	\item \label{EX:S^+-conservativity}Prove the equivalence~\eqref{E:S^+-conservativity}, that is for any set $X\cup\lbrace\beta\rbrace\subseteq\Forms$,
	\[
	X\models_{\langle\FormAl[\alpha]^{+},\Sigma_{\mathcal{S}}\rangle}\beta
	~\Longleftrightarrow~X\models_{\langle\FormAl,\Sigma_{\mathcal{S}}\rangle}\beta.
	\]

	(\textit{Hint}: Show that for any valuation $v$ in $\FormAl$, there is a valuation $v^{\ast}$ in $\FormAl[\alpha]^{+}$ such that $v^{\ast}\!\!\upharpoonright\!\!\Forms=v$.)
	\item\label{EX:(S)_plus} Let $\Lan^{+}$ be a primitive extensions of a language $\Lan$ and let $\mathcal{S}$ be an abstract logic in $\Lan$. Show that the relation $\vdash_{(\mathcal{S})^{+}}$ of Definition~\ref{D:(S)^+} satisfies the properties (a)--(b) of Definition~\ref{D:consequnce-relation-single}.
	\item \label{EX:kappa-compactness_(S)-plus} Prove that $(\mathcal{S})^{+}$ in Proposition~\ref{P:(S)^{+}-existence} is $\kappa$-compact.
\end{enumerate}

\chapter[Matrix Consequence]{Matrix Consequence}\label{chapter:matrix-consequence}	

\section{Single-matrix consequence}\label{section:single-matrix-consequence}
Two circumstances draw our attention to a single matrix consequence. The first one is Lindenbaum's theorem (Proposition~\ref{P:lindenbaum-theorem}) which provides a universal way of definition of a weakly adequate logical matrix, namely the Lindenbaum matrix, which being defined for a given consequence relation, validates all of the theses of this consequence relation and only them. The other circumstance is related to the consequence relation $\vdash_2$ (Section~\ref{section:consequence-defining}) which turned out to be equal to the single-matrix relation $\models_{\textbf{B}_{2}}$. These circumstances lead to the the following concept.

We recall (Section~\ref{section:consequence-defining}) that
a logical matrix $\mat{M}$ is called adequate for an abstract logic $\mathcal{S}$ if $\vdash_{\mathcal{S}}\,=\,\models_{\textbf{M}}$; that is, for any set $X$ of $\Lan$-formulas and any $\Lan$-formula $\alpha$,
\[
X\vdash_{\mathcal{S}}\alpha\Longleftrightarrow X\models_{\textbf{M}}\alpha.
\]

In connection with the notion of an adequate matrix, two questions arise. Is it true that any abstract logic has an adequate matrix? And if the answer to the first question is negative, what are the properties of an abstract logic that make it have an adequate matrix.

The first question can be answered in the negative as follows. Suppose $\mathcal{S}$ is a nontrivial abstract logic in $\Lan$ with $\card{\Var_{\Lan}}\ge\aleph_{0}$, which has a nonempty consistent set, say $X_0$, and such that $\bm{T}_{\mathcal{S}}\neq\emptyset$. Now, let
$\vdash_{\mathcal{S}}^{\circ}$ be a relation obtained from $\vdash_{\mathcal{S}}$ according to~\eqref{E:nonempty-consequence}. In virtue of Proposition~\ref{P:nonempty-consequence}, $\vdash_{\mathcal{S}}^{\circ}$ is a consequence relation. We claim that there is no adequate matrix for $\vdash_{\mathcal{S}}^{\circ}$. (Such a relation $\vdash_{\mathcal{S}}^{\circ}$ can be exemplified, e.g., by $\models_{\textbf{B}_2}^{\circ}$ where $\models_{\textbf{B}_2}$ is a single-matrix consequence of Section~\ref{section:consequence-defining}.)

Indeed, for contradiction, assume that a matrix $\mat{M}=\langle\alg{A}, D\rangle$ is adequate for $\vdash_{\mathcal{S}}^{\circ}$. Since $X_0$ is consistent w.r.t. $\mathcal{S}$, it is also consistent w.r.t. $\vdash_{\mathcal{S}}^{\circ}$.
This and the assumption that $X_0$ is nonempty imply that $D\neq\emptyset$.
Assume that $a\in D$.
By premise, there is a thesis $\alpha\in\bm{T}_{\mathcal{S}}$. Let us take a variable $p\in\Var_{\Lan}\setminus\Var(\alpha)$. By definition, $\emptyset\not\vdash_{\mathcal{S}}^{\circ}\alpha$ and $p\vdash_{\mathcal{S}}^{\circ}\alpha$. The former implies that there is a valuation $v$ in $\alg{A}$ such that $v[\alpha]\notin D$. Now we define a valuation:
\[
v^{\prime}[q]:=\begin{cases}
\begin{array}{cl}
v[q] &\text{if $q\neq p$}\\
a &\text{if $q=p$}.
\end{array}
\end{cases}
\]
It is clear that $v^{\prime}[\alpha]=v[\alpha]$ and hence $p\not\models_{\textbf{M}}\alpha$.\\

Trying to find a characteristic for abstract logic to have an adequate matrix, we consider any logical matrix $\mat{M}:=\langle\alg{A},D\rangle$ of type $\Lan$.

Now, let $\alpha, \beta$ and $\gamma$ be $\Lan$-formulas with $(\Var(\alpha)\cup\Var(\beta))\cap\Var(\gamma)=\emptyset$. Assume that for some valuation $v_0$, $v_0(\gamma)\in D$. (This assumption is equivalent to the condition that the set $\lbrace\gamma\rbrace$ is consistent w.r.t. the consequence relation $\models_{\textbf{M}}$.) We prove that
\[
\gamma,\alpha\models_{\textbf{M}}\beta~~\textit{implies}~~\alpha\models_{\textbf{M}}\beta.
\]

Indeed, for contrapositive, suppose that $\alpha\not\models_{\textbf{M}}\beta$;
that is, for some valuation $v$, $v(\alpha)\in D$, but $v(\beta)\notin D$.

Let us define a valuation
\[
v_1(p)=\begin{cases}
\begin{array}{cl}
v_0[p] &\text{if $p\in\Var(\gamma)$}\\
v[p] &\text{otherwise}.
\end{array}
\end{cases}
\]

We note that $v_1[\gamma]=v_0[\gamma]$, $v_1[\alpha]=v[\alpha]$ and  $v_1[\beta]=v[\beta]$. This implies that $\gamma,\alpha\not\models_{\textbf{M}}\beta$.\\

The last consideration leads to the following definition.
\begin{defn}[uniform abstract logic]\label{D:uniform-algebraic-logic}
	An abstract logic $\mathcal{S}$ is said to be \textbf{uniform} if 	for any set
	$X\cup Y\cup\lbrace\alpha\rbrace$ of $\Lan$-formulas with $\Var(X\cup\lbrace\alpha\rbrace)\cap\Var(Y)=\emptyset$ and {\em$\textbf{Cn}_{\mathcal{S}}(Y)\neq\Forms_{\mathcal{L}}$}, if $X,Y\vdash_{\mathcal{S}}\alpha$, then $X\vdash_{\mathcal{S}}\alpha$.
\end{defn}

We note that the consequence relation $\models_{\textbf{B}_{2}}^{\circ}$ mentioned above is finitary but not uniform. (Exercise~\ref{section:single-matrix-consequence}.\ref{EX:b-two-circle})
The importance of the last definition becomes clear from Proposition~\ref{P:los-suszko} below. We prove this proposition by means of two lemmas.

\begin{lem}\label{L:los-suszko-1}
	If an abstract logic $\mathcal{S}$	has an adequate matrix, then $\mathcal{S}$ is uniform.
\end{lem}
\begin{proof}
	Let a matrix {\em$\mat{M}=\langle\alg{A},D\rangle$} be adequate for $\mathcal{S}$.
	Suppose $X,Y\vdash_{\mathcal{S}}\alpha$, where $\Var(X,\alpha)\cap\Var(Y)=\emptyset$
	and $\textbf{Cn}_{\mathcal{S}}(Y)\neq\Forms_{\mathcal{L}}$. We will show that $X\models_{\textbf{M}}\alpha$, which, in virtue of the adequacy of $\mat{M}$ for $\mathcal{S}$, will in turn imply that $X\vdash_{\mathcal{S}}\alpha$.
	
	For contradiction, assume that for some valuation $v$ in $\alg{A}$, $v[X]\subseteq D$ but $v[\alpha]\notin D$. Since $Y$ is consistent w.r.t. $\mathcal{S}$, there is a valuation $v^{\prime}$ such that $v^{\prime}[Y]\subseteq D$. Then, we define a valuation $v^{\prime\prime}$ as follows:
	\[
	v^{\prime\prime}[p]:=\begin{cases}
	\begin{array}{cl}
	v[p] &\text{if $p\in\Var(X,\alpha)$}\\
	v^{\prime}[p] &\text{if $p\in\Var(Y)$}.
	\end{array}
	\end{cases}
	\]
	
	It must be clear that $v^{\prime\prime}[\beta]=v[\beta]$, whenever $\beta\in X\cup\lbrace\alpha\rbrace$, and $v^{\prime\prime}[\beta]=v^{\prime}[\beta]$, if $\beta\in Y$. Therefore, $X,Y\not\models_{\textbf{M}}\alpha$; that is $X,Y\not\vdash_{\mathcal{S}}\alpha$. A contradiction.
\end{proof}

As preparation to the next proposition we observe the following.
\begin{lem}\label{L:(S)_plus-uniformity}
	Let $\mathcal{S}$ be a uniform structural abstract logic in a language $\Lan$ with $\card{\Var_{\Lan}}\ge\aleph_{0}$ and $\Lan^{+}$ be a primitive extension of $\Lan$. Then the abstract logic $(\mathcal{S})^{+}$ is also uniform.
\end{lem}
\begin{proof}
	Suppose for a set $X\cup Y\cup\lbrace\alpha\rbrace\subseteq\Forms_{\mathcal{L}^{+}}$,
	$X,Y\vdash_{(\mathcal{S})^{+}}\alpha$, providing that $\Var(X\cup\lbrace\alpha
	\rbrace)\cap\Var(Y)=\emptyset$ and $Y$ is consistent w.r.t. $(\mathcal{S})^{+}$.
	Then, by definition, there is $X_0\cup Y_0\cup\lbrace\beta\rbrace\subseteq\Forms_{\mathcal{L}}$ and an $\Lan^{+}$-substitution $\sigma$ such that $\sigma(X_0)\subseteq X$, $\sigma(Y_0)\subseteq Y$, $\sigma(\beta)=\alpha$ and $X_0,Y_0\vdash_{\mathcal{S}}\beta$. It is clear that
	$\Var(X_0\cup\lbrace\beta\rbrace)\cap\Var(Y_0)=\emptyset$. Also, $Y_0$ is consistent w.r.t. $\mathcal{S}$, for if it were otherwise, in virtue of conservativity of $(\mathcal{S})^{+}$ over $\mathcal{S}$, we would have $\sigma(Y_0)\vdash_{(\mathcal{S})^{+}}\gamma$, for any arbitrary $\Lan^{+}$-formula $\gamma$. This would imply the inconsistency of $Y$ in $(\mathcal{S})^{+}$.
	
	Since $\mathcal{S}$ is uniform, $X_0\vdash_{\mathcal{S}}\beta$ and hence $X_0\vdash_{(\mathcal{S})^{+}}\beta$, for $(\mathcal{S})^{+}$ is a conservative extension of $\mathcal{S}$. And since $(\mathcal{S})^{+}$ is structural (Proposition~\ref{P:(S)_plus}), $\sigma(X_0)\vdash_{(\mathcal{S})^{+}}\alpha$ and hence
	$X\vdash_{(\mathcal{S})^{+}}\alpha$.
\end{proof}

\begin{prop}[Lo\'{s}-Suszko-W\'{o}jcicki theorem]\label{P:los-suszko}
	Let $\mathcal{S}$ be a finitary structural abstract logic in a language $\Lan$ with $\card{\Var_{\Lan}}\ge\aleph_{0}$. Then $\mathcal{S}$ has an adequate matrix if, and only if, $\mathcal{S}$ is uniform.
\end{prop}
\begin{proof}
	The ``only-if'' part follows from Lemma~\ref{L:los-suszko-1}. So we continue with the proof of the ``if'' part.
	
	We distinguish two cases: $\mathcal{S}$ is trivial and $\mathcal{S}$ is nontrivial. If the former is the case, then the matrix $\langle\Forms_{\mathcal{L}},\Forms_{\mathcal{L}}\rangle$ is adequate for $\mathcal{S}$.
	
	Now, we assume that $\mathcal{S}$ is nontrivial; that is, the set $\theoryC$ of consistent sets w.r.t. $\mathcal{S}$ is nonempty. 
	
	We introduce for each $X\in\theoryC$ and each $p\in\Var_{\Lan}$, a new sentential variable $p_X$ so that for any $p,q\in\Var_{\Lan}$ and any $X,Y\in\theoryC$,
	\[
	p_X=q_Y\Longleftrightarrow p=q~\text{and}~X=Y.
	\]
	Accordingly, for an arbitrary $X\in\theoryC$, we denote:
	\[
	\Var_X=\set{p_X}{p\in\Var_{\Lan}}.
	\]
	Thus for any $X,Y\in\theoryC$,
	\begin{equation}\label{E:implication-1}
		X\neq Y\Longrightarrow\Var_X\cap\Var_Y=\emptyset.
	\end{equation}
	
	A language $\Lan^{+}$ is defined as a primitive extension of $\Lan$, by adding the new sentential variables to $\Var_{\Lan}$. Then, we introduce $(\mathcal{S})^{+}$ according to Definition~\ref{D:(S)^+}.
	
	Next, for each $X\in\theoryC$, we define two $\Lan^{+}$-substitutions:
	\[
	\begin{array}{l}
	\sigma_X:p\mapsto p_{X},~\text{for any $p\in\Var_{\Lan}$},~\text{and}~\sigma_X:p_Y\mapsto p_Y,
	~\text{for any $p_Y\in\Var_{\Lan^{+}}\setminus\Var_{\Lan}$};\\
	\xi_X: p_X\mapsto p,~\text{for any $p_X\in\Var_X$ and $\xi_X:q\mapsto q$, for any $q\in\Var_{\Lan^{+}}\setminus\Var_X$}.
	\end{array}
	\]
	
	We observe: for any $X\in\theoryC$,
	\begin{equation}\label{E:zig-zag}
		\xi_{X}\circ\sigma_X(\beta)=\beta,~\text{for all $\beta\in\Forms_{\mathcal{L}}$}.
	\end{equation}
	
	Then, we make two observations.
	First, since for any $X\in\theoryC$, $\Var(\sigma_X(X))\subseteq\Var_X$, \eqref{E:implication-1} yields that for any $X,Y\in\theoryC$,
	\begin{equation}\label{E:implication-2}
		X\neq Y\Longrightarrow\Var(\sigma_X(X))\cap\Var(\sigma_Y(Y))=\emptyset.
	\end{equation}
	
	Second, 
	\begin{equation}\label{E:second-observation}
		\textit{for each $X\in\theoryC$, the set $\sigma_X(X)$ is consistent w.r.t. $(\mathcal{S})^{+}$.}
	\end{equation}
	
	Indeed, for contradiction, assume that $\sigma_{X}(X)$ is inconsistent. Let $\beta$ be any formula of $\Forms_{\mathcal{L}}$. By assumption, $\sigma_{X}(X)\vdash_{(\mathcal{S})^{+}}\sigma_{X}(\beta)$. Since $(\mathcal{S})^{+}$ is structural (Proposition~\ref{P:(S)^+}), we have that $\xi_{X}(\sigma_{X}(X))\vdash_{(\mathcal{S})^{+}}\xi_{X}(\sigma_{X}(\beta))$ and hence, in virtue of \eqref{E:zig-zag}, $X\vdash_{(\mathcal{S})^{+}}\beta$. Since $(\mathcal{S})^{+}$ is conservative over $\mathcal{S}$, we obtain that $X\vdash_{\mathcal{S}}\beta$. A contradiction.
	
	Now, we define:
	\[
	\mat{M}:=\langle\Forms_{\mathcal{L}^{+}},\textbf{Cn}_{(\mathcal{S})^{+}}(\bigcup
	\set{\sigma_X(X)}{X\in\theoryC})\rangle.
	\]
	
	We aim to show that $\mat{M}$ is adequate for $\mathcal{S}$; that is, for any set $Y\cup\lbrace\alpha\rbrace\subseteq\Forms_{\mathcal{L}}$,
	\[
	Y\vdash_{\mathcal{S}}\alpha\Longleftrightarrow Y\models_{\textbf{M}}\alpha.
	\]
	
	Suppose $Y\vdash_{\mathcal{S}}\alpha$. Since $(\mathcal{S})^{+}$ is conservative over $\mathcal{S}$, this implies that $Y\vdash_{(\mathcal{S})^{+}}\alpha$. Now, let $\sigma$ be any $\Lan^{+}$-substitution.
	Assume that $\sigma(Y)\subseteq\textbf{Cn}_{(\mathcal{S})^{+}}(\bigcup
	\set{\sigma_X(X)}{X\in\theoryC})$. Since $(\mathcal{S})^{+}$ is structural, we have: $\sigma(\alpha)\in\textbf{Cn}_{(\mathcal{S})^{+}}(\sigma(Y))$ and hence
	$\sigma(\alpha)\in\textbf{Cn}_{(\mathcal{S})^{+}}(\bigcup
	\set{\sigma_X(X)}{X\in\theoryC})$.
	
	Conversely, suppose that $Y\not\vdash_{\mathcal{S}}\alpha$. We denote:
	\[
	X_0:=\textbf{Cn}_{\mathcal{S}}(Y).
	\]
	Thus $\alpha\notin X_0$ and hence $X_0\in\theoryC$. The latter implies that
	$\sigma_{X_0}(X_0)\subseteq\textbf{Cn}_{(\mathcal{S})^{+}}(\bigcup
	\set{\sigma_X(X)}{X\in\theoryC})$.
	
	For contradiction, we assume that
	\[
	\sigma_{X_0}(\alpha)\in\textbf{Cn}_{(\mathcal{S})^{+}}(\bigcup
	\set{\sigma_X(X)}{X\in\theoryC}). \tag{$\ast$}
	\]
	In virtue of finitariness of $(\mathcal{S})^{+}$ (Proposition~\ref{P:(S)^{+}-existence}), there is a finite number of sets $X_1,\ldots, X_n\in\theoryC$ such that 
	\[
	\sigma_{X_0}(\alpha)\in\textbf{Cn}_{(\mathcal{S})^{+}}(\sigma_{X_0}(X_0)\cup\sigma_{X_1}
	(X_1)\cup\ldots\cup\sigma_{X_n}(X_n)).
	\]
	Since, according to~\eqref{E:second-observation}, $\sigma_{X_n}(X_n)$ is consistent and
	\[
	\Var(\sigma_{X_0}(X_0)\cup\sigma_{X_1}
	(X_1)\cup\ldots\cup\sigma_{X_{n-1}}(X_{n-1}))\cap\Var(\sigma_{X_{n}}(X_{n}))=\emptyset,
	\]
	in virtue of the uniformity of $(\mathcal{S})^{+}$ (Lemma~\ref{L:(S)_plus-uniformity}),
	\[
	\sigma_{X_0}(X_0),\sigma_{X_1}
	(X_1),\ldots\sigma_{X_{n-1}}(X_{n-1})\vdash_{(\mathcal{S})^{+}}\sigma_{X_0}(\alpha)
	\]
	Repeating this argument several times, we arrive at the conclusion that 
	$\sigma_{X_0}(X_0)\vdash_{(\mathcal{S})^{+}}\sigma_{X_0}(\alpha)$. Then, applying the substitution $\xi_{X_0}$, we get $X_0\vdash_{(\mathcal{S})^{+}}\alpha$. Finally, the conservativity of $(\mathcal{S})^{+}$ over $\mathcal{S}$ yields $X_0\vdash_{\mathcal{S}}\alpha$. A contradiction.
\end{proof}

The case of nonfinitary abstract logics needs an additional investigation.
\begin{defn}[couniform abstract logic]
	An abstract logic $\mathcal{S}$ in a language $\Lan$ is called couniform if for any collection $\lbrace X_i\rbrace_{i\in I}$ of formula sets with $\Var(X_i)\cap\Var(X_j)=\emptyset$, providing that $i\neq j$, and $\Var(\bigcup_{i\in I}\lbrace X_i\rbrace) \neq\Var_{\Lan}$, there is at least one $X_i$ which is inconsistent, whenever the union $\bigcup_{i\in I}\lbrace X_i\rbrace$ is inconsistent.	
\end{defn}

The next lemma shows that couniformity is necessary for an abstract logic to have an adequate matrix.
\begin{lem}\label{L:couniform}
	If and abstract logic $\mathcal{S}$ has an adequate matrix, then $\mathcal{S}$ is couniform.
\end{lem}
\begin{proof}
	Let a matrix $\mat{M}=\langle\alg{A},D\rangle$ be adequate for $\mathcal{S}$. We assume that a family $\lbrace X_i\rbrace_{i\in I}$ of formula sets satisfies all the premises of couniformity. Suppose each $X_i$ is consistent w.r.t. $\mathcal{S}$. Then for each $i\in I$, there is a formula $\alpha_i$ such that $X_i\not\vdash_{\mathcal{S}}\alpha_i$. This implies that there is a valuation $v_i$ in $\alg{A}$ such that $v_i[X_i]\subseteq D$ but
	$v_i[\alpha_i]\notin D$. Let us select either one of $v_i[\alpha_i]$ and denote
	\[
	a:=v_i[\alpha_i].
	\]
	Then, we select any variable $p\in\Var_{\Lan}\setminus\Var(\bigcup_{i\in I}\lbrace X_i\rbrace)$ and for each $v_i$, define:
	\[
	v_{i}^{\ast}[q]:=\begin{cases}
	\begin{array}{cl}
	v_i[q] &\text{if $q\in\Var(\bigcup_{i\in I}\lbrace X_i\rbrace)$}\\
	a &\text{if $q\in\Var_{\Lan}\setminus\Var(\bigcup_{i\in I}\lbrace X_i\rbrace)$}.
	\end{array}
	\end{cases}
	\]
	
	It must be clear that for each $i\in I$, $v_{i}^{\ast}[X_i]\subseteq D$. And since $\Var(X_i)\cap\Var(X_j)=\emptyset$, providing $i\neq j$, we also have that $v_{i}^{\ast}[\bigcup_{i\in I}\lbrace X_i\rbrace]
	\subseteq D$. Thus $\bigcup_{i\in I}\lbrace X_i\rbrace\not\vdash_{\mathcal{S}} p$, that is the set $\bigcup_{i\in I}\lbrace X_i\rbrace$ is consistent.
\end{proof}

The next lemma is rather technical and we will need it for the proof of Proposition~\ref{P:wojcicki} below.
\begin{lem}\label{L:(S)^{+}-couniform}
	Let $\mathcal{S}$ be a couniform structural abstract logic in a language $\Lan$ with $\card{\Var_{\Lan}}\ge\aleph_{0}$ and $\Lan^{+}$ be a primitive extension of $\Lan$. Then the abstract logic $(\mathcal{S})^{+}$ is also couniform.
\end{lem}
\begin{proof}
	Let $\lbrace X_i\rbrace_{i\in I}$ be a nonempty family of nonempty sets of $\Lan^{+}$-formulas, about which we assume that the union $\bigcup_{i\in I}\lbrace X_i\rbrace$ is inconsistent and $\Var(\bigcup_{i\in I}\lbrace X_i\rbrace)\neq\Var_{\Lan^{+}}$. Let us select any $p\in\Var_{\Lan^{+}}\setminus\Var(\bigcup_{i\in I}\lbrace X_i\rbrace)$. Since $\bigcup_{i\in I}\lbrace X_i\rbrace$ is inconsistent, we have that $\bigcup_{i\in I}\lbrace X_i\rbrace\vdash_{(\mathcal{S})^{+}}p$. By definition of $(\mathcal{S})^{+}$ (Definition~\ref{D:(S)^+}), there are a set $Y$ of $\Lan$-formulas, an $\Lan$-variable $q$ and an $\Lan$-substitution $\sigma$ such that
	$\sigma(Y)\subseteq\bigcup_{i\in I}\lbrace X_i\rbrace$, $\sigma(q)=p$ and $Y\vdash_{\mathcal{S}}q$. 
	
	Next we define: for each $i\in I$,
	\[
	Y_i:=\set{\gamma\in Y}{\sigma(\gamma)\in X_i}.
	\]
	
	Firstly, we notice that for each $i\in I$, $\sigma(Y_i)\subseteq X_i$, and $\bigcup_{i\in I}\lbrace Y_i\rbrace\vdash_{\mathcal{S}}q$.
	(The latter is true because $\bigcup_{i\in I}\lbrace Y_i\rbrace=Y$.)
	Secondly, $q\in\Var_{\Lan}\setminus\Var(\bigcup_{i\in I}\lbrace Y_i\rbrace)$. Indeed, if it were the case that $q\in\Var(\bigcup_{i\in I}\lbrace Y_i\rbrace)$, then we would have that $p\in
	\Var(\bigcup_{i\in I}\lbrace X_i\rbrace)$.
	
	Thus we obtain that $\bigcup_{i\in I}\lbrace Y_i\rbrace$ is inconsistent w.r.t. $\mathcal{S}$. Since, by definition, $\mathcal{S}$ is couniform, there is $i_0\in I$ such that $Y_{i_0}$ is inconsistent w.r.t. $\mathcal{S}$. This implies that  Since $(\mathcal{S})^{+}$ is a conservative extension of $\mathcal{S}$, we receive that $Y_{i_0}\vdash_{(\mathcal{S})^{+}}q$. And since 
	$(\mathcal{S})^{+}$ is structural (Proposition~\ref{P:(S)^+}), $\sigma(Y_{i_0})\vdash_{(\mathcal{S})^{+}}p$ and hence $X_{i_0}\vdash_{(\mathcal{S})^{+}}p$. We recall that $p\notin\Var(X_{i_0})$. Therefore, the structurality of $(\mathcal{S})^{+}$ implies that $X_{i_0}$ is inconsistent.
\end{proof}

\begin{prop}[W\'{o}jcicki theorem]\label{P:wojcicki}
	Let $\mathcal{S}$ be a structural abstract logic in a language $\Lan$ with $\card{\Var_{\Lan}}\ge\aleph_{0}$. Then $\mathcal{S}$ has an adequate matrix if, and only if, $\mathcal{S}$ is both uniform and couniform.
\end{prop}
\begin{proof}
	The ``only-if'' part follows from Lemmas~\ref{L:los-suszko-1} and~\ref{L:couniform}. 
	
	The proof of the ``if'' part repeats the steps of the ``if'' part of the proof of the Lo\'{s}-Suszko theorem (Proposition~\ref{P:los-suszko}) up to the assumption $(\ast)$.
	Then, we continue as follows.
	
	First, we notice that, in virtue of~\eqref{E:second-observation} and of the couniformity 
	of $(\mathcal{S})^{+}$ (Lemma~\ref{L:(S)^{+}-couniform}), the set $\bigcup\set{\sigma_{X}(X)}{X\in\theoryC}$ is consistent. Thus, we have:
	\[
	\sigma_{X_0}(X_0),\bigcup\set{\sigma_{X}(X)}{X\in\theoryC~\text{and}~X\neq X_0}
	\vdash_{(\mathcal{S})^{+}}\sigma_{X_0}(\alpha),
	\]
	where the set $\bigcup\set{\sigma_{X}(X)}{X\in\theoryC~\text{and}~X\neq X_0}$ is consistent w.r.t. $(\mathcal{S})^{+}$.
	
	Let us denote
	\[
	Y_0:=\bigcup\set{\sigma_{X}(X)}{X\in\theoryC~\text{and}~X\neq X_0}.
	\]
	With the help of~\eqref{E:implication-1} and~\eqref{E:implication-2}, we conclude that $\Var(\sigma_{X_0}(X_0)\cup\lbrace\sigma_{X_0}(\alpha)\rbrace)\cal\Var(Y_0)=\emptyset$.
	Then, using the uniformity of $(\mathcal{S})^{+}$, we derive that $\sigma_{X_0}(\alpha)
	\vdash_{(\mathcal{S})^{+}}\sigma_{X_0}(\alpha)$. This, in virtue of structurality of $(\mathcal{S})^{+}$ (Proposition~\ref{P:(S)^+}), implies that
	$\xi_{X_0}(\sigma_{X_0}(\alpha))
	\vdash_{(\mathcal{S})^{+}}\xi_{X_0}(\sigma_{X_0}(\alpha))$ and hence, according to~\eqref{E:zig-zag}, $X_0\vdash_{(\mathcal{S})^{+}}\alpha$. Because of the conservativeness of $(\mathcal{S})^{+}$ over $\mathcal{S}$, the latter leads to $X_0\vdash_{\mathcal{S}}\alpha$. A contradiction.
\end{proof}

\paragraph{Exercises~\ref{section:single-matrix-consequence}}
\begin{enumerate}
	\item\label{EX:b-two-circle} Prove that the relation $\models_{\textbf{B}_{2}}^{\circ}$ obtained from $\models_{\textbf{B}_2}$ according to \eqref{EX:nonempty-consequence} is finitary but not uniform.
\end{enumerate}

\section{Finitary matrix consequence}\label{section:finitary-matrix-consequence}
Given a (nonempty) class $\mathcal{M}$ of $\Lan$-matrices, an abstract logic $\mathcal{S}_{\mathcal{M}}$, according to Proposition~\ref{P:matrix-con-is-con-relation},
is structural and, according to Corollary~\ref{C:S-matrix-completeness}, is determined by the class of all $\mathcal{S}_{\mathcal{M}}$-models. The last class definitely contains $\mathcal{M}$. For the purpose of this section, it will not be an exaggeration to assume that $\mathcal{M}$ coincides with the class of all $\mathcal{S}_{\mathcal{M}}$-models.
Under this assumption, the question arises: What characteristics should $\mathcal{M}$ have in order for $\mathcal{S}_{\mathcal{M}}$ to be finitary? Proposition~\ref{P:ultraclosedness} below proposes an answer to this question.

However, first we discuss the property that the matrix consequence of any finite matrix is finitary. To prove this, we use topological methods discussed in Section~\ref{section:topology}. 

Let $\mat{M}=\langle\alg{A},D\rangle$ be a finite logical matrix of type $\Lan$ and let $\valA:=\alg{A}^{\Var_{\Lan}}$ (the set of all valuations in algebra $\alg{A}$). We regard $|\alg{A}|$ as a finite topological space with discrete topology and $\valA$ as the cartesian power of the space $|\alg{A}|$ relative to the product topology. According to~Corollary~\ref{C:product-finite-discrete-spaces}, the space $\valA$ is compact.

For any formula $\alpha\in\Forms_{\mathcal{L}}$, we denote:
\[
\valA(\alpha):=\set{v\in\valA}{v[\alpha]\in D}.
\]

We aim to show that each $\valA(\alpha)$ is open and closed in $\valA$. The idea is that, since an arbitrary formula $\alpha$ contains finitely many sentential variables and the set $|\alg{A}|$ is finite, there are only finitely many restricted assignments that validate $\alpha$ and there are finitely many restricted assignments that refute it in $\mat{M}$.

We show that each $\valA(\alpha)$ is an open set in the product topology. The closedness of it, that is the openness of its complement, is similar.

Given an arbitrary formula $\alpha$, we define a binary relation on $\valA(\alpha)$ as follows:
\[
v\equiv w~\stackrel{\text{df}}{\Longleftrightarrow}~v(p)=w(p),~\text{for any $p\in\Var(\alpha)$}.
\] 

It is clear that the relation $\equiv$ is an equivalence. For any $v\in\valA(\alpha)$, we denote:
\[
|v|:=\set{w\in\valA(\alpha)}{w\equiv v}. \tag{the equivalence class generated by $v$}
\]

It is well know that $\set{|v|}{v\in\valA(\alpha)}$ is a partition of $\valA(\alpha)$. If we show that each $|v|$ is open in the product topology, so will be $\valA(\alpha)$.

We observe that
\[
|v|=\!\!\!\prod_{~p\in\Var_{\Lan}}Z_p,
\]
where
\[
Z_p=\begin{cases}
\begin{array}{cl}
\lbrace v(p)\rbrace &\text{if $p\in\valA(\alpha)$}\\
|\alg{A}| &\text{otherwise}.
\end{array}
\end{cases}
\]

According to the description~\eqref{E:representation} of the product topology in Section~\ref{section:topology}, this demonstrates that each $|v|$ is open. In a similar manner, we prove that each
\[
\valA\setminus\valA(\alpha)=\set{v\in\valA}{v[\alpha]\notin D}
\]
is also open.\\

Now, assume that for a set $X\cup\lbrace\alpha\rbrace\subseteq\Forms_{\mathcal{L}}$,
$X\models_{\textbf{M}}\alpha$. This implies (is even equivalent to) that $\bigcap_{\beta\in X}\valA(\beta)\subseteq\valA(\alpha)$. The last inclusion in turn is equivalent to the following:
\begin{equation}\label{E:open-cover}
	\valA(\alpha)\cup\bigcup_{\beta\in X}(\valA\setminus\valA(\beta))=
	\valA.
\end{equation}

This means that $\lbrace\valA(\alpha\rbrace\cup\lbrace(\valA\setminus\valA(\beta))
\rbrace_{\beta\in X}$ is an open cover of $\valA$. Since the space 
$\valA$ is compact, it has a finite subcover 
$\lbrace\valA(\alpha\rbrace\cup\lbrace(\valA\setminus\valA(\beta))
\rbrace_{\beta\in Y}$, for some $Y\Subset X$. The latter is equivalent to $Y\models_{\textbf{M}}\alpha$. 

Thus we have proved the following.
\begin{prop}
	A single matrix consequence relative to a finite matrix is finitary.
\end{prop}

\begin{cor}\label{C:finitariness-M-consequence}
	Let $\mathcal{M}$ be a nonempty finite family of finite matrices of type $\Lan$. Then the $\mathcal{M}$-consequence is finitary.
\end{cor}
\noindent\textit{Proof}~is left to the reader. (Exercise~\ref{section:single-matrix-consequence}.\ref{EX:finitariness-M-consequence})\\

Now we turn to a more general case.

\begin{prop}[Bloom-Zygmunt criterion]\label{P:ultraclosedness}
	Let $\mathcal{M}$ be the class of all models of an abstract logic $\mathcal{S}$ in $\Lan$. Then $\mathcal{S}$ is finitary if, and only if, the class $\mathcal{M}$ is closed w.r.t. ultraproducts. 
\end{prop}
\begin{proof}
	Suppose $\mathcal{S}$ is finitary. Then, according to Corollary~\ref{C:logic-by-modus-rules-criterion}, there is a set $\mathcal{R}$ of modus rules such that $\mathcal{S}=\mathcal{S}_{\mathcal{R}}$. Thus, $\mathcal{M}$ is the class of all $\mathcal{S}_{\mathcal{R}}$-models. Then, in virtue of Corollary~\ref{C:S-models-modus-rules}, $\mathcal{M}$ is closed under ultraproducts.
	
	Next, let the class $\mathcal{M}$ be closed w.r.t. ultraproducts. For contradiction, assume that $\mathcal{S}$ is not finitary. That is, for some set $X\cup\lbrace\alpha\rbrace\subseteq\Forms_{\mathcal{L}}$, $\alpha\in\ConS{X}$ but for any $Y\Subset X$, $\alpha\notin Y$. 
	
	We denote:
	\[
	I:=\set{Y}{Y\Subset X}
	\]
	and for any $Y\in I$,
	\[
	\langle Y\rangle:=\set{Z\in I}{Y\subseteq Z}.
	\]
	
	We observe that for any $Y_1,\ldots,Y_n\in I$,
	\[
	\langle Y_1\rangle\cap\ldots\cap\langle Y_n\rangle=\langle Y_1\cap\ldots\cap Y_n\rangle.
	\]
	That is, the set $\set{\langle Y\rangle}{Y\in I}$ has the finite intersection property. According to the Ultrafilter Theorem (\cite{chang-keisler1990}, proposition 4.1.3), there is an ultrafilter $\nabla$ over $I$ such that
	\begin{equation}\label{E:ultrafilter-1}
		\set{\langle Y\rangle}{Y\in I}\subseteq\nabla.
	\end{equation}
	
	Given an $\Lan$-formula $\beta\in X$, we note that
	\[
	\langle\lbrace\beta\rbrace\rangle=\set{Y\in I}{\beta\in Y}\subseteq\set{Y\in I}{\beta\in\ConS{Y}}.
	\]
	This, in view of~\eqref{E:ultrafilter-1}, implies that for arbitrary formula $\beta\in X$,
	\begin{equation}\label{E:ultrafilter-2}
		\set{Y\in I}{\beta\in\ConS{Y}}\in\nabla.
	\end{equation}
	
	Using the notation 	$\Lin_{\,\mathcal{S}}[Y]:=\langle\FormAl,\ConS{Y}\rangle$ (Section~\ref{section:con-via-matrices}), we define:
	\[
	\mat{M}:=\prod_{F}\Lin_{\,\mathcal{S}}[Y].
	\]
	Thus
	\[
	\mat{M}=\langle(\FormAl)^{I}\slash\nabla,D\rangle,
	\]
	where a set $D\subseteq |(\FormAl)^{I}\slash\nabla|$ and such that for any $\theta\in|(\FormAl)^{I}|$,
	\begin{equation}\label{E:ultrafilter-3}
		\theta\slash\nabla\in D~\Longleftrightarrow~\set{Y\in I}{\theta_{Y}\in\ConS{Y}}\in\nabla.
	\end{equation}
	
	For any $\Lan$-formula $\beta$, we define:
	\[
	\overline{\beta}:I\longrightarrow\lbrace\beta\rbrace
	\]
	and then, for any $\beta\in\Forms_{\mathcal{L}}$,
	\[
	v:\beta\mapsto\overline{\beta}\slash\nabla.
	\]
	
	It must be clear that $v:\FormAl\longrightarrow(\FormAl)^{I}\slash\nabla$ is a homomorphism and, hence, can be regarded as an assignment of the $\Lan$-formulas in $(\FormAl)^{I}\slash\nabla$.
	
	We aim to show that $v[X]\subseteq D$ but $v[\alpha]\notin D$. This will imply that $\mat{M}$ is not an $\mathcal{S}$-model and hence does not belong to $\mathcal{M}$. The latter in turn will implies that $\mathcal{M}$ is not closed under ultraproducts, which is contrary to the assumption.
	
	Indeed, it follows from \eqref{E:ultrafilter-3} and \eqref{E:ultrafilter-2} that for any $\beta\in X$,
	\[
	v[\beta]=\overline{\beta}\slash\nabla\in D.
	\]
	On the other hand, 
	\[
	\set{Y\in I}{\alpha\in\ConS{Y}}=\emptyset.
	\]
	Therefore, in view of~\eqref{E:ultrafilter-3}, $v[\alpha]=\overline{\alpha}\slash\nabla\notin D$. 
\end{proof}

\paragraph{Exercises~\ref{section:finitary-matrix-consequence}}
\begin{enumerate}
	\item \label{EX:finitariness-M-consequence} Prove Corollary~\ref{C:finitariness-M-consequence}.
\end{enumerate}

\section{The conception of separating means}\label{section:separating-means}
Considering an abstract logic $\mathcal{S}$ as a consequence relation $\vdash_{\mathcal{S}}$, there are two main tasks: to confirm $X\vdash_{\mathcal{S}}\alpha$ or to refute it. Given a finitary abstract logic $\mathcal{S}$, these two tasks can be regarded as one twofold problem (of extreme importance): Is  there an algorithm that for any set $X\cup\lbrace\alpha\rbrace\Subset\Forms_{\mathcal{L}}$, decides whether $X\vdash_{\mathcal{S}}\alpha$ is true or false.

Trying to confirm $X\vdash_{\mathcal{S}}\alpha$, one can use rules of inference (sound w.r.t. $\mathcal{S}$) or a suitable adequate matrix or a set of matrices which determine $\mathcal{S}$. On the other hand, to refute $X\vdash_{\mathcal{S}}\alpha$, we need ``means'' that would separate $X$ from $\alpha$, for any $X\cup\lbrace\alpha\rbrace\subseteq\Forms_{\mathcal{L}}$ or at least for any $X\cup\lbrace\alpha\rbrace\Subset\Forms_{\mathcal{L}}$. According to Corollary~\ref{C:Lindenbaum-completeness}, the Lindenbaum atlas for $\mathcal{S}$ is a universal tool, for it works, at least in principal, for both tasks. However, it is too complicated and can be used mainly to prove very general theorems. For the refutation task, the Lindenbaum atlas is very inefficient. Nevertheless, in many cases, as we will see, it can be simplified to an extent to be useful.

A general view on the conception of \textit{separating means} was formulated by A. Kuznetsov~\cite{kuznetsov79}. According to him, if all formulas of a set $X$ stand in a relation $R$ to an object $\mathfrak{A}$, but a formula $\alpha$ does not stand in $R$ to $\mathfrak{A}$, we say that $\mathfrak{A}$ \textit{separates} $X$ from $\alpha$ w.r.t. $R$.
In this case, the character of the relation $R$ must be such that from the separation of $X$ from $\alpha$ it follows that $X\vdash_{\mathcal{S}}\alpha$ does not hold.\footnote{The conception of separating means, proposed by Kuznetsov, was intended to be applied not only to the issue of derivability, but also to that of expressibility understood within the framework of a given calculus.} He also emphasized that the object $\mathfrak{A}$ which was employed as a separating means in the above sense has often been of algebraic nature.

While the Lindenbaum atlas relative to an abstract logic $\mathcal{S}$ can play a role of universal separating means, but for a particular pair $X$ and $\alpha$, in case 
$X\vdash_{\mathcal{S}}\alpha$ does not hold, potentially any $\mathcal{S}$-matrix can be a separating means. Therefore, it would be interesting to find an algebraic connection between the Lindenbaum atlas or Lindenbaum matrix relative to $\mathcal{S}$ and any $\mathcal{S}$-matrix. This will be the topic of the next chapter.

\chapter[Unital Logics]{Unital Abstract Logics}
\label{chapter:lindenbaum-tarski}

\section{Unital algebraic expansions}	\label{section:chapter-unital-preliminaries}
From the discussion in Section~\ref{section:separating-means}, the following questions arise. 
\begin{description}
	\item[(1)]~\,Is there a regular method of simplifying the Lindenbaum atlas (Lindenbaum matrix) relative to an abstract logic $\mathcal{S}$ to obtain an adequate atlas (or matrix), or at least weakly adequate matrix that would be a more convenient separating means than $\Lin[\Sigma_{\mathcal{S}}]$ or $\LinS$?
	\item[(2)]~\,Is there a regular method of simplifying the Lindenbaum atlas (Lindenbaum matrix) relative to an abstract logic $\mathcal{S}$ to have a connection between these simplified forms of $\Lin[\Sigma_{\mathcal{S}}]$ or $\LinS$ and other $\mathcal{S}$-atlases and $\mathcal{S}$-matrices, respectively?
\end{description}

As to the question the question $\bm{(1)}$, even any finite matrix generates abstract logic, the Lindenbaum matrix of which is always infinite, for the cardinality of any Lindenbaum matrix equals the cardinality of the set of formulas. Take, for instance, the matrix $\booleTwo$ of Section~\ref{S:two-valued}. We showed (Section~\ref{section:inference-rules}) that $\booleTwo$ is adequate for the consequence relation $\vdash_{2}$. We denote the corresponding abstract logic by $\Cl$. Even a rough comparison between $\Lin_{\textsf{Cl}}$ and $\booleTwo$ demonstrates the difference between their complexities.

One thing is remarkable about the logical matrices discussed in Sections~\ref{S:two-valued}--\ref{section:dummett}. The logical filters of all of them consist of a single element. A logical matrix $\mat{M}=\langle\alg{A},D\rangle$ is called \textit{\textbf{unital}} if $D$ is a one-element set. We denote the designated element of a unital matrix by $\one$ (perhaps with subscript). Thus, if $\mat{M}=\langle\alg{A},\lbrace\one\rbrace\rangle$, we can consider  
the algebra $\langle\alg{A},\one\rangle$ as a semantics
for $\Lan$-formulas.  We call $\langle\alg{A},\one\rangle$ a \textit{\textbf{unital expansion}} of $\alg{A}$. Very often, such an expansion is not needed. Namely, when $D=\lbrace c\rbrace$, where $c$ is a constant from the signature of {\alg{A}}, we count that $\one=c$.   

Thus, in case \mat{M} is unital, instead of saying that an $\Lan$-formula $\alpha$ is valid in $\mat{M}$, we say that $\alpha$ is \textit{\textbf{valid}} in $\langle\alg{A},\one\rangle$, meaning that for any valuation $v$ in $\alg{A}$, $v[\alpha]=\one$. Further, given a unital expansion $\langle\alg{A},\one\rangle$, for any set $X\cup\lbrace\alpha\rbrace\subseteq\Forms_{\mathcal{L}}$, we define: 
\[
X\models_{\langle\textbf{A},\one\rangle}\alpha\stackrel{\text{df}}{\Longleftrightarrow}v[X]\subseteq\lbrace\one\rbrace\Rightarrow v[\alpha]=\one,~\text{for any valuation $v$ in $\alg{A}$}.
\]

The last definition allows us to adopt the definition of $\mathcal{S}$-model to algebras
$\langle\alg{A},\one\rangle$. Namely such an algebra is an (\textit{\textbf{algebraic}}) $\mathcal{S}$-\textit{\textbf{model}} if for any set $X\cup\lbrace\alpha\rbrace\subseteq\Forms_{\mathcal{L}}$,
\[
X\vdash_{\mathcal{S}}\alpha~\Longrightarrow~X\models_{\langle\textbf{A},\one\rangle}\alpha.
\]

Further, we say that an algebra $\langle\alg{A},\one\rangle$ is \textit{\textbf{adequate }}for an abstract logic $\mathcal{S}$ if for any set $X\cup\lbrace\alpha\rbrace\subseteq\Forms_{\mathcal{L}}$,
\[
X\vdash_{\mathcal{S}}\alpha~\Longleftrightarrow~X\models_{\langle\textbf{A},\one\rangle}\alpha;
\]
and that $\alg{A}$ is \textit{\textbf{weakly adequate}} for $\mathcal{S}$ if for any $\alpha\in\Forms_{\mathcal{L}}$,
\[
\alpha\in\ThmS~\Longleftrightarrow\text{for any valuation $v$ in $\alg{A}$},
v[\alpha]=\one.
\]

Finally, the definition of $\mathcal{M}$-consequence (Section~\ref{section:con-via-matrices}) can be adopted for the case when $\mathcal{M}=\lbrace\langle\alg{A}_i,\one_i\rangle\rbrace_{i\in I}$. Namely for such a class $\mathcal{M}$, we define:
\[
X\models_{\mathcal{M}}\alpha~\stackrel{\text{df}}{\Longleftrightarrow}~
X\models_{\langle\textbf{A}_i,\one_i\rangle}\alpha,~\text{for each $\langle\alg{A}_i,\one_i\rangle$}.
\]
Abstract logics which can be defined as the $\mathcal{M}$-consequence of this type are called \textit{assertional}; see~\cite{font2016}, definition 3.5.

It will be convenient to treat the validity of an $\Lan$-formula $\alpha$ in 
$\langle\alg{A},\one\rangle$ as the validity of a first order formula in a first order model.

Namely, given a sentential language $\Lan$, we define a first order language $\FOstar$ whose individual variables are the sentential variables of $\Lan$ and the connectives of $\Func_{\mathcal{L}}$ and the constants of $\Cons_{\mathcal{S}}$ are the functional symbols. (The constant symbols are treated as 0-ary functional symbols). In addition,  $\FOstar$ has a 0-ary functional symbol $\one$ which is interpreted in the $\FOstar$-models by a constant $\one$. Also, $\FOstar$ has symbol `$\approx$' which is interpreted in the $\FOstar$-models as equality, the first order logical connectives
$\&$ (conjunction), $\Rightarrow$ (implication) and universal quantifier $\forall$.
The parentheses `(' and `)' are used in the usual way.

Thus each $\Lan$-formula becomes in $\FOstar$ a term. In addition, $\FOstar$ has terms like
$\alpha(p,\ldots,\one,\ldots, q)$ obtained from $\alpha(p,\ldots,r,\ldots, q)$ by replacement of a variable $r$ with the constant $\one$.

The atomic formulas of $\FOstar$, called also \textit{\textbf{equalities}}, are the expressions of the form:
\[
\mathfrak{r}\approx\mathfrak{t},
\]
where $\mathfrak{r}$ and $\mathfrak{t}$ are $\FOstar$-terms.

Thus $\FOstar$ as a logical system is a first order logic with equality.

Given an $\FOstar$-formula $\phi$, its universal closure is denoted by
\[
\forall\ldots\forall\phi.
\]

The universal closure of an equality is often referred to as an \textit{identity}.
Given equalities $\phi,\phi_1,\ldots,\phi_n$, the formula
\[
\forall\ldots\forall((\phi_1\&\ldots\&\phi_n)\Rightarrow\phi)
\]
is called a \textit{quasiidentity}.

With each expansion $\langle\alg{A},\one\rangle$ we associate an $\FOstar$-model $\fA=\langle\alg{A},\one,\approx\rangle$ so that,
talking about validity of identities and quasiidentities in $\fA$, we, abusing notation, will be writing
\[
\fA\models\alpha(p,\one,\ldots)\approx\beta(q,\one,\ldots)~~\text{and}~~
\fA\models(\phi_1\&\ldots\&\phi_n)\Rightarrow\phi
\]
instead of, respectively, $\fA\models\forall\ldots\forall\alpha(p,\one,\ldots)\approx\beta(q,\one,\ldots)$
and\\ $\fA\models\forall\ldots\forall((\phi_1\&\ldots\&\phi_n)\Rightarrow\phi)$.
Thus, it must be clear that for any $\Lan$-formula $\alpha$ and any $\FOstar$-model
$\fA=\langle\alg{A},\one,\approx\rangle$,
\begin{equation}\label{E:model=algebra-validity-1}
\models_{\langle\textbf{A},\one\rangle}\alpha~\Longleftrightarrow~\fA\models\alpha\approx\one.
\end{equation}

Let $X:=\lbrace\alpha_1,\ldots,\alpha_n\rbrace$. We denote:
\[
X\approx\one\stackrel{\text{df}}{\Longleftrightarrow}\alpha_1\approx\one\&\ldots\&\alpha_n\approx\one.
\]
Thus, in this notation, for any set $X\cup\lbrace\alpha\rbrace\Subset\Forms_{\mathcal{L}}$, we can observe that
\begin{equation}\label{E:model=algebra-validity-2}
X\models_{\langle\textbf{A},\one\rangle}\alpha~\Longleftrightarrow~\fA\models (X\approx\one\Rightarrow\alpha\approx\one).
\end{equation}
(If $X=\emptyset$, we understand the last equivalence as~\eqref{E:model=algebra-validity-1}.)

\begin{prop}\label{P:S-algebraic-models=q-variety}
Let $\mathcal{S}$ be a finitary abstract logic. Then all algebraic $\mathcal{S}$-models form a quasivariety.
\end{prop}
\begin{proof}
Indeed, let $\left\langle \alg{A},\one\right\rangle $ be a unital algebraic expansion. Then, since $\mathcal{S}$ is finitary, $\left\langle \alg{A},\one\right\rangle$ is an $\mathcal{S}$-model if, and only if,
$\left\langle \alg{A},\one,\approx\right\rangle\models X\approx\one\Rightarrow \alpha\approx\one$, for any $X\cup\lbrace\alpha\rbrace\Subset\Forms_{\mathcal{L}}$ such that $X\vdash_{\mathcal{S}}\alpha$.	
\end{proof}

\noindent\textit{Remark}. We note that there are finitary abstract logics $\mathcal{S}$ which are not determined by their algebraic $\mathcal{S}$-models. For instance, let $\mathcal{S}^{\to}$ be the abstract logic in a language containing only one connective $\to$ and no other constants, defined by premiseless rule
$\bm{\alpha}\to\bm{\alpha}$ and modus ponens. Then $\mathcal{S}^{\to}$ has no adequate finite matrix and is not determined by any nonempty set of unital matrices; see~\cite{mckinsey-tarski1948}, p. 6, and~\cite{mendelson2015}, exercise 1.52, and~\cite{font2016}, example 6.78.
 
\section{Unital abstract logics}\label{section:unital-logics}
Given a set $X\subseteq\Forms_{\mathcal{L}}$, we denote by $\theta(X)$ the congruence on $\FormAl$, generated by the set $X\times X$. We will be interested in congruences $\theta(D)$, where $D$ is a theory of an abstract logic  $\mathcal{S}$, in particular in congruences $\theta(\ThmS)$.  

The following observation will be useful in the sequel.
\begin{prop}\label{P:motivating}
	A structural abstract logic $\mathcal{S}$ is nontrivial if, and only if, it satisfies the following property: For any $p,q\in\Var_{\mathcal{L}}$,
	\begin{equation}\label{E:motivating}
	p\neq q~\Longrightarrow~p\slash\theta(\ThmS)\neq q\slash\theta(\ThmS).
	\end{equation}
\end{prop}
\begin{proof}
	To prove the `only-if' part, we assume that $\mathcal{S}$ is nontrivial. Then, since $\mathcal{S}$ is structural, $\ThmS$ does not contain sentential variables. Let $M$ be the congruence class w.r.t. $\theta(\ThmS)$, which contains the variable $p$. Now, we define: $M_1:=\lbrace p\rbrace$ and $M_2:=M\setminus\lbrace p\rbrace$. The new partition of $\Forms_{\mathcal{L}}$ induces the equivalence $\theta$.
	It is obvious that $\theta\subset\theta(\ThmS)$. We will show that $\theta$ is a congruence on $\FormAl$.
	
	Let $F$ be an arbitrary $n$-ary connective with $n\ge 1$. Assume that for formulas $\alpha_1,\ldots,\alpha_n,\beta_1,\ldots,\beta_n$, each pair $(\alpha_i,\beta_i)\in\theta$.
	Then each pair $(\alpha_i,\beta_i)\in\theta(\ThmS)$ and hence 
	$(F\alpha_1\ldots\alpha_n,F\beta_1\ldots\beta_n)\in\theta(\ThmS)$. This means that either both $F\alpha_1\ldots\alpha_n$ and $F\beta_1\ldots\beta_n$ belong to $M_2$ or they belong to a class of $\theta(\ThmS)$ that is different from $M$ and hence is a class of $\theta$.
	This implies that $\lbrace p\rbrace$ and $\lbrace q\rbrace$ are two distinct congruence classes w.r.t. $\theta(\ThmS)$.
	
	Conversely, suppose the implication \eqref{E:motivating} holds. Then, since one of the congruence classes w.r.t. $\theta(\ThmS)$ contains $\ThmS$ and there are more than two classes,
	$\ThmS\neq\Forms_{\mathcal{L}}$.
\end{proof}

\subsection{Definition and some properties}

\begin{defn}[unital abstract logic]\label{D:unital-logic}
A structural abstract logic $\mathcal{S}$ is called \textbf{unital} if any $\mathcal{S}$-theory $D$ is a congruence class w.r.t. the congruence $\theta(D)$.
\end{defn}

Since any congruence class cannot be empty, from Definition~\ref{D:unital-logic} and Proposition~\ref{P:motivating}, we immediately obtain the following.
\begin{prop}\label{P:lindenbaum-connection}
For any unital abstract logic $\mathcal{S}$, $\ThmS\neq\emptyset$; in addition, $\mathcal{S}$ is nontrivial if, and only,  \eqref{E:motivating} holds.
\end{prop}

As we will see below (Corollary~\ref{C:compatible-congruence-2}), the unitality of an abstract logic is connected to the following concept.

\begin{defn}\label{D:compatible-congruence}
Let $X\subseteq\Forms_{\mathcal{L}}$ and $X\neq\emptyset$. A congruence $\theta$ on $\FormAl$ is said to be \textbf{compatible} with $X$
if for any $(\alpha,\beta)\in\theta$,
\[  
\alpha\in X~\Longleftrightarrow~\beta\in X;
\]
in other words, $X$ is the union of congruence classes w.r.t. $\theta$.
\end{defn}

It is obvious that 
\[
\theta_0:=\set{(\alpha,\alpha)}{\alpha\in\FormAl}
\]
is a congruence of $\FormAl$, which is compatible with any nonempty set $X\subseteq\Forms$.

\begin{prop}\label{P:compatible-congruence}
Let $X\subseteq\Forms_{\mathcal{L}}$ and $X\neq\emptyset$.
For any congruence $\theta$ on $\FormAl$, $X$ is a congruence class w.r.t. $\theta$ if, and only if, the following conditions are satisfied$\,:$
{\em\[
\begin{array}{cl}
(\text{a}) &\theta(X)\subseteq\theta;\\
(\text{b}) &\theta~\textit{is compatible with}~ X.
\end{array}
\]}
\end{prop}
\begin{proof}
Suppose $X$ is a congruence class w.r.t. $\theta$. Then (a) and (b) are obviously true.

Conversely, suppose both (a) and (b) are fulfilled. Then, according to (a), all formulas in $X$ are $\theta$-congruent to one another. In addition, in virtue of (b), each formula in $X$ cannot be $\theta$-congruent to a formula beyond $X$.
\end{proof}

The next two corollaries follow immediately from Proposition~\ref{P:compatible-congruence}.
\begin{cor}\label{C:compatible-congruence-1}
Let $\mathcal{S}$ be an abstract logic with $\ThmS\neq\emptyset$ and $D$ be an $\mathcal{S}$-theory. Then $D$ is a congruence class w.r.t. $\theta(D)$ if, and only if, 
$\theta(D)$ is compatible with $D$.
\end{cor}
\begin{cor}\label{C:compatible-congruence-2}
Let $\mathcal{S}$ be a structural abstract logic with $\ThmS\neq\emptyset$. Then the following conditions are equivalent$\,:$
{\em\[
\begin{array}{cl}
(\text{a}) & \textit{$\mathcal{S}$ is unital};\\
(\text{b}) & \textit{$\theta(D)$ is compatible with $D$, for each $\mathcal{S}$-theory $D$};\\
(\text{c}) & \textit{for any $\mathcal{S}$-theory $D$, there is a congruence $\theta$ on $\FormAl$}\\
&\textit{such that $\theta(D)\subseteq\theta$ and $\theta$ is compatible with $D$}.
\end{array}
\]}
\end{cor}

\subsection{Some subclasses of unital logics}\label{section:subclasses-unital}
In this subsection, we introduce several classes of abstract logics, which were considered in the literature and which turned out to be included in the class of unital logics.

\begin{defn}[implicative logic]\label{D:weakly-implicative}
A structural logic $\mathcal{S}$ in a language $\Lan$ is called \textbf{implicative} if $\Lan$ contains a binary connective $ \rightarrow $ such that for any $\mathcal{S}$-theory $D$, the following conditions are satisfied$\,:$
{\em\[
\begin{array}{cl}
(\text{a}) &\text{$\alpha\rightarrow\alpha\in D$, for any formula $\alpha$};\\
 (\text{b}) &\text{$\alpha\rightarrow\beta\in D$ and $\beta\rightarrow\gamma\in D$ imply $\alpha\rightarrow\gamma\in D$, for any $\alpha,\beta, \gamma\in\Forms_{\mathcal{L}}$};\\
(\text{c}) &\text{$\alpha\in D$ and $\alpha\rightarrow\beta\in D$ imply $\beta\in D$, for any $\alpha,\beta\in\Forms_{\mathcal{L}}$};\\
(\text{d}) &\text{$\beta\in D$ implies $\alpha\rightarrow\beta\in D$, for any $\alpha,\beta\in\Forms_{\mathcal{L}}$};\\
(\text{e}) &\text{if $\alpha_i\rightarrow\beta_i\in D$, 
	 $\beta_i\rightarrow\alpha_i\in D$, $1\le i\le n$, and $F$ is any $n$-ary connective},\\
 &\text{then $F\alpha_1\ldots\alpha_n\rightarrow F\beta_1\ldots\beta_n\in D$, for any formulas $\alpha_1,\ldots,\alpha_n,\beta_1,\ldots,\beta_n$}.\\
\end{array}
\]}
\end{defn}

We note that if a logic $\mathcal{S}$ is implicative, then $\bm{T}_{\mathcal{S}}\neq\emptyset$.

Let $\mathcal{S}$ be (at least) an implicative logic and $D$ be an $\mathcal{S}$-theory. We define:
\[
(\alpha,\beta)\in\theta_{D}~\stackrel{\text{df}}{\Longleftrightarrow}~
\text{$\alpha\rightarrow\beta\in D$ and $\beta\rightarrow\alpha\in D$}.
\]
\begin{prop}
Let $D$ be a theory of an implicative logic $\mathcal{S}$. Then $\theta_{D}$ is a congruence on $\FormAl$
and $D$ is a congruence class w.r.t. $\theta_{D}$.
\end{prop}
\begin{proof}
That $\theta_{D}$ is a congruence follows straightforwardly from (a), (b) and (e) of Definition~\ref{D:weakly-implicative}.

Next, we first note that, by premise, $D\neq\emptyset$. Secondly, if both
$\alpha$ and $\beta$ belong to $D$, then, according to (d) of Definition~\ref{D:weakly-implicative}, both $\alpha\rightarrow\beta$ and $\beta\rightarrow\alpha $ belong to $D$, that is $(\alpha,\beta)\in\theta_{D}$. And, thirdly, let $(\alpha,\beta)\in\theta_{D}$ and $\alpha\in D$. Then $\beta\in D$. In virtue of Proposition~\ref{P:compatible-congruence}, this implies that $D$ is a congruence class w.r.t. $\theta_{D}$. 
\end{proof}
\begin{cor}
	If an abstract logic $\mathcal{S}$ is implicative, then $\mathcal{S}$ is unital.
\end{cor}

The calculus in the language $\Lan_A$, which is defined by the inference rules ax1--ax10 and (b-$iii$) (modus ponens) of Section~\ref{section:inference-rules}, is called the \textit{\textbf{classical propositional logic}} and denoted by $\Cl$.
The calculus in $\Lan_A$, which is  defined by the rules ax1--ax9, ax11 (Section~\ref{section:modus-rules-vs-rules}) and modus ponens, is called the \textit{\textbf{intuitionistic propositional logic}} and denoted by $\Int$.
Also, it will be convenient to consider another calculus in the language $\Lan_A$, \textsf{P}, which is defined by the inference rules ax1--ax8 and (b-$iii$) (modus ponens) of Section~\ref{section:inference-rules}. We call $\Pos$ \textit{\textbf{positive logic}}.

It must be clear that for any set $X\cup\lbrace\alpha\rbrace\subseteq\Forms_{\mathcal{L}_A}$
\begin{equation}\label{E:two-implications}
	X\vdash_{\textsf{P}}\alpha~\textit{implies both}~X\vdash_{\textsf{Int}}\alpha
	~\textit{and}~X\vdash_{\textsf{Cl}}\alpha.
\end{equation}

Further, it can be easily shown that {\Pos}, $\Cl$ and $\Int$ are implicative. (See Exercise~\ref{section:unital-logics}.\ref{EX:Cl-In-implicative}.)\\

Another class of abstract logics, the class of Fregean logics, stems from 
\textit{Frege's principle of compositionality} (FPC). 

An abstract logic $\mathcal{S}$ has the FPC if for for any set $X\cup\lbrace\gamma(p,\ldots),\alpha,\beta\rbrace\subseteq\Forms_{\mathcal{L}}$,
\[
\text{$X,\alpha\vdash_{\mathcal{S}}\beta$ and $X,\beta\vdash_{\mathcal{S}}\alpha$ imply $X,\gamma(\alpha,\ldots)\vdash_{\mathcal{S}}\gamma(\beta,\ldots)$}.
\]

This leads to the following (equivalent) definition.
\begin{defn}[Fregean relation, Fregean logic]
Given a structural logic $\mathcal{S}$ and an $\mathcal{S}$-theory $D$, the relation on $\FormAl$ defined by the equivalence
{\em\[
(\alpha,\beta)\in\Lambda D~\stackrel{\text{df}}{\Longleftrightarrow}~
\text{$D,\alpha\vdash_{\mathcal{S}}\beta$ and $D,\beta\vdash_{\mathcal{S}}\alpha$}
\]}
is called a \textbf{Fregean relation relative to} $D$. A structural logic
is \textbf{Fregean} if for every $\mathcal{S}$-theory $D$, $\Lambda D$ is a congruence on $\FormAl$.
\end{defn}

We leave the reader to prove that for any structural logic $\mathcal{S}$, the properties of satisfying the FPC and being Fregean are equivalent. (See Exercise~\ref{section:lindenbaum-algebra}.\ref{EX:Fregean-equivalence}.)
\begin{prop}
If $\mathcal{S}$ is a Fregean logic with $\ThmS\neq\emptyset$, then $\mathcal{S}$ is unital.
\end{prop}
\begin{proof}
Let $D$ be an arbitrary $\mathcal{S}$-theory. We observe that
$\theta(D)\subseteq\Lambda D$. Also, $\Lambda D$ is compatible with $D$.
Applying Corollary~\ref{C:compatible-congruence-2}, we obtain that $\mathcal{S}$ is unital.
\end{proof}

\paragraph{Exercises~\ref{section:unital-logics}}
	\begin{enumerate}
	\item \label{EX:Cl-In-implicative}Prove that the abstract logics $\Pos$, $\Cl$ and $\Int$ are implicative.
	\item Prove that the abstract logics $\Pos$, $\Cl$ and $\Int$ are Fregean.
	\item\label{EX:Fregean-equivalence} Prove that a structural logic $\mathcal{S}$ is Fregean if, and only if, it satisfies the FPC.
	\end{enumerate}

\section{Lindenbaum-Tarski algebra}\label{section:lindenbaum-algebra}

\subsection{Definition and properties}
The concept of unital abstract logic in turn induces the following important notion.
\begin{defn}[Lindenbaum-Tarski algebra]\label{D:LT-algebra}
	Let $\mathcal{S}$ be a unital logic and $D$ be an $\mathcal{S}$-theory. The algebra
	{\em$\LT[D]:=\langle\FormAl\slash\theta(D),\one_D\rangle$}, where $\one_{D}$ is $D$ understood as a congruence class w.r.t. $\theta(D)$, $($or very often only the first component $\FormAl\slash\theta(D)$$)$ is called the \textbf{Lindenbaum-Tarski algebra} $($of $\mathcal{S}$$)$ \textbf{relative to} $D$. The algebra {\em$\LT:=\langle\FormAl\slash\theta(\bm{T}_{\mathcal{S}}),\one_{\bm{T}_{\mathcal{S}}}\rangle$} $($or the first component $\FormAl\slash\theta(\bm{T}_{\mathcal{S}})$$)$ is called simply the \textbf{Lindenbaum-Tarski algebra} of $\mathcal{S}$.
\end{defn}

We observe that, given a unital logic $\mathcal{S}$ and any $\mathcal{S}$-theories $D$ and $D^{\prime}$, if $D^{\prime}\subseteq D$, then
$\theta(D)\subseteq\theta(D^{\prime})$ and hence the algebra $\LT[D]$ is a homomorphic image of $\LT[D^{\prime}]$. Thus, for any $\mathcal{S}$-theory $D$, the algebra $\LT[D]$ is a homomorphic image of $\LT$.

\begin{lem}\label{L:congruence-substitution}
	Let $\theta$ be a congruence on $\FormAl$. Then for any valuation $v$ in \mbox{$\FormAl\slash\theta$},  there is an $\Lan$-substitution $\sigma$ such that for any formula $\alpha$,  \mbox{$v[\alpha]=\sigma(\alpha)\slash\theta$}.
\end{lem}
\begin{proof}
	Let $v$ be any valuation in $\FormAl\slash\theta$; that is, $v$ is a homomorphism of
	$\FormAl$ to $\FormAl\slash\theta$. Suppose $v[p]=\beta_p\slash\theta$, where $p\in\Var_{\mathcal{L}}$. Then we define a substitution as follows:
	\[
	\sigma:p\mapsto\beta_p,
	\]
	for all $p\in\Var_{\mathcal{L}}$, with any selection of $\beta_p$ from the class $\beta_p\slash\theta$. Thus, we have: $v[p]=\sigma(p)\slash\theta$, for all $p\in\Var_{\mathcal{L}}$.
	
	Now let us take any formula $\alpha$ and assume that 
	$p_1,\ldots,p_n$ are all variables  that occur in $\alpha$. Then we have:
	\[
	v[\alpha]=\alpha(v[p_1],\ldots,v[p_n])=\alpha(\sigma(p_1)\slash\theta,\ldots,\sigma(p_n)\slash\theta)=\sigma(\alpha)\slash\theta.
	\]
\end{proof}

\begin{prop}\label{P:S-algebraic-models}
Let $\mathcal{S}$ be a unital logic. Each {\em$\LT[D]$} is an $\mathcal{S}$-model.
\end{prop}
\begin{proof}
Suppose that $X\vdash_{\mathcal{S}}\alpha$. Let us take any $\LT[D]$ and consider a valuation $v$ in it. In virtue of Lemma~\ref{L:congruence-substitution}, there is a substitution $\sigma$ such that for any formula $\beta$, $v[\beta]=\sigma(\beta)\slash\theta(D)$. Further, we note that, since $\mathcal{S}$ is structural, $\sigma(X)\vdash_{\mathcal{S}}\sigma(\alpha)$. 

Now, suppose that for any $\gamma\in X$, $v[\gamma]=\one_{D}$. Then, since $\mathcal{S}$ is unital, for any $\gamma\in X$, $\sigma(\gamma)\in D$. This implies that $\sigma(\alpha)\in D$, that is $v[\alpha]=\one_{D}$.
\end{proof}
\begin{cor}\label{C:quasiidentity-sound}
Let $\mathcal{S}$ be a unital logic. Then for any set \mbox{$X\cup\lbrace\alpha\rbrace\Subset\Forms_{\mathcal{L}}$},
and for any $\mathcal{S}$-theory $D$, 
{\em\[
X\vdash_{\mathcal{S}}\alpha~\Longrightarrow~\LT[D]\models X\approx\one_{D}\Rightarrow\alpha\approx\one_{D}.
\]}
\end{cor}
\begin{proof}
	We apply successively Proposition~\ref{P:S-algebraic-models} and~\eqref{E:model=algebra-validity-2}.
\end{proof}

\begin{prop}\label{P:preLindenbaum-algebra}
	Let $\mathcal{S}$ be a unital logic and let {\em$D\!:=\ConS{X}$}. Then for any formula $\alpha$,
	{\em\[
		\models_{\LT[D]}\alpha\Longleftrightarrow~\models_{\LinS[X]}\alpha.
		\]}
	Hence, {\em$\LT$} is weakly adequate for $\mathcal{S}$.
\end{prop}
\begin{proof}
	First, assume that $\models_{\LinS[X]}\alpha$; that is for any substitution $\sigma$, $\sigma(\alpha)\in D$.
	Let $v$ be a valuation in $\LT[D]$. By Lemma~\ref{L:congruence-substitution}, there is a substitution $\sigma_v$ such that $v[\alpha]=\sigma_v(\alpha)\slash\theta(D)$. By premise, this implies that $v[\alpha]=\one_{D}$.
	
	Conversely, suppose for some substitution $\sigma$, $\sigma(\alpha)\notin D$. We define a valuation in $\LT[D]$ as follows: $v[p]:=\sigma(p)\slash\theta(D)$, for any $p\in\Var_{\mathcal{L}}$. Then $v[\alpha]=\sigma(\alpha)\slash\theta(D)\neq\one_D$.
\end{proof}

\begin{cor}
	Let $D$ be an $\mathcal{S}$-theory of a unital logic $\mathcal{S}$. Then for any
	$\mathcal{S}$-thesis $\alpha$, the algebras {\em$\LT[D]$} validate $\alpha$.
\end{cor}
\begin{proof}
	Having noticed that $\LT[D]$ is a homomorphic image of $\LT$, we apply the last part of Proposition~\ref{P:preLindenbaum-algebra}.
\end{proof}

We recall that $\Sigma_{\mathcal{S}}$ denotes the set of all $\mathcal{S}$-theories.
\begin{cor}\label{C:LT-determination}
Any unital logic $\mathcal{S}$ is determined by the class \mbox{\em$\set{\LT[D]}{D\in\Sigma_{\mathcal{S}}}$}.
\end{cor}
\begin{proof}
Indeed, if $X\vdash_{\mathcal{S}}\alpha$, then, in virtue of Proposition~\ref{P:S-algebraic-models}, $X\models_{\LT[D]}\alpha$, for each $D\in\Sigma_{\mathcal{S}}$.

On the other hand, if $X\not\vdash_{\mathcal{S}}\alpha$, then $\not\models_{\LinS[X]}\alpha$ and hence, according to Proposition~\ref{P:preLindenbaum-algebra}, $\not\models_{\LT[D]}\alpha$, where $D=\ConS{X}$.
\end{proof}

\begin{cor}\label{C:preLindenbaum-algebra}
	Let $\mathcal{S}$ be a unital logic.  If a set $X\subseteq
	\Forms_{\mathcal{L}}$ is closed under arbitrary substitutions, then for any formula $\alpha$,
	{\em\[
		X\vdash_{\mathcal{S}}\alpha~\Longleftrightarrow~\models_{\LT[D]}
		\alpha,
		\]}
	where {\em$D=\ConS{X}$}.
\end{cor}
\begin{proof}
We have:
\[
\begin{array}{rl}
X\vdash_{\mathcal{S}}\alpha \Longleftrightarrow \!\!\!&\models_{\LinS[X]}\alpha~~[\text{Lindenbaum's theorem (Proposition~\ref{P:lindenbaum-theorem})}]\\
\Longleftrightarrow \!\!\!&\models_{\LT[D]}\alpha.~~[\text{Proposition~\ref{P:preLindenbaum-algebra}}]\\
\end{array}
\]
\end{proof}

Now we define a special valuation in $\LT[D]$:
\[
\bm{v}_{D}: p\mapsto p\slash\theta(D),\tag{\textit{Lindenbaum valuation}}
\]
for any $p\in\Var_{\mathcal{L}}$. 
\begin{prop}\label{P:valuations-in-LT}
	Let $\mathcal{S}$ be a unital logic and {\em$D:=\ConS{X}$}.
	Then for any formula $\alpha$, the following conditions are equivalent$\,:$
	{\em\[
		\begin{array}{cl}
		(\text{a}) & X\vdash_{\mathcal{S}}\alpha;\\
		(\text{b}) &X\models_{\LT[D]}\alpha;\\
		(\text{c}) &\bm{v}_D[\alpha]=\one_{D}.
		\end{array}
		\]}
	In particular, the following conditions are equivalent$\,:$
	{\em\[
		\begin{array}{cl}
		(\text{a}^{\ast}) & \alpha\in\ThmS;\\
		(\text{b}^{\ast}) &v[\alpha]=\one_{\bm{T}_{\mathcal{S}}},~\textit{for any valuation $v$ in}~ \LT;\\
		(\text{c}^{\ast}) &\bm{v}_{\bm{T}_{\mathcal{S}}}[\alpha]=\one_{\bm{T}_{\mathcal{S}}}.
		\end{array}
		\]}
\end{prop}
\begin{proof}
	The implication (\text{a})$\Rightarrow$(\text{b}) follows from Corollary~\ref{C:quasiidentity-sound}.
	
	Assume (b). Let us take any $\beta\in X$. We have: $\bm{v}_{D}[\beta]=\beta\slash\theta(D)=\one_{D}$. Hence (c). 
	
	Now assume (c), that is $\bm{v}_D[\alpha]=\one_{D}$. This means that $\alpha\slash\theta(D)=\one_{D}$, that is $\alpha\in D$. Hence (a).
\end{proof}

We aim to show that the Lindenbaum-Tarski algebras are good as separating means (Section~\ref{section:separating-means}); or, better to say, they are useful for creating separating means. With this in mind, given a unital logic $\mathcal{S}$ in a language $\Lan$, we define the following classes of unital expansions.
\begin{itemize}
	\item $\PS$ is the class of all algebraic $\mathcal{S}$-models; according to Proposition~\ref{P:S-algebraic-models=q-variety}, $\PS$ is a quasivariety.
	\item $\QS$ is the quasivariety generated by the set $\set{\LT[D]}{D\in\Sigma_{\mathcal{S}}}$; according to Proposition~\ref{P:S-algebraic-models}, $\QS\subseteq\PS$; also, if $\mathcal{S}$ is finitary, then, in virtue of Corollary~\ref{C:quasiidentity-sound}, $\QS$ is determined by the quasiidentities $X\approx\one\Rightarrow\alpha\approx\one$, for any $X\cup\lbrace\alpha\rbrace\Subset\FormAl$ such that $X\vdash_{\mathcal{S}}\alpha$.
	\item $\TS$ is the class of all unital expansions $\langle\alg{A},\one\rangle$ such that
	 $\langle\alg{A},\one\rangle\models\alpha\approx\one$, for any $\alpha\in\ThmS$; it is obvious that
	 $\set{\LT[D]}{D\in\Sigma_{\mathcal{S}}}\subseteq\TS$.
	
\end{itemize}

Summing up, we conclude this subsection as follows.
\begin{theo}
Let $\mathcal{S}$ be a unital logic. Then the following conditions are equivalent$\,:$
{\em\[
\begin{array}{cl}
(\text{a})	&X\vdash_{\mathcal{S}}\alpha;\\
(\text{b})	&\text{$X\models_{\langle\textbf{A},\one\rangle}\alpha$, for any 
$\langle\textbf{A},\one\rangle\in\PS$};\\
(\text{c})	&\text{$X\models_{\langle\textbf{A},\one\rangle}\alpha$, for any 
	$\langle\textbf{A},\one\rangle\in\QS$};\\
(\text{d})	&\text{$X\models_{\LT[D]}\alpha$, for any $D\in\Sigma_{\mathcal{S}}$}.
\end{array}
\]}
\end{theo}
\begin{proof}
The implication (a)$\Rightarrow$(b) is due to definition of $\PS$; (b)$\Rightarrow$(c) is due to inclusion $\QS\subseteq\PS$; (c)$\Rightarrow$(d) is due to definition of $\QS$; and
(d)$\Rightarrow$(a) is due to Corollary~\ref{C:LT-determination}.
\end{proof}

Thus, we see that if $X\not\vdash_{\mathcal{S}}\alpha$, an algebra $\langle\alg{A},\one\rangle$ separating $X$ from $\alpha$ 
 can be found among the algebras of the class 
$\QS=\Su\Pred\lbrace\LT[D]\rbrace_{D\in\Sigma_{\mathcal{S}}}$; cf.~\cite{malcev1973}, {\S} 11, theorem 4. As we will see in the next subsection, under a restriction on $\mathcal{S}$, separating algebras (in the above sense) can be reached as homomorphic images of $\LT$. In the proof of this below, the class $\TS$ plays a key role. The importance of the class $\TS$ also lies in the fact that if $X\not\vdash_{\mathcal{S}}\alpha$, a separating algebra
$\langle\alg{A},\one\rangle$ is likely to belong to $\QS\cap\TS$.

\subsection{Implicational unital logics}
In view of~\eqref{E:model=algebra-validity-1}, it is clear that the class $\TS$ is a variety. We aim to show that, under above restriction on $\mathcal{S}$, $\LT$ is a free algebra over $\TS$.

\begin{defn}[implicational abstract logic]
Let $\mathcal{S}$ be an abstract logic  in a language $\Lan$. Assume that 
there is an $\Lan$-formula of two sentential variables, $\im(p,q)$, such that for any $\Lan$-formulas $\alpha$ and $\beta$, any algebra {\em$\langle\alg{A},\one\rangle\in\TS$} and any valuation $v$ in $\alg{A}$,
\[
v(\alpha)=v(\beta)~\Longleftrightarrow~v(\im(\alpha,\beta))=\one~\text{and}~v(\im(\beta,\alpha))=\one.
\]
Then we call $\mathcal{S}$ \textbf{implicational} w.r.t. $\im(p,q)$.
\end{defn}

The abstract logics $\Pos$, $\Cl$ and $\Int$ are examples of implicational logics. (Exercise~\ref{section:lindenbaum-algebra}.\ref{EX:Cl-Int-implicational})

Below we will be considering implicational logics which are also unital. Since for such a logic $\mathcal{S}$, $\ThmS\neq\emptyset$ (Proposition~\ref{P:lindenbaum-connection}), we select an arbitrary $\Lan$-formula $\bigstar\in\ThmS$.

In virtue of Proposition~\ref{P:valuations-in-LT}, for any valuation $v$ in $\LT$,
\begin{equation}\label{E:T_S-star-thesis}
v(\bigstar)=\one_{\bm{T}_{\mathcal{S}}}.
\end{equation}
(Exercise~\ref{section:lindenbaum-algebra}.\ref{EX:T_S-star-thesis})

We recall the signature of any unital expansion $\langle\alg{A},\one\rangle$ is of type $\Lan$ expanded by an additional constant symbol $\one$ which is interpreted by the 0-ary operation $\one$ in $\langle\alg{A},\one\rangle$ (denoted is denoted by $\one_{\bm{T}_{\mathcal{S}}}$ in $\LT$). We denote this extended type by $\Lan^{\star}$.

Let $\mathbf{t}$ be a term of type $\Lan^{\star}$. The \textbf{\textit{conversion}}
of $\mathbf{t}$ is the $\Lan$-formula (or term of type $\Lan$), denoted by $\mathbf{t}^{\star}$, which is obtained from $\mathbf{t}$ by simultaneous replacement of all occurrences of $\one$, if any, in $\mathbf{t}$ by $\bigstar$. 

Consequently, any valuation $v$ in an algebra $\alg{A}$ of type $\Lan$ is extended to
the valuation $v^{\star}$ in $\langle\alg{A},\one\rangle$ such that $v^{\star}(p)=v(p)$, for any $p\in\Var_{\mathcal{L}}$, and $v^{\star}(\one)=\one$, where the first occurrence of `$\one$' is a 0-ary constant of $\Lan^{\star}$, and the second is the 0-ary operation $\one$ of $\langle\alg{A},\one\rangle$. We refer to $v^{\star}$ as a $\star$-\textit{\textbf{extension}} of $v$. We note that the valuation $v$ can be restored from $v^{\star}$ in the sense that if $w$ is any valuation in $\langle\alg{A},\one\rangle$, there is a unique valuation $v$ in $\alg{A}$ such that $v^{\star}=w$. (Exercise~\ref{section:lindenbaum-algebra}.\ref{EX:v-restored})

Using~\eqref{E:T_S-star-thesis}, we observe that for any $\Lan^{\star}$-terms $\mathbf{t}$ and $\mathbf{r}$ and any valuation $v^{\star}$ in $\LT$,
\begin{equation}\label{E:LT-term-valuation}
v^{\star}(\mathbf{t})=v^{\star}(\mathbf{r})~\Longleftrightarrow~v(\mathbf{t}^{\star})=v(\mathbf{r}^{\star}).
\end{equation} 
(Exercise~\ref{section:lindenbaum-algebra}.\ref{EX:LT-term-valuation})

For any terms $\mathbf{t}$ and $\mathbf{r}$ of type $\Lan^{\star}$, we denote:
\begin{center}
	$\TS\models\mathbf{t}\approx\mathbf{r}~\stackrel{\text{df}}{\Longleftrightarrow}~$
	for any $\left\langle \alg{A},\one\right\rangle \in\TS$, $\models_{\langle\textbf{A},\one\rangle}\mathbf{t}\approx\mathbf{r}$.
\end{center}

The equivalence~\eqref{E:LT-term-valuation} can be generalized as follows.
\begin{lem}\label{L:LT=free-algebra}
Let $\langle \alg{A},\one\rangle\in\TS$. Then for any term $\mathbf{t}$
of type $\Lan^{\star}$ and any valuation $v^{\star}$ in $\langle \alg{A},\one\rangle$,
$v^{\star}(\mathbf{t})=v(\mathbf{t}^{\star})$.
\end{lem}
\noindent\textit{Proof}~can be carried out by induction on the number of functional symbols of the language $\Lan^{\star}$. We leave it to the reader. (Exercise~\ref{section:lindenbaum-algebra}.\ref{EX:LT=free-algebra})\\

\begin{prop}\label{P:LT=free-algebra}
Let $\mathcal{S}$ be nontrivial unital logic which is implicational w.r.t.
$\im(p,q)$. Then for terms $\mathbf{t}$ and $\mathbf{r}$ of type $\Lan^{\star}$,
\[
\TS\models\mathbf{t}\approx\mathbf{r}~\Longleftrightarrow~\bm{v}_{0}^{\star}(\mathbf{t})=\bm{v}_{0}^{\star}(\mathbf{r}),
\]
where $\bm{v}_{0}^{\star}$ is the $\star$-extension of $\bm{v}_{0}$. Hence {\em$\LT$} is a free algebra of rank $\card{\Lan}$ over the variety $\TS$ and also over the quasivariety $\QS\cap\TS$.
\end{prop}
\begin{proof}
We successively obtain:
\[
\begin{array}{rcl}
\TS\models\mathbf{t}\approx\mathbf{r}
&\Longleftrightarrow &\TS\models\mathbf{t}^{\star}\approx\mathbf{r}^{\star}\quad[\text{Lemma~\ref{L:LT=free-algebra}}]\\
&\Longleftrightarrow &\TS\models\im(\mathbf{t}^{\star},\mathbf{r}^{\star})\approx\one_{\bm{T}_{\mathcal{S}}}~\text{and}~\TS\models\im(\mathbf{r}^{\star},\mathbf{t}^{\star})\approx\one_{\bm{T}_{\mathcal{S}}}\\
&&[\text{for $\mathcal{S}$ is implicational w.r.t. $\im(p,q)$}]\\
&\Longleftrightarrow &\im(\mathbf{t}^{\star},\mathbf{r}^{\star})\in\ThmS~
\text{and}~\im(\mathbf{r}^{\star},\mathbf{t}^{\star})\in\ThmS\\
&&[\text{since $\LT\in\TS$ and Proposition~\ref{P:preLindenbaum-algebra}, part 2}]\\
&\Longleftrightarrow &\bm{v}_{0}(\im(\mathbf{t}^{\star},\mathbf{r}^{\star}))=\one_{\bm{T}_{\mathcal{S}}}~
\text{and}~\bm{v}_{0}(im(\mathbf{r}^{\star},\mathbf{t}^{\star}))=\one_{\bm{T}_{\mathcal{S}}}\\
&&[\text{Proposition~\ref{P:preLindenbaum-algebra}, part 2}]\\
&\Longleftrightarrow &\bm{v}_{0}(\mathbf{t}^{\star})=\bm{v}_{0}(\mathbf{r}^{\star})
\quad[\text{for $\mathcal{S}$ is implicational w.r.t. $\im(p,q)$}]\\
&\Longleftrightarrow &\bm{v}_{0}^{\star}(\mathbf{t}^{\star})=\bm{v}_{0}^{\star}(\mathbf{r}^{\star})
\quad[\text{in virtue of~\eqref{E:LT-term-valuation}}].
\end{array}
\]

According to~\cite{malcev1973}, {\S} 12.2, theorem 1, this implies that $\LT$ is a free algebra over the class $\TS$. In virtue of Proposition~\ref{P:motivating}, the rank of $\LT$ equals $\card{\Lan}$.

Since the quasivariety $\QS\cap\TS$ contains all algebras $\LT[D]$ and because of Corollary~\ref{C:LT-determination}, $\mathcal{S}$ is determined by
$\QS\cap\TS$. It is obvious that $\LT$ is a free algebra over $\QS\cap\TS$.
\end{proof}

\subsection{Examples of Lindenbaum-Tarski algebras}\label{section:some-lindenbaum-algebras}
Let $\mathcal{S}$ be a nontrivial unital logic in a language $\Lan$ with $\card{\Var_{\mathcal{L}}}=\kappa$. Then, according to~\eqref{E:motivating},
$\card{|\LT|}\ge\kappa$. In particular, $\card{|\LTCl|}=\card{|\LTInt|}=\card{\Var_{\mathcal{L}_A}}$.
It must be clear that $\set{p\slash\theta(\ThmS)}{p\in\Var_{\mathcal{L}}}$ is the set of free generators of $\LT$, if the latter can be considered as a free algebra of $\TS$.
For some questions like decision problem for the consequence relation or theoremhood of $\mathcal{S}$, it suffices to consider the formulas built of a limited set of variables. Thus, if $\Var\subseteq\Var_{\mathcal{L}}$ with $\card{\Var}=\kappa$ and our consideration of $\mathcal{S}$ is limited to the formulas built from $\Var$, we denote by $\LT(\kappa)$ a subalgebra of $\LT$, generated by the set $\set{p\slash\theta(\ThmS)}{p\in\Var}$. The notation is justified by the fact if $\Var^{\prime}\subseteq\Var_{\mathcal{L}}$ and
$\card{\Var^{\prime}}=\card{\Var}$, then the subalgebras of $\LT$, generated by
$\set{p\slash\theta(\ThmS)}{p\in\Var}$ and, respectively, by 
$\set{p\slash\theta(\ThmS)}{p\in\Var^{\prime}}$ are isomorphic; cf.~\cite{gratzer2008}, {\S}24, corollary 2. Since each such $\LT(\kappa)$ is a free algebra in $\QS$ (or even in $\TS$), it is called the free algebra of rank $\kappa$. We note that, since $\LT$ is always has at least one 0-ary operation, $\one$, even it makes sense to talk about a free algebra of rank 0.

We consider the free algebras $\LTCl(0)$, $\LTInt(0)$, $\LTCl(1)$
$\LTCl(2)$ and $\LTInt(1)$, in order. The first two algebras should be understood as subalgebras of $\LTCl$ and $\LTInt$, generated by their unit, $\one$. The resulting algebras are isomorphic (see Section~\ref{section:boolean-algebra} and Section~\ref{section:heyting-algebra}) and can be depicted by the diagram:

\begin{figure}[!ht]	
	\[
	\ctdiagram{
		\ctnohead
		\ctinnermid
		\ctel 0,0,0,20:{}
		\ctv 0,0:{\bullet}
		\ctv 0,20:{\color{red}\bullet}
		\ctv 0,27:{\mathbf{1}}
	}
	\]
	\caption{Lindenbaum-Tarski algebra $\LTCl(0)$, or equivalently $\LTInt(0)$}
\end{figure}

Before discussing the structure of the other just mentioned Lindenbaum-Tarski algebras, we first consider the varieties
$\mathbb{T}_{\textsf{Cl}}$ and $\mathbb{T}_{\textsf{Int}}$. Namely, we aim to show that these varieties coincide with the varieties of Boolean and Heyting algebras, respectively.

Since the classes $\mathbb{T}_{\textsf{Cl}}$ and $\mathbb{T}_{\textsf{Int}}$ consist of the algebras validating the theses of $\Cl$ and $\Int$, respectively, (defined in Section~\ref{section:subclasses-unital}) we take a look at these calculi more closely. 
\begin{lem}\label{L:rules-sound-wrt-P-Int-Cl}
In addition to modus ponens, the inference rules {\em(\text{a}--$i$)--(\text{a}--$iii$)} and 
{\em(\text{b}--$i$)--(\text{b}--$ii$)} of Section~\ref{section:inference-rules} are sound w.r.t. {\em\Pos} and hence w.r.t. {\em\Int} and {\em\Cl}.
\end{lem}
\begin{proof}
To prove the main assertion, we have to show that for any $\Lan_A$-formulas $\alpha$ and $\beta$, the following hold:
\[
\alpha,\beta\vdash_{\textsf{P}}\alpha\wedge\beta,\quad\alpha\vdash_{\textsf{P}}
\alpha\vee\beta,\quad\beta\vdash_{\textsf{P}}\alpha\vee\beta,\quad\alpha\wedge\beta
\vdash_{\textsf{P}}\alpha~\text{and}~\alpha\wedge\beta
\vdash_{\textsf{P}}\beta.
\]
When it's done, we apply~\eqref{E:two-implications}.
We leave this task to the reader. (Exercise~\ref{section:lindenbaum-algebra}.\ref{EX:lemma-rules-sound-wrt-P-Int-Cl})
\end{proof}

The next lemma is known as the deduction theorem.
\begin{lem}\label{L:deduction-theorem}
The hyperule {\em(\text{c}--$i$)} of Section~\ref{section:inference-rules} is sound w.r.t. {\em\Pos} and hence w.r.t. {\em\Int} and {\em\Cl}.
\end{lem}
\begin{proof}
First of all, we notice that for any $\Lan_A$-formulas $\alpha$ and $\beta$,
\[
\begin{array}{cll}
1. &\vdash_{\textsf{P}}\alpha\rightarrow(\beta\rightarrow\alpha) &[\text{an instance of ax1}]\\
2. &\vdash_{\textsf{P}}\alpha\rightarrow((\beta\rightarrow\alpha)\rightarrow\alpha) &[\text{an instance of ax1}]\\
3. &\vdash_{\textsf{P}}(\alpha\rightarrow(\beta\rightarrow\alpha))\rightarrow
((\alpha\rightarrow((\beta\rightarrow\alpha)\rightarrow\alpha))
\rightarrow(\alpha\rightarrow\alpha)) &[\text{an instance of ax2}]\\
4. &\vdash_{\textsf{P}}(\alpha\rightarrow((\beta\rightarrow\alpha)\rightarrow\alpha))
\rightarrow(\alpha\rightarrow\alpha) &[\text{from 1 and 3 by (b--$iii$)}]\\
5. &\vdash_{\textsf{P}}\alpha\rightarrow\alpha. &[\text{from 2 and 4 by (b--$iii$)}]
\end{array}
\]

Thus for any formula $\alpha$,
\[
\vdash_{\textsf{P}}\alpha\rightarrow\alpha. \tag{$\ast$}
\]
Now assume that $X,\alpha\vdash_{\textsf{P}}\beta$. That is, there is a \Pos-derivation from $X,\alpha$:
\[
\alpha_1,\ldots,\beta.\tag{$\ast\ast$}
\]
Let us form a finite sequence:
\[
\alpha\rightarrow\alpha_1,\ldots,\alpha\rightarrow\beta
\]
and denote:
$\gamma_i:=\alpha\rightarrow\alpha_i$, where $1\le i\le n$; so that $\gamma_n:=\alpha\rightarrow\beta$. By induction on the number $n$ of the formulas of $(\ast\ast)$, we prove that $\vdash_{\textsf{P}}\gamma_n$.

If $n=1$, then either $\beta=\alpha$ or $\beta\in X$. If the former is the case, then we apply ($\ast$). If $\beta\in X$, then we successively obtain:
$\vdash_{\textsf{P}}\beta\rightarrow(\alpha\rightarrow\beta)$ and $X\vdash_{\textsf{P}}\alpha\rightarrow\beta$.

Next, assume that $\beta$ is obtained from $\alpha_i$ and $\alpha_j=\alpha_i\rightarrow\beta$. By induction hypothesis,
$\vdash_{\textsf{P}}\alpha\rightarrow\alpha_i$ and 
$\vdash_{\textsf{P}}\alpha\rightarrow(\alpha_i\rightarrow\beta)$. Using ax2 and
(b--$iii$) twice, we receive that $\vdash_{\textsf{P}}\alpha\rightarrow\beta$.

We obtain the last part of the statement with the help of~\eqref{E:two-implications}.
\end{proof}

\begin{lem}\label{L:hyperrules-c_ii-c_iii}
The hyperrules {\em(c--$ii$)} and {\em(c--$iii$)} of Section~\ref{section:inference-rules} are sound w.r.t. {\em\Int} and {\em\Cl}.
Hence the rule {\em(a--$iv$)}is sound w.r.t. {\em\Int} and {\em\Cl}.
\end{lem}
\begin{proof}
The proof of the first part is straightforward and is left to the reader. (Exercise~\ref{section:lindenbaum-algebra}.\ref{EX:lemma-hyperrules-c_ii-c_iii})

The proof of the second part is as follows. Let $\mathcal{S}$ be either {\Int} or {\Cl}. Since $\alpha,\neg\alpha\vdash_{\mathcal{S}}\alpha$ and 
$\alpha,\neg\alpha\vdash_{\mathcal{S}}\neg\alpha$, by (c-$ii$), we conclude that $\alpha\vdash_{\mathcal{S}}\neg\neg\alpha$.
\end{proof}
\begin{cor}
The logic {\em\Cl} is an extension of {\em\Int}.
\end{cor}
\begin{proof}
We show that $\vdash_{\textsf{Cl}}\text{ax11}$, that is, for any $\Lan_A$-formulas $\alpha$ and $\beta$,
\begin{equation}\label{E:Int-axiom-valid-in-Cl}
\vdash_{\textsf{Cl}}\beta\rightarrow(\neg\beta\rightarrow\alpha).
\end{equation}

Indeed, since $\beta,\neg\beta,\neg\alpha\vdash_{\textsf{Cl}}\beta$ and
$\beta,\neg\beta,\neg\alpha\vdash_{\textsf{Cl}}\neg\beta$, by the hyperrrule (c--$ii$), we derive that $\beta,\neg\beta\vdash_{\textsf{Cl}}\neg\neg\alpha$.
Also, since, in virtue of Lemma~\ref{L:hyperrules-c_ii-c_iii}, $\neg\neg\alpha\vdash_{\textsf{Cl}}\alpha$, by cut, we obtain that
$\beta,\neg\vdash_{\textsf{Cl}}\alpha$. Then, we apply twice the hyperrule
(c--$i$).
\end{proof}

\begin{lem}\label{L:distributive-lettice-sound}
For any $\Lan_A$-formulas $\alpha$, $\beta$ and $\gamma$, the following formulas and their converses are theses of {\em\Pos}:
{\em
	\[
\begin{array}{cl}
(\text{a}) &(\alpha\wedge\beta)\rightarrow(\beta\wedge\alpha),\\
(\text{b}) &(\alpha\wedge(\beta\wedge\gamma))\rightarrow((\alpha\wedge\beta)\wedge\gamma),\\
(\text{c}) &((\alpha\wedge\beta)\wedge\gamma)\rightarrow(\alpha\wedge(\beta\wedge\gamma)),\\
(\text{d}) &((\alpha\wedge\beta)\vee\beta)\rightarrow\beta,\\
(\text{e}) &(\alpha\vee\beta)\rightarrow(\beta\vee\alpha),\\
(\text{f}) &(\alpha\vee(\beta\vee\gamma))\rightarrow((\alpha\vee\beta)\vee\gamma),\\
(\text{g}) &((\alpha\wedge\beta)\vee\beta)\rightarrow\beta,\\
(\text{h}) &\alpha\wedge(\alpha\vee\beta)\rightarrow\alpha,\\
(\text{j}) &(\alpha\wedge(\beta\vee\gamma))\rightarrow((\alpha\wedge\beta)\vee(\alpha\wedge\gamma)),\\
(\text{i}) &(\alpha\vee(\beta\wedge\gamma))\rightarrow((\alpha\vee\beta)\wedge(\alpha\vee\gamma)).
\end{array}
\]}
\end{lem}
\begin{proof}
We show about one of these formulas, namely (j), that it and its converse are theses of \Pos. First, we prove that 
\[
\vdash_{\textsf{P}}(\alpha\wedge(\beta\vee\gamma))\rightarrow((\alpha\wedge\beta)\vee(\alpha\wedge\gamma)). 
\]

We obtain:
\[
\begin{array}{rll}
1a. &\alpha\wedge(\beta\vee\gamma)\vdash_{\textsf{P}}\alpha
&[\text{Lemma~\ref{L:rules-sound-wrt-P-Int-Cl} and rule (b--$i$)}]\\
2a. &\alpha\wedge(\beta\vee\gamma)\vdash_{\textsf{P}}\beta\vee\gamma
&[\text{Lemma~\ref{L:rules-sound-wrt-P-Int-Cl} and rule (b--$ii$)}]\\
3a. &\alpha,\beta\vdash_{\textsf{P}}\alpha\wedge\beta 
&[\text{Lemma~\ref{L:rules-sound-wrt-P-Int-Cl} and rule (a--$i$)}]\\
4a. &\alpha\wedge(\beta\vee\gamma),\beta\vdash_{\textsf{P}}\alpha\wedge\beta
&[\text{from 1$a$ and 3$a$ by cut (Section~\ref{section:consequence-relation})}]\\
5a. &\alpha\wedge\beta\vdash_{\textsf{P}}(\alpha\wedge\beta)\vee(\alpha\wedge\gamma) &[\text{Lemma~\ref{L:rules-sound-wrt-P-Int-Cl} and rule (a--$ii$)}]\\
6a. &\alpha\wedge(\beta\vee\gamma),\beta\vdash_{\textsf{P}}(\alpha\wedge\beta)
\vee(\alpha\wedge\gamma) &[\text{from 4$a$ and 5$a$ by cut}]\\
7a. &\alpha,\gamma\vdash_{\textsf{P}}\alpha\wedge\gamma
&[\text{Lemma~\ref{L:rules-sound-wrt-P-Int-Cl} and rule (a--$i$)}]\\
8a. &\alpha\wedge(\beta\vee\gamma),\gamma\vdash_{\textsf{P}}\alpha\wedge\gamma
&[\text{from 1$a$ and 7$a$ by cut}]\\
9a. &\alpha\wedge\gamma\vdash_{\textsf{P}}(\alpha\wedge\beta)\vee(\alpha\wedge\gamma) &[\text{Lemma~\ref{L:rules-sound-wrt-P-Int-Cl} and rule (a--$iii$)}]\\
10a. &\alpha\wedge(\beta\vee\gamma),\gamma\vdash_{\textsf{P}}
(\alpha\wedge\beta)\vee(\alpha\wedge\gamma) &[\text{from 8$a$ and 9$a$ by cut}]\\
11a. &\alpha\wedge(\beta\vee\gamma),\beta\vee\gamma\vdash_{\textsf{P}}
(\alpha\wedge\beta)\vee(\alpha\wedge\gamma) 
&[\text{from 6$a$ and 10$a$ by hyperrule (c--$iii$)}]\\
12a. &\alpha\wedge(\beta\vee\gamma)\vdash_{\textsf{P}}
(\alpha\wedge\beta)\vee(\alpha\wedge\gamma) &[\text{from 2$a$ and 11$a$ by cut}]\\
13a. &\vdash_{\textsf{P}}(\alpha\wedge(\beta\vee\gamma))\rightarrow((\alpha\wedge\beta)\vee(\alpha\wedge\gamma)). &[\text{from 12$a$ by hyperrule (c--$i$)}]
\end{array}
\]

Further, we have:
\[
\begin{array}{rll}
1b. &\alpha\wedge\beta\vdash_{\textsf{P}}\alpha &[\text{Lemma~\ref{L:rules-sound-wrt-P-Int-Cl} and rule (b--$i$)}]\\
2b. &\alpha\wedge\beta\vdash_{\textsf{P}}\beta &[\text{Lemma~\ref{L:rules-sound-wrt-P-Int-Cl} and rule (b--$ii$)}]\\
3b. &\beta\vdash_{\textsf{P}}\beta\vee\gamma &[\text{Lemma~\ref{L:rules-sound-wrt-P-Int-Cl} and rule (a--$ii$)}]\\
4b. &\alpha\wedge\beta\vdash_{\textsf{P}}\beta\vee\gamma
&[\text{from 2$b$ and 3$b$ by cut}]\\
5b. &\alpha\wedge\gamma\vdash_{\textsf{P}}\alpha &[\text{Lemma~\ref{L:rules-sound-wrt-P-Int-Cl} and rule (b--$i$)}]\\
6b. &\alpha\wedge\gamma\vdash_{\textsf{P}}\gamma &[\text{Lemma~\ref{L:rules-sound-wrt-P-Int-Cl} and rule (b--$ii$)}]\\
7b. &\gamma\vdash_{\textsf{P}}\beta\vee\gamma &[\text{Lemma~\ref{L:rules-sound-wrt-P-Int-Cl} and rule (a--$iii$)}]\\
8b. &\alpha\wedge\gamma\vdash_{\textsf{P}}\beta\vee\gamma
&[\text{from 6$b$ and 7$b$ by cut}]\\
9b. &(\alpha\wedge\beta)\vee(\alpha\wedge\gamma)\vdash_{\textsf{P}}\beta\vee\gamma
&[\text{Lemma~\ref{L:hyperrules-c_ii-c_iii} and hyperrule (c--$iii$)}]\\
10b. &(\alpha\wedge\beta)\vee(\alpha\wedge\gamma)\vdash_{\textsf{P}}\alpha
&[\text{Lemma~\ref{L:hyperrules-c_ii-c_iii} and hyperrule (c--$iii$)}]\\
11b. &\alpha,\beta\vee\gamma\vdash_{\textsf{P}}\alpha\wedge(\beta\vee\gamma) &[\text{Lemma~\ref{L:rules-sound-wrt-P-Int-Cl} and rule (a--$i$)}]\\
12b. &(\alpha\wedge\beta)\vee(\alpha\wedge\gamma),\alpha\vdash_{\textsf{P}}
\alpha\wedge(\beta\vee\gamma) &[\text{from 9$b$ and 11$b$ by cut}]\\
13b. &(\alpha\wedge\beta)\vee(\alpha\wedge\gamma)\vdash_{\textsf{P}}
\alpha\wedge(\beta\vee\gamma) &[\text{from 10$b$ and 12$b$ by cut}]\\
14b. &\vdash_{\textsf{P}}((\alpha\wedge\beta)\vee(\alpha\wedge\gamma))\rightarrow
(\alpha\wedge(\beta\vee\gamma)). &\text{from 13$b$ by hyperrule (c--$i$)}
\end{array}
\]

We leave for the reader to complete the proof of this lemma. (Exercise~\ref{section:lindenbaum-algebra}.\ref{EX:distributive-lettice-sound})
\end{proof}

\begin{lem}\label{L:Cl-implications}
For any $\Lan_A$-formulas $\alpha$ and $\beta$, the following formulas and their converses are theses of {\em\Cl}:
{\em\[
	\begin{array}{cl}
(\text{k}) &((\alpha\wedge\neg\alpha)\vee\beta)\rightarrow\beta,\\
(\text{l}) &((\alpha\vee\neg\alpha)\wedge\beta)\rightarrow\beta.
	\end{array}
	\]}
\end{lem}
\begin{proof}
We prove that (k) and its converse are theses of {\Cl} and leave the rest to the reader. (Exercise~\ref{section:lindenbaum-algebra}.\ref{EX:Cl-implications})

First, we prove that
\[
\vdash_{\textsf{Cl}}((\alpha\wedge\neg\alpha)\vee\beta)\rightarrow\beta.
\]

Indeed, we have:
\[
\begin{array}{cll}
1a. &\alpha\wedge\neg\alpha\vdash_{\textsf{Cl}}\beta &[\text{Lemma~\ref{L:rules-sound-wrt-P-Int-Cl}, rules (b-$i$), (b--$ii$), \eqref{E:Int-axiom-valid-in-Cl} and cut}]\\
2a. &\beta\vdash_{\textsf{Cl}}\beta\\
3a. &(\alpha\wedge\neg\alpha)\vee\beta\vdash_{\textsf{Cl}}\beta
&[\text{Lemma~\ref{L:hyperrules-c_ii-c_iii} and hyperule (c--$iii$)}]\\
4a. &\vdash_{\textsf{Cl}}((\alpha\wedge\neg\alpha)\vee\beta)\rightarrow\beta.
&[\text{Lemma~\ref{L:hyperrules-c_ii-c_iii} and hyperule (c--$i$)}]
\end{array}
\]

For the converse, we have:
\[
\begin{array}{cll}
1b. &\vdash_{\textsf{Cl}}\beta\rightarrow((\alpha\wedge\neg\alpha)\vee\beta).
&[\text{rule ax7}]
\end{array}
\]
\end{proof}

\begin{prop}
{\em$\mathbb{T}_{\textsf{Cl}}$} coincides with the class of all Boolean algebras. 
\end{prop}
\begin{proof}
Let $\alg{B}=\left\langle \textsf{B};\wedge,\vee,\neg,\one\right\rangle$ be a Boolean algebra. Expanding this algebra with an operation
\[
x\rightarrow y:=\neg x\vee y, \tag{$\ast$}
\]
according to Proposition~\ref{P:boolean-algebra-as-heyting}, we obtain a Heyting algebra $\alg{H}=\left\langle \textsf{B};\wedge,\vee,\rightarrow,\neg,\one\right\rangle$. In virtue of
Proposition~\ref{P:Int-properties}, the rules ax1--ax9 are valid in $\alg{H}$ and hence in \alg{B}. In view of Corollary~\ref{C:Cl-property}, the rule ax10 is also valid in \alg{B}. And, because of the property $(\text{h}_2)$
(Section~\ref{section:heyting-algebra}), modus ponens preserves this validity over any \Cl-derivation.

Now, assume that an algebra $\alg{A}=\left\langle \textsf{A};\wedge,\vee,\neg,\one\right\rangle\in\mathbb{T}_{\textsf{Cl}}$. That is, for any $\Lan_A$-formula $\alpha\in\bm{T}_{\textsf{Cl}}$,
$\alg{A}\models\alpha\approx\one$. This means that for any $\alpha\in\bm{T}_{\textsf{Cl}}$ and any $\Lan_A$-valuation $v$ in \alg{A}, $v(\alpha)=\one$. Therefore, in virtue of~Lemma~\ref{L:distributive-lettice-sound} and the fact that $\Pos$ is implicational w.r.t $\rightarrow$, the properties $(\text{l}_1)$--$(\text{l}_4)$ are valid in \alg{A}. Namely, we have:
(a) and its converse prove the validity of ($\text{l}_1$--$i$); (b) and its converse prove the validity of ($\text{l}_1$--$ii$); (c) and its converse prove the validity of ($\text{l}_2$--$i$);
(f) and its converse prove the validity of ($\text{l}_2$--$ii$);
(g) and its converse prove the validity of ($\text{l}_3$--$i$);
(h) and its converse prove the validity of ($\text{l}_3$--$ii$);
(j) and its converse prove the validity of ($\text{l}_4$--$i$);
(i) and its converse prove the validity of ($\text{l}_4$--$ii$).
Thus $\langle \textsf{A};\wedge,\vee\rangle$ is a distributive lattice.

Further, if $\bigstar$ is an arbitrary {\Cl}-thesis and $p$ is an arbitrary variable, then
\begin{equation}\label{E:Cl-thesis}
\vdash_{\textsf{Cl}}p\rightarrow\bigstar.
\end{equation}
This implies that for any element $x\in\textsf{A}$, $x\le\one$, that is
the equalities $(\text{b}_{1})$ hold in $\langle \textsf{A};\wedge,\vee,\one\rangle$,
where $\one$ turns to be a greatest element in the sense of $\le$ (Section~\ref{section:boolean-algebra}).

Finally, we use Lemma~\ref{L:Cl-implications} and the fact that {\Cl} is and implicational abstract logic w.r.t. $\rightarrow$ to conclude that $\alg{A}$ is a Boolean algebra.
\end{proof}
\begin{cor}\label{C:LT_Cl(k)}
{\em$\LTCl$} is a free Boolean algebra of rank $\card{\Var_{\mathcal{L}_A}}$.
Hence, {\em$\LTCl(\kappa)$}, where $\kappa\le\card{\Var_{\mathcal{L}_A}}$,  is a free Boolean algebras of rank $\kappa$.
\end{cor}

According to Corollary~\ref{C:LT_Cl(k)} and Proposition~\ref{P:finitely-generated-boolean}, any $\LTCl(\kappa)$ of a finite rank $\kappa$ is finite. Let us take $\kappa=1$, assuming that $\Var=\lbrace p\rbrace$. For simplicity, we denote
\[
[p]:=p\slash\theta(\bm{T}_{\textsf{Cl}}).
\]

It is not difficult to show that the subalgebra of $\LTCl$ generated by $[p]$ consists of the four pairwise distinct elements: $[p]$, $[\neg p]$, $[p\vee\neg p]$ and $[p\wedge\neg p]$. (Exercise~\ref{section:some-lindenbaum-algebras}.\ref{EX:four-elements})This subalgebra is $\LTCl(1)$. The relation $\le$ between $[\alpha]$ and $[\beta]$ of this algebra is defined, according to \eqref{E:ordering-in-lattice}, as follows:
\[
[\alpha]\le[\beta]\stackrel{\text{df}}{\Longleftrightarrow}[\alpha]\wedge[\beta]=[\alpha].
\]

Thus, we arrive at the following diagram:

\begin{figure}[!ht]	
	\[
	\ctdiagram{
		\ctnohead
		\ctinnermid
		\ctel 0,0,20,20:{}
		\ctel 0,0,-20,20:{}
		\ctel 0,40,20,20:{}
		\ctel 0,40,-20,20:{}
		\ctv 0,0:{\bullet}
		\ctv 20,20:{\color{red}\bullet}
		\ctv 28,20:{[p]}
		\ctv -20,20:{\bullet}
		\ctv -32,20:{[\neg p]}
		\ctv 0,40:{\bullet}
		\ctv 0,47:{\mathbf{1}=[p\vee\neg p]}
		\ctv 0,-8:{\mathbf{0}=[p\wedge\neg p]}
	}
	\]
	\caption{Lindenbaum-Tarski algebra $\LTCl(1)$}
\end{figure}

If $\Var=\lbrace p,q\rbrace$, we obtain $\LTCl(2)$ as a subalgebra of $\LTCl$ generated by the classes $[p]$ and $[q]$. If we identify any congruence class of this subalgebra with its representative, we can depict 	$\LTCl(2)$ by the following diagram:
\pagebreak
\begin{figure}[!ht]
	
	\[
	\ctdiagram{
		\ctnohead
		\ctinnermid
		\ctel 0,0,30,30:{}
		\ctel 0,0,-30,30:{}
		\ctel 0,60,30,30:{}
		\ctel 0,60,-30,30:{}
		\ctel 0,30,30,60:{}
		\ctel 0,30,-30,60:{}
		\ctel 0,90,30,60:{}
		\ctel 0,90,-30,60:{}
		\ctel 0,0,0,30:{}
		\ctel -30,30,-30,60:{}
		\ctel 30,30,30,60:{}
		\ctel 0,60,0,90:{}
		\ctel 120,30,90,60:{}
		\ctel 120,30,150,60:{}
		\ctel 90,60,120,90:{}
		\ctel 150,60,120,90:{}
		\ctel 120,60,150,90:{}
		\ctel 120,60,90,90:{}
		\ctel 150,90,120,120:{}
		\ctel 90,90,120,120:{}
		\ctel 120,30,120,60:{}
		\ctel 90,60,90,90:{}
		\ctel 150,60,150,90:{}
		\ctel 120,90,120,120:{}
		\ctel 0,0,120,30:{}
		\ctel -30,30,90,60:{}	
		\ctel 30,30,150,60:{}	
		\ctel 0,60,120,90:{}	
		\ctel 0,30,120,60:{}		
		\ctel 30,60,150,90:{}
		\ctel -30,60,90,90:{}
		\ctel 0,90,120,120:{}
		\ctv 0,0:{\bullet}
		\ctv 30,30:{\bullet}
		\ctv -30,30:{\bullet}
		\ctv 0,60:{\color{red}\bullet}
		\ctv 0,53:{p}
		\ctv 0,30:{\bullet}
		\ctv -45,30:{p\land q}
		\ctv 30,60:{\bullet}
		\ctv -30,60:{\bullet}
		\ctv -80,60:{(p \land q) \lor (\neg p \land \neg q)}
		\ctv 0,90:{\bullet}
		\ctv -22,92:{\neg p \lor \neg q}
		\ctv 120,30:{\bullet}
		\ctv 136,31:{\neg p \land q}
		\ctv 150,60:{\bullet}
		\ctv 200,60:{ (p \land \neg q) \lor (\neg p \land q)}
		\ctv 90,60:{\color{red}\bullet}
		\ctv 120,90:{\bullet}
		\ctv 120,60:{\bullet}
		\ctv 128,60:{\neg p}
		\ctv 150,90:{\bullet}
		\ctv 170,90:{\neg p \lor \neg q}
		\ctv 90,90:{\bullet}
		\ctv 120,120:{\bullet}
		\ctv 46,27:{p \land \neg q}
		\ctv -4,23:{\neg p \land \neg q}
		\ctv 72,95:{\neg p \lor q}
		\ctv 122,96:{p \lor q}
		\ctv 38,58:{\neg q}
		\ctv 85,64:{q}
		\ctv 120,127:{\mathbf{1}}
		\ctv 0,-10:{\mathbf{0}}		
	}
	\]
	\caption{Lindenbaum-Tarski algebra $\LTCl(2)$}
\end{figure}	

Now we turn to Heyting algebras.
\begin{lem}\label{L:Int-implications}
Let $\bigstar$ be an arbitrary {\em\Int}-thesis. For any $\Lan_A$-formulas $\alpha$, $\beta$ and $\gamma$, the following implications and their converses are {\em\Int}-theses:
{\em\[
\begin{array}{cl}
(\text{a}) &(\alpha\wedge(\alpha\rightarrow\beta))\rightarrow(\alpha\wedge\beta),\\
(\text{b}) &((\alpha\rightarrow\beta)\wedge\beta)\rightarrow\beta,\\
(\text{c}) &((\alpha\rightarrow\beta)\wedge(\alpha\rightarrow\gamma)
\rightarrow(\alpha\rightarrow(\beta\wedge\gamma))\\
(\text{d}) &(\alpha\wedge(\beta\rightarrow\beta))\rightarrow\alpha,\\
(\text{e}) &(\neg\bigstar\vee\alpha)\rightarrow\alpha,\\
(\text{f}) &\neg\alpha\rightarrow(\alpha\rightarrow\neg\bigstar).
\end{array}
\]}
\end{lem}
\begin{proof}
We prove that $\vdash_{\textsf{Int}}(\neg\bigstar\vee\alpha)\rightarrow\alpha$
and $\vdash_{\textsf{Int}}\alpha\rightarrow(\neg\bigstar\vee\alpha)$ and leave for the reader to complete this proof. (Exercise~\ref{section:lindenbaum-algebra}.\ref{EX:Int-implications})

Indeed, we obtain:
\[
\begin{array}{cll}
1. &\vdash_{\textsf{Int}}\bigstar &[\text{by premise $\bigstar$ is an {\Int}-thesis}]\\
2. &\vdash_{\textsf{Int}}\bigstar\rightarrow(\neg\bigstar\rightarrow\alpha)
&[\text{rule ax11}]\\
3. &\vdash_{\textsf{Int}}\neg\bigstar\rightarrow\alpha
&[\text{from $1$ and $2$ by modus ponens]}\\
4. &\vdash_{\textsf{Int}}\alpha\rightarrow\alpha
&[\text{because of $(\ast)$ in the proof of Lemma~\ref{L:deduction-theorem}}]\\
5. &\vdash_{\textsf{Int}}(\neg\bigstar\vee\alpha)\rightarrow\alpha.
&[\text{from ax8, $3$ and $4$ by modus ponens}] 
\end{array}
\]

Further, $\alpha\rightarrow(\neg\bigstar\vee\alpha)$ is an instantiation of ax7.
\end{proof}

\begin{prop}
{\em$\mathbb{T}_{\textsf{Int}}$} coincides with the class of all Heyting algebras.
\end{prop}
\begin{proof}
Let {\em$\alg{H}=\langle\textsf{H};\wedge,\vee,\rightarrow,\neg,\one\rangle$}
be a Heyting algebra. In virtue of
Proposition~\ref{P:Int-properties}, the rules ax1--ax9 are valid in $\alg{H}$.
In virtue of the properties \eqref{E:zero-in-heyting}, \eqref{E:pseudo-complementation} and~\eqref{E:less-than=implication}, we receive that ax11 is also valid in \alg{H}. Noticing, that modus ponens preserves validity in any Heyting algebra, we conclude that if $\alpha\in\bm{T}_{\textsf{Int}}$, then $\alg{H}\models\alpha\approx\one$.

Next, assume that an algebra $\alg{H}=\langle\textsf{H};\wedge,\vee,\rightarrow,\neg,\one\rangle\in\mathbb{T}_{\textsf{Int}}$. That is, for any $\alpha\in\bm{T}_{\textsf{Int}}$ and any
valuation $v$ in \alg{H}, $v(\alpha)=\one$. We show that the properties
$(\text{l}_1)$--$(\text{l}_4)$, $(\text{b}_1)$, and $(\text{h}_1)$--$(\text{h}_6)$ are valid in \alg{H}.

In virtue of Lemma~\ref{L:distributive-lettice-sound} and the fact that $\Pos$ is implicational w.r.t $\rightarrow$, the properties $(\text{l}_1)$--$(\text{l}_4)$ are valid in \alg{H}. Thus $\langle\textsf{H};\wedge,\vee\rangle$ is a distributive lattice.

Further, similarly to~\eqref{E:Cl-thesis}, we derive that 
	\[
	\vdash_{\textsf{Int}}p\rightarrow\bigstar,
	\]
	for any variable $p$ and \textsf{Int}-thesis $\bigstar$. This implies that the equalities $(\text{b}_1)$ are true in \alg{H}.
	
	To prove the properties $(\text{h}_1)$--$(\text{h}_6)$, we apply Lemma~\ref{L:Int-implications}.
\end{proof}
 \begin{cor}
 	{\em$\LTInt$} is a free Heyting algebra of rank $\card{\Var_{\mathcal{L}_A}}$.
 	Hence, {\em$\LTInt(\kappa)$}, where $\kappa\le\card{\Var_{\mathcal{L}_A}}$,  is a free Heyting algebras of rank $\kappa$.
 \end{cor}

Below we give a description of $\LTInt(1)$. For this, we consider the following formulas of an arbitrary variable $p$.

First, we define:
\[
P_0:=p\land\neg p, \qquad P_1:=\neg p, \qquad P_2:=p, \qquad P_{\infty}:=p\to p.
\]
Then for any $n\ge 0$, we define:
\[
P_{2n+3}:=P_{2n+1}\to P_{2n}~~\text{and}~~P_{2n+4}:= P_{2n+1}\lor P_{2n+2}.
\]

The elements of $\LTInt(1)$ are the congruence classes w.r.t. $\theta(\bm{T}_{\textsf{Int}})$. It will be convenient, in writing, to depict each congruence class by its representative; in other words we will treat $P_n$, on the one hand, as an $\Lan_A$-formula and, on the other, as the congruence classes it generates. In accordance with this agreement, we have:
\[
\one:= P_{\infty}~~\text{and}~~\zero:=P_0.
\]
Also, we will be writing $P_i\le P_j$ in the sense that $P_{i}\to P_{j}=\one$, that is $\vdash_{\textsf{Int}}P_{i}\to P_{j}$; compare with~\eqref{E:pseudo-complementation}.

We conclude with the following observation, due to Iwao Nishimura~\cite{nishimura1960}; see also~L. Rieger~\cite{rieger1957}.

\begin{prop}
Every $\Lan_A$-formula of one variable $p$ is equivalent in {\em\Int} to one and only one formula $P_i$, which by themselves are not equivalent to each other and related to each other as shown in the following diagram.
\end{prop}


\begin{figure}[ht]
	\[	
	\ctdiagram{
		\ctnohead
		\ctinnermid
		\ctel 0,0,40,40:{}
		\ctel 0,0,-20,20:{}
		\ctel -20,20,40,80:{}
		\ctel 20,20,-20,60:{}
		\ctel -20,100,40,40:{}
		\ctel 40,120,-20,60:{}
		\ctel -10,130,40,80:{}
		\ctel 30,150,-20,100:{}
		\ctel 10,150,40,120:{}
		\ctv 0,0:{\bullet}
		\ctv 20,20:{\color{red}\bullet}
		\ctv 54,20:{P_2 = p}		
		\ctv -20,20:{\bullet}
		\ctv -60,20:{P_1 = \neg p}		
		\ctv 0,40:{\bullet}
		\ctv -50,40:{P_4 = p \lor \neg p}
		\ctv 40,40:{\bullet}
		\ctv 84,40:{P_3 = \neg\neg p}
		\ctv 20,60:{\bullet}
		\ctv 74,60:{P_6=\neg p \lor \neg\neg p}
		\ctv -20,60:{\bullet}
		\ctv -70,60:{P_5=\neg\neg p \to p}
		\ctv 0,80:{\bullet}
		\ctv -50,80:{P_8 = P_5 \lor P_3}		
		\ctv 40,80:{\bullet}
		\ctv 92,80:{P_7 = P_5 \to P_4}		
		\ctv 20,100:{\bullet}
		\ctv 74,100:{P_{10}=P_5 \lor P_7}		
		\ctv 0,120:{\bullet}
		\ctv -48,120:{P_{12} = P_7 \lor P_9}
		\ctv -20,100:{\bullet}
		\ctv -70,100:{P_9 = P_7 \to P_6}
		\ctv 40,120:{\bullet}
		\ctv 90,120:{P_{11} = P_9 \to P_8}
		\ctv 20,140:{\bullet}
		\ctv 15,160:{\dots\dots\dots\dots\dots\dots}
		\ctv 15,165:{\dots\dots\dots\dots\dots\dots}
		\ctv 15,170:{\dots\dots\dots\dots\dots\dots}
		\ctv 10,180:{\bullet}
		\ctv 10,189:{\mathbf{1}=p\to p}
		\ctv 0,-8:{\mathbf{0}=p\land\neg p}		
	}
	\]
	\caption{Lindenbaum-Tarski algebra $\LTInt(1)$}
\end{figure}

\paragraph{Exercises~\ref{section:lindenbaum-algebra}}
\begin{enumerate}
		\item \label{EX:Cl-Int-implicational} Show that the abstract logics $\Pos$, $\Cl$ and $\Int$ are implicational.
	\item \label{EX:T_S-star-thesis}Prove~\eqref{E:T_S-star-thesis}.
	\item \label{EX:v-restored}Prove that for any valuation $w$ in $\langle\alg{A},\one\rangle$, there is a unique valuation $v$ in $\alg{A}$ such that $v^{\star}=w$.
	\item \label{EX:LT-term-valuation}Prove~\eqref{E:LT-term-valuation}.
	\item \label{EX:LT=free-algebra} Prove Lemma~\ref{L:LT=free-algebra}.
	\item \label{EX:lemma-rules-sound-wrt-P-Int-Cl}Complete the proof of Lemma~\ref{L:rules-sound-wrt-P-Int-Cl}.
	\item \label{EX:lemma-hyperrules-c_ii-c_iii}Show that the hyperrules (c--$ii$) and (c--$iii$) are sound w.r.t. {\Cl} and {\Int}.
	\item \label{EX:distributive-lettice-sound} Complete the proof of Lemma~\ref{L:distributive-lettice-sound}.
	\item\label{EX:Cl-implications} Complete the proof of Lemma~\ref{L:Cl-implications}.
	\item \label{EX:four-elements} Show that the elements $[p]$, $[\neg p]$, $[p\vee\neg p]$ and $[p\wedge\neg p]$ are pairwise distinct in $\LTCl$.
	(Hint: prove that the matrix $\booleTwo$ is a $\Cl$-model.)
	\item\label{EX:Int-implications} Complete the proof of Lemma~\ref{L:Int-implications}.
\end{enumerate}

\bibliographystyle{alpha}

\begin{thebibliography}{Dum59}
	
	\bibitem[Bar77]{barwise1977}
	{\em Handbook of mathematical logic}.
	\newblock North-Holland Publishing Co., Amsterdam-New York-Oxford, 1977.
	\newblock Edited by Jon Barwise, With the cooperation of H. J. Keisler, K.
	Kunen, Y. N. Moschovakis and A. S. Troelstra, Studies in Logic and the
	Foundations of Mathematics, Vol. 90.
	
	\bibitem[Ber26]{bernays1926}
	Paul Bernays.
	\newblock Untersuchung des {A}ussagenkalk\"{u}ls der ``{P}rincipia
	{M}athematica''.
	\newblock {\em Math. Z.}, 25:305--320, 1926.
	
	\bibitem[Boo03]{boole1854}
	George Boole.
	\newblock {\em The laws of thought}.
	\newblock Great Books in Philosophy. Prometheus Books, Amherst, NY, 2003.
	\newblock Reprint of the 1854 original, With an introduction by John Corcoran.
	
	\bibitem[Bou98]{bourbaki1998}
	Nicolas Bourbaki.
	\newblock {\em General topology. {C}hapters 1--4}.
	\newblock Elements of Mathematics (Berlin). Springer-Verlag, Berlin, 1998.
	\newblock Translated from the French, Reprint of the 1989 English translation.
	
	\bibitem[CK90]{chang-keisler1990}
	C.~C. Chang and H.~J. Keisler.
	\newblock {\em Model theory}, volume~73 of {\em Studies in Logic and the
		Foundations of Mathematics}.
	\newblock North-Holland Publishing Co., Amsterdam, third edition, 1990.
	
	\bibitem[Cur66]{curry1966}
	Haskell~B. Curry.
	\newblock {\em A theory of formal deducibility}.
	\newblock Second edition, reprinting. Notre Dame Mathematical Lectures, No. 6.
	University of Notre Dame Press, South Bend, Ind., 1966.
	
	\bibitem[Dum59]{dummett1959}
	Michael Dummett.
	\newblock A propositional calculus with denumerable matrix.
	\newblock {\em J. Symb. Logic}, 24:97--106, 1959.
	
	\bibitem[Fon16]{font2016}
	Josep~Maria Font.
	\newblock {\em Abstract algebraic logic}, volume~60 of {\em Studies in Logic
		(London)}.
	\newblock College Publications, London, 2016.
	\newblock An introductory textbook, Mathematical Logic and Foundations.
	
	\bibitem[G{\"o}d32]{godel1932}
	Kurt G{\"o}del.
	\newblock Zum intuitionistischen {A}ussagenkalk\"{u}l ({G}erman).
	\newblock {\em Anzeiger der {A}kademie der {W}issenschaften in {W}ien},
	69:65--66, 1932.
	\newblock Both German original and English translation ``On the intuitionistic
	propositional calculus'' in:~\cite{godel1986}, pp. 223--225.
	
	\bibitem[G{\"o}d86]{godel1986}
	Kurt G{\"o}del.
	\newblock {\em Collected works. {V}ol. {I}}.
	\newblock The Clarendon Press, Oxford University Press, New York, 1986.
	\newblock Publications 1929--1936, Edited and with a preface by Solomon
	Feferman.
	
	\bibitem[Gr{\"a}08]{gratzer2008}
	George Gr{\"a}tzer.
	\newblock {\em Universal algebra}.
	\newblock Springer, New York, second edition, 2008.
	\newblock With appendices by Gr\"atzer, Bjarni J\'onsson, Walter Taylor, Robert
	W. Quackenbush, G\"unter H. Wenzel, and Gr\"atzer and W. A. Lampe.
	
	\bibitem[Hai86]{hailperin1986}
	Theodore Hailperin.
	\newblock {\em Boole's logic and probability}, volume~85 of {\em Studies in
		Logic and the Foundations of Mathematics}.
	\newblock North-Holland Publishing Co., Amsterdam, second edition, 1986.
	\newblock A critical exposition from the standpoint of contemporary algebra,
	logic and probability theory.
	
	\bibitem[Kel75]{kelley1975}
	John~L. Kelley.
	\newblock {\em General topology}.
	\newblock Springer-Verlag, New York-Berlin, 1975.
	\newblock Reprint of the 1955 edition [Van Nostrand, Toronto, Ont.], Graduate
	Texts in Mathematics, No. 27.
	
	\bibitem[KK62]{kneales1962}
	William Kneale and Martha Kneale.
	\newblock {\em The development of logic}.
	\newblock Clarendon Press, Oxford, 1962.
	
	\bibitem[Kle52]{kleene1952}
	Stephen~Cole Kleene.
	\newblock {\em Introduction to metamathematics}.
	\newblock D. Van Nostrand Co., Inc., New York, N. Y., 1952.
	
	\bibitem[Kuz79]{kuznetsov79}
	Alexander~V. Kuznetsov.
	\newblock Means for detection of nondeducibility and inexpressibility.
	\newblock In {\em Logical inference ({M}oscow, 1974)}, pages 5--33. ``Nauka'',
	Moscow, 1979.
	
	\bibitem[{\L}uk20]{lukasiewicz1920}
	Jan {\L}ukasiewicz.
	\newblock O logice tr{\'o}jwarto{\'s}ciowej.
	\newblock {\em Ruch {F}ilozoficzny}, 5:170--171, 1920.
	\newblock English translation: On three-valued logic, in:
	\cite{lukasiewicz1970}, pp. 87--88.
	
	\bibitem[{\L}uk30]{lukasiewicz1930}
	Jan {\L}ukasiewicz.
	\newblock Philosophische {B}emerkungen zu mehrwertigen {S}ystemen des
	{A}ussagenkalk{\"u}ls ({G}erman).
	\newblock {\em Comptes rendus des s{\'e}ances de la {S}oci{\'e}t{\'e} des
		sciences et des letters de {V}arsovie}, Classe III, vol. xxiii:51--77, 1930.
	\newblock English translation: Philosophical remarks on many-valued systems of
	propositional logic, in: \cite{mccall1967}, pp. 40--65.
	
	\bibitem[{\L}uk57]{luk1957}
	Jan {\L}ukasiewicz.
	\newblock {\em Aristotle's syllogistic from the standpoint of modern formal
		logic}.
	\newblock The Clarendon Press, Oxford, 2 edition, 1957.
	
	\bibitem[{\L}uk64]{lukasiewicz1929}
	Jan {\L}ukasiewicz.
	\newblock {\em Elements of mathematical logic}.
	\newblock Translated from Polish by Olgierd Wojtasiewicz. International Series
	of Monographs on Pure and Applied Mathematics. Vol. 31. A Pergamon Press
	Book; The Macmillan Co., New York, 1964.
	\newblock English translation of the second edition of the 1929 authorized
	lecture notes.
	
	\bibitem[{\L}uk70]{lukasiewicz1970}
	Jan {\L}ukasiewicz.
	\newblock {\em Selected works}.
	\newblock North-Holland Publishing Co., Amsterdam-London; PWN-Polish Scientific
	Publishers, Warsaw, 1970.
	\newblock Edited by L. Borkowski, Studies in Logic and the Foundations of
	Mathematics.
	
	\bibitem[Mal73]{malcev1973}
	Anatoly~I. Mal{`c}ev.
	\newblock {\em Algebraic systems}.
	\newblock Akademie-Verlag, Berlin, 1973.
	\newblock Posthumous edition, edited by D. Smirnov and M. Ta\u{\i}clin,
	Translated from the Russian by B. D. Seckler and A. P. Doohovskoy.
	
	\bibitem[McK39]{mckinsey1939}
	J.~C.~C. McKinsey.
	\newblock Proof of the independence of the primitive symbols of {H}eyting's
	calculus of propositions.
	\newblock {\em J. Symbolic Logic}, 4:155--158, 1939.
	
	\bibitem[Me67]{mccall1967}
	Storrs McCall~(editor).
	\newblock {\em Polish logic: 1920--1939}.
	\newblock With an introduction by {T}adeusz {K}otarbi\'{n}ski. Clarendon Press,
	Oxford, 1967.
	
	\bibitem[Men15]{mendelson2015}
	Elliott Mendelson.
	\newblock {\em Introduction to mathematical logic}.
	\newblock Textbooks in Mathematics. CRC Press, Boca Raton, FL, sixth edition,
	2015.
	
	\bibitem[MT48]{mckinsey-tarski1948}
	J.~C.~C. McKinsey and Alfred Tarski.
	\newblock Some theorems about the sentential calculi of {L}ewis and {H}eyting.
	\newblock {\em J. Symbolic Logic}, 13:1--15, 1948.
	
	\bibitem[Nis62]{nishimura1960}
	Iwao Nishimura.
	\newblock On formulas of one variable in intuitionistic propositional calculus.
	\newblock {\em J. Symbolic Logic}, 25:327--331, 1962.
	
	\bibitem[Rie57]{rieger1957}
	Ladislav Rieger.
	\newblock A remark on the so-called free closure algebras.
	\newblock {\em Czechoslovak Math. J.}, 7(82):16--20, 1957.
	
	\bibitem[RS70]{rasiowa-sikorski1970}
	Helena Rasiowa and Roman Sikorski.
	\newblock {\em The mathematics of metamathematics}.
	\newblock PWN---Polish Scientific Publishers, Warsaw, third edition, 1970.
	\newblock Monografie Matematyczne, Tom 41.
	
	\bibitem[Wit01]{wittgenstein2001}
	Ludwig Wittgenstein.
	\newblock {\em Tractatus Logico-Philosophicus}.
	\newblock Routledge, second edition, 2001.
	\newblock Translated by D. Pears and B. McGuinness, Prefaced by Bertrand
	Russell.
	
	\bibitem[WR97]{whitehead-russell}
	Alfred~North Whitehead and Bertrand Russell.
	\newblock {\em {\it {P}rincipia mathematica} to *56}.
	\newblock Cambridge Mathematical Library. Cambridge University Press,
	Cambridge, 1997.
	\newblock Reprint of the second (1927) edition.
	
\end{thebibliography}

\end{document}